\documentclass{amsart}
\usepackage[utf8]{inputenc}
\usepackage{geometry}
\usepackage{enumerate}
\usepackage{amsmath,amssymb}
\usepackage{cite}
\usepackage{amsfonts}
\usepackage{mathrsfs}
\usepackage{galois}
\usepackage[toc,page]{appendix}
\usepackage{xcolor}
\usepackage[colorlinks, linkcolor = blue, anchorcolor = blue, citecolor = blue, backref=page]{hyperref}

\newtheorem{Def}{Definition}[section]
\newtheorem{Thm}[Def]{Theorem}
\newtheorem{eg}[Def]{Example}
\newtheorem{Lem}[Def]{Lemma}
\newtheorem{Prop}[Def]{Proposition}
\newtheorem{Cor}[Def]{Corollary}
\newtheorem{Rem}[Def]{Remark}
\newtheorem{Con}[Def]{Conjecture}

\numberwithin{equation}{section}

\newcommand{\BB}{\mathbb{B}}

\newcommand{\RR}{\mathbb{R}}
\newcommand{\SSp}{\mathbb{S}}

\newcommand{\ZZ}{\mathbb{Z}}

\newcommand{\cA}{\mathcal{A}}
\newcommand{\cB}{\mathcal{B}}
\newcommand{\cC}{\mathcal{C}}

\newcommand{\cE}{\mathcal{E}}
\newcommand{\cF}{\mathcal{F}}

\newcommand{\cH}{\mathcal{H}}
\newcommand{\cI}{\mathcal{I}}

\newcommand{\cL}{\mathcal{L}}
\newcommand{\cM}{\mathcal{M}}

\newcommand{\cR}{\mathcal{R}}
\newcommand{\cS}{\mathcal{S}}
\newcommand{\cT}{\mathcal{T}}

\newcommand{\cV}{\mathcal{V}}

\newcommand{\scH}{\mathscr{H}}

\newcommand{\scL}{\mathscr{L}}

\newcommand{\scR}{\mathscr{R}}

\newcommand{\scT}{\mathscr{T}}
\newcommand{\scU}{\mathscr{U}}
\newcommand{\scV}{\mathscr{V}}
\newcommand{\scW}{\mathscr{W}}

\newcommand{\mbfF}{\mathbf{F}}
\newcommand{\mbfG}{\mathbf{G}}

\newcommand{\mbfI}{\mathbf{I}}

\newcommand{\mbfP}{\mathbf{P}}
\newcommand{\mbfQ}{\mathbf{Q}}

\newcommand{\I}{\mathcal{I}}

\newcommand{\Int}{\text{Int}}
\newcommand{\Clos}{\text{Clos}}

\newcommand{\spt}{\text{spt}}

\newcommand{\graph}{\mathrm{graph}}
\newcommand{\Euc}{\text{Euc}}
\newcommand{\orig}{\mathbf{0}}
\newcommand{\vol}{\mathrm{vol}}

\newcommand{\ES}{\mathrm{ES}}
\newcommand{\ind}{\mathrm{ind}}

\title{On mean curvature flow translators with prescribed ends}
\author{Ao Sun$^{1}$}
\address{$^1$Lehigh University, Department of Mathematics, Chandler-Ullmann Hall, Bethlehem, PA 18015; email: aos223@lehigh.edu}
\author{Zhihan Wang$^2$}
\address{$^2$Department of Mathematics, Princeton University, Fine Hall, 304 Washington Road, Princeton, NJ 08540
	\\
	Current address: Cornell University, Department of Mathematics, 310 Malott Hall, Ithaca, NY 14853; email: zhihanw@math.princeton.edu}
\usepackage[foot]{amsaddr}
\date{\today}

\begin{document}
	
	\begin{abstract}
		
		Given a smooth closed embedded self-shrinker $S$ with index $I$ in $\RR^{n}$, we construct an $I$-dimensional family of complete translators polynomially asymptotic to $S\times\RR$ at infinity, which answers a long-standing question by Ilmanen. We further prove that $\RR^{n+1}$ can be decomposed in many ways into a one-parameter family of closed sets $\coprod_{a\in \RR} T_a$, and each closed set $T_a$ contains a complete translator asymptotic to $S\times\RR$ at infinity. If the closed set $T_a$ fattens, namely it has nonempty interior, then there are at least two translators asymptotic to each other at an exponential rate, which can be viewed as a kind of nonuniqueness. We show that this fattening phenomenon is non-generic but indeed happens.
	\end{abstract}
	\maketitle
	\tableofcontents

	\section{Introduction} \label{Sec_Intro}
	
	In this paper, we study translators of mean curvature flow. Let $\RR^{n+1}=\RR^n\times\RR$ be parametrized by $(x,z)$, and $z$ is the last coordinate function, with $\partial_z$ being the coordinate vector. A \textbf{translator} is a hypersurface $\Sigma\subset \RR^{n+1}$ such that \[
	\vec{H}_\Sigma=\vec{v}^\perp,  \]
	for some unit vector $\vec{v}\in \RR^{n+1}$, where $\vec{H}_\Sigma$ is the mean curvature vector of $\Sigma$ and $\perp$ denotes the projection onto the normal bundle of $\Sigma$. By composing with a rotation in $\RR^{n+1}$, we may assume without loss of generality that $\vec{v} = \partial_z$ throughout this article.
	
	Every translator $\Sigma$ is naturally associated with an eternal mean curvature flow $\{\Sigma+t\partial_z\}_{t\in\RR}$ in $\RR^{n+1}$.  In this way,  translators are natural type II singularity models of mean curvature flow, see \cite{AV97_DegenerateNeckpinches, Hamilton95_HarnackMCF}.  Various examples of mean curvature flow translators have been constructed in the literature, especially in $\RR^3$, see \cite{Mullins56_MCF, AltschulerWu94_Transl, ClutterbuckSchnurerSchulze07_TranslCatenoid,  WangXJ11_ConvexMCF, Nguyen09_TranslTrident, Nguyen13_Transl, Nguyen15_DoublyPeriodTransl, DavilaDelPinoNguyen17_FiniteTopTransl, Smith15_TranslPrescibGenus, BourniLangfordTinaglia20_Transl, HIMW19_TranslGraph, HMW19_ScherkTransl, HMW22_SemiGraphTransl}.  
	The analogs of translators also play essential roles in the study of other geometric flows. For example, there are various studies of Ricci flow that focus on the steady Ricci solitons, which are known to be the analog of mean curvature flow translators. Such solutions are important to the study of Ricci flow singularities, see the pioneering work of Hamilton \cite{Hamilton88_Cigar}.
	
	Start with a translator $\Sigma\subset \RR^{n+1}$ of \textbf{finite entropy} (see the definition in Section \ref{Subsec_Prelim MCF and Entropy}), an argument of Ilmanen \cite{Ilmanen94_EllipReg} shows that any tangent flow at $-\infty$ of the ancient mean curvature flow $\{\Sigma + t\partial_z\}_{t<0}$ is generated by a \textbf{self-shrinker} splitting in $z$-direction (see Corollary \ref{Cor_Pre_Blow down transl MCF split} for a precise statement). 
	Here, a \textbf{self-shrinker} is a (weak) submanifold in $\RR^{n+1}$ satisfying $\vec{H}=-X^\perp/2$. Self-shrinkers are known to be the type I singularity model for mean curvature flows and have been widely studied in recent years. Besides the simplest example of generalized round cylinders $\SSp_{\sqrt{2k}}^k\times \RR^{n-k}\subset \RR^{n+1}$, more examples of self-shrinkers have been constructed, see \cite{Angenent92_Doughnut, Moller11_ClosedShrinker, Nguyen14_Shrinker, 
		KapouleasKleeneMoller18_DesingShrinker, Ketover16_Shrinker, DruganLeeNguyen18_Survey_SymShrinker, SunWangZhou20_MinmaxShrinker, KapouleasMcGrath20_DoublingShrinker, BuzanoNguyenSchulz21_ShrinkerGenus, Riedler22_ClosedShrinker}. We include a brief discussion of self-shrinkers in Section \ref{Subsec_Prelim MCF and Entropy}. 
	
	From this asymptotic point of view, all the examples above of translators turn out to be either of infinite entropy, or have a tangent flow at infinity to be the generalized round cylinders $\SSp^k_{\sqrt{2k}}\times \RR^{n-k}$, possibly with multiplicities.  This brings to us the natural question of whether other self-shrinkers can be the asymptotic of some translator in $\RR^{n+1}$.
	
	In this paper, we construct new families of translators with prescribed ends. Throughout this paper, we will be focused on hypersurfaces, namely submanifolds with codimension $1$.

	\begin{Thm}\label{Thm_Intro_main}
		Suppose $S\subset \RR^n$ is a closed embedded smooth self-shrinker with \textbf{index $I$}. Then there exists an $I$-parameter family of complete embedded translators in $\RR^{n+1}$, possibly with singularities of dimension $\leq n-7$, each of which has the unique tangent flow at $-\infty$ to be $S\times\RR$ with multiplicity one.
	\end{Thm}
	
	Our main theorem answers a question of Ilmanen. As the parabolic blow-down limit of a translator gives a corresponding self-shrinker, in \cite[Appendix J]{Ilmanen94_EllipReg}, Ilmanen asked for what class of translators and self-shrinkers, this correspondence is one-to-one. Theorem \ref{Thm_Intro_main} implies that for any closed self-shrinker other than the sphere, this correspondence is not one-to-one, but multi-family-to-one. 
	
	Recall that self-shrinkers are critical points of the Gaussian area functional $\int_S e^{-|x|^2/4}$,  and the \textbf{index} of a self-shrinker $S$ is defined as the dimension of the negative eigenspace of the Jacobi operator $-L_S:= -(\Delta_S - \nabla_S + |A_S|^2 + 1/2)$ of the Gaussian area functional on $S$.	
	It was proved by Colding-Minicozzi \cite{ColdingMinicozzi12_generic} that any closed self-shrinker in $\RR^{n}$ other than the round sphere has index $I\geq n+2$. Because the space of translations in $\RR^{n+1}$ acting on translators has dimension $n+1$, our construction gives at least one dimensional family of pairwise non-congruent translators when $S$ is not a round sphere.  On the other hand, if $S$ is the round sphere, then by the work of Hershkovits \cite{Hershkovits20_Translators} and Choi-Haslhofer-Hershkovits-White \cite{ChoiHaslhoferHershkovitsWhite22_AncientMCF}, any such translator with the tangent flow at $-\infty$ to be $\SSp^{n-1}_{\sqrt{2n-2}}\times \RR$ is a translation of the well-known bowl soliton \cite{AltschulerWu94_Transl} in $\RR^{n+1}$, which is recovered from our construction.
	
	The possible existence of low dimensional singularities in the Theorem comes from the regularity theory of minimizing hypersurfaces. In particular, Theorem \ref{Thm_Intro_main} produces smooth translators with prescribed end in $\RR^{n+1\leq 7}$.  Hereafter in this paper, when we refer to a hypersurface, we allow for a singular set of codimension $\geq 7$ unless otherwise specified. 
	
	\subsection{Translating ends}
	Our construction is highly motivated by the work of Chan \cite{Chan97} to construct minimizing hypersurfaces that are asymptotic to a given strictly minimizing hypercone. In Chan's work, she first constructed a continuous family of exterior minimal graphs over the given cone. As an analogy, we derive,
	\begin{Thm}[c.f. Theorem \ref{Thm_Moduli of transl end }] \label{Thm_Intro_Moduli of Transl end} 
		Suppose $S\subset \RR^n$ is a closed embedded smooth self-shrinker with unit normal field $\nu$ and \textbf{index $I$}. Then for every sufficiently large $z_0>1$, there exists a continuous $\BB^I$-parametrized family of functions $\{u_{\varphi, z_0}\in C^2(S\times [z_0, +\infty))\}_{\varphi\in \BB^I}$ such that for every $\varphi\in \BB^I$, $u_{\varphi, z_0}(\cdot, z)$ polynomially tends to $0$ as $z\to +\infty$ and that \[
		\ES[u_{\varphi, z_0}]:= \{\left(\sqrt{z}(x+u(x, z)\nu_x), z\right): x\in S, z> z_0\},  \]
		satisfies the translator equation in $\RR^n\times (z_0, +\infty)$. Here $\BB^I$ is the unit ball in $\RR^I$.  Such $\ES[u_{\varphi, z_0}]$ is referred to as a \textbf{pl-simple translating end} over $S\times \RR_+$.
		
		Moreover, for different $\varphi \neq \varphi'$, $u_{\varphi, z_0}-u_{\varphi', z_0}$ has a polynomial decaying lower bound;  And up to an exponential decaying error, the above families exhaust all possibly pl-simple translating ends over $S\times \RR_+$.
	\end{Thm}
	Here, a ``pl-simple end'' stands for a ``polynomially asymptotic simple translating end'', and we refer the readers to Definition \ref{Def_Moduli_p-simple Transl end} for the precise meaning of them. A precise description of this family $\{u_{\varphi, z_0}\}_{\varphi\in \BB^I}$ is in Theorem \ref{Thm_Moduli of transl end }. 
	
	As opposed to the minimal surface case where the given minimal cone itself provides a canonical exterior graph, there's no canonical translating end asymptotic to $S\times \RR$. This somehow suggests that the construction of translating ends is less like a perturbative argument.
	
	One technical difficulty for our construction is to study the dominated linear operator of translator-mean-curvature, $(\partial_z^2+\partial_z+z^{-1}L_S)$, where recall $L_S$ is the Jacobi operator of the self-shrinker $S$. This operator is highly degenerate as $z\to\infty$, which makes it hard to obtain uniform estimates. One novel discovery is that this operator has both elliptic and parabolic features.  We obtain the desired asymptotic estimate by using both elliptic estimates and parabolic estimates for different transformations of this operator.  Based on this, Theorem \ref{Thm_Intro_Moduli of Transl end} is proved by a fixed point argument.
	
	We also find a natural \textbf{one-sided deformation} action $\varpi$ of $\RR$ on the space of pl-simple translating end over $S\times\RR_+$. For every $a\neq 0\in \RR$, $\varpi(a)$ maps any pl-simple end $\Sigma_0$ to some pl-simple end $\Sigma_a$ lying on one side of $\Sigma_0$. See Lemma \ref{Lem_Moduli_one-sided transl end}. A similar one-sided deformation action for minimal hypersurfaces is mentioned in \cite[Section 4, (2)]{Chan97}. In the context of mean curvature flow, the one-sided perturbation also plays a significant role, see \cite{ColdingMinicozzi12_generic, CIMW13_EntropyMinmzer, CCMS20_GenericMCF, SunXue2021_initial_closed, SunXue2021_initial_conical}.
	
	\subsection{Complete translators}
	To go from translating ends to complete translators, we use the following variational characterization of translators. Recall that translators are critical points of the weighted area functional \[
	\I(\Sigma)= \int_\Sigma e^z\ d\scH^n(x, z).  \]
	Hence, the translators are minimal hypersurfaces with respect to a noncomplete metric of $\RR^{n+1}$. This elliptic characterization plays a significant role in Ilmanen's elliptic regularization \cite{Ilmanen94_EllipReg}. 
	
	Now we state our theorem on the construction of complete translators.
	\begin{Thm}[c.f. Theorem \ref{Thm_Pf main thm_Transl region}] \label{Thm_Intro_Transl region}
		For each pl-simple translating end $\Sigma_0\subset \RR^n\times (z_0, +\infty)$ constructed in Theorem \ref{Thm_Intro_Moduli of Transl end}, there's a unique closed subset $T[\Sigma_0]\subset \RR^{n+1}$ which contains all $\I$-minimizing translators exponentially asymptotic to $\Sigma_0$, and such that one of the following holds,
		\begin{enumerate} [(i)]
			\item $T[\Sigma_0]$ is the support of a hypersurface, and hence is the unique $\I$-minimizing translator exponentially asymptotic to $\Sigma_0$;
			\item $T[\Sigma_0]$ has non-empty interior, and $\partial T[\Sigma_0] = T^+ \sqcup T^-$, each of $T^\pm$ is an $\I$-minimizing translator exponentially asymptotic to $\Sigma_0$.
		\end{enumerate}
		Moreover, $T[\Sigma_0]$ varies upper-semi-continuously with respect to the ends $\Sigma_0$; And if let $\{\Sigma_a := \varpi(a)[\Sigma_0]\}_{a\in \RR}$ be the one-sided deformations of $\Sigma_0$ as mentioned above, then 
		\begin{align}
			\RR^{n+1} = \coprod_{a\in \RR} T[\varpi(a)[\Sigma_0]].  \label{Intro_R^(n+1) decomp into transl region}	
		\end{align}
	\end{Thm}
	The idea to construct complete translators that are \textbf{fast} asymptotic to a given end $\Sigma_0$ is also motivated by Chan \cite{Chan97}. We briefly describe it here, and the details are carried out in Section \ref{Sec_Pf of Main Thm}. For each $R$ sufficiently large, we minimize $\I$-functional among hypersurfaces with boundary $\Sigma_0\cap \RR^n\times\{R\}$ to find $\Sigma_R$, an $\I$-minimizer with boundary. Then we send $R\to +\infty$ and take the limit of $\Sigma_R$.  The key is to argue that outside a uniform compact subset (independent of $R$), $\Sigma_R$ stays close to $\Sigma_0$. To see this, for a fixed $\delta\ll 1$, we focus on the maximal region where $\Sigma_R$ is $\delta$-close to $\Sigma_0$. The asymptotic analysis of the operator $(\partial_z^2 + \partial_z + z^{-1}L_S)$ enable us to improve this $\delta$-closeness to a fast decaying estimate between $\Sigma_0$ and $\Sigma_R$. Then, suppose for contradiction, the $\delta$-close-region of $\Sigma_R$ and $\Sigma_0$ tends to infinity as $R\to +\infty$, an appropriate parabolic blow down will produce an ancient rescaled mean curvature flow fast asymptotic to $S$ near $-\infty$ and not equal to $S$ itself.  This violates a Liouville-type Theorem of \cite{CCMS20_GenericMCF}.
	
	One may compare the decomposition (\ref{Intro_R^(n+1) decomp into transl region}) with the well-known minimal foliation of Hardt-Simon \cite{HardtSimon85}.   Recall that for a minimizing hypercone $C\subset \RR^{n+1}$, \cite{HardtSimon85} constructed a foliation $\{\cS_t\}_{t\in \RR}$ of $\RR^{n+1}$ by minimizing hypersurfaces, where $\cS_0 = C$ and $\cS_{\pm t} = t\cdot \cS_{\pm 1}$ are smooth minimizing hypersurfaces for every $t>0$. For translators, rigid motions of a single translator $\Sigma\subset\RR^{n+1}$ is never a foliation, unless $\Sigma$ is convex. Moreover, $T[\Sigma_a]$ in (\ref{Intro_R^(n+1) decomp into transl region}) need not be hypersurfaces, see the following Theorem \ref{Thm_Intro_Fattening or not}.
	
	Notice that Theorems \ref{Thm_Intro_Moduli of Transl end} and \ref{Thm_Intro_Transl region} prove Theorem \ref{Thm_Intro_main}, since each two different pl-simple translating ends $\Sigma_\pm$ over $S\times \RR_+$ are asymptotic to each other polynomially and not better, which means the corresponding $T[\Sigma_\pm]$ are distinct. 
	
	Another remark is that our construction (by using the implicit function theorem) only constructed a subset of translating ends and translators, but not necessarily all of them. It would be an interesting question that if our construction can obtain all the translators that are asymptotic to $S\times\RR$.

	
	\subsection{Fattening phenomena}
	
	Another novel discovery in this paper is the fattening phenomenon. As in Theorem \ref{Thm_Intro_Transl region}, $T[\Sigma_0]$ either is the support of a translator, or it has nonempty interior. We call $T[\Sigma_0]$ in case (ii) \textbf{fattening}. We show that non-fattening is a generic phenomenon, while fattening can also happen.
	
	\begin{Thm} [c.f. Corollary \ref{Cor_Pf main thm_generic nonfattening} \& Thereom \ref{Thm_Angenent torus fattens}] \label{Thm_Intro_Fattening or not}
		For every closed embedded self-shrinker $S\subset \RR^n$, for generic pl-simple ends $\Sigma_0$, $T[\Sigma_0]$ does not fattens; 
		On the other hand, suppose $S$ is the Angenent torus \cite{Angenent92_Doughnut} in $\RR^3$. Then there exists a pl-simple translating end $\Sigma_0$, such that $T[\Sigma_0]$ fattens. 
	\end{Thm}
	
	In the minimal surface analogy, Chan \cite[Section 4, (1)]{Chan97} asked whether the minimizing hypersurface fast asymptotic to an exterior minimal graph is unique or not. Our Theorem \ref{Thm_Intro_Fattening or not} provides a negative answer to the analogy of this question for translators.
	
	In our construction, the fattening shows up because of the topology gap. More precisely, we find examples in the one-parameter family of translators in Theorem \ref{Thm_Intro_Transl region} with different topology types. This implies that there must be a jump in this one-parameter family which corresponds to fattening.
	
	We adopt the terminology ``fattening'' from the study of level set flow. Level set flow is a weak formulation of mean curvature flow. Each time slice of the mean curvature flow is viewed as the nodal set of the level set flow function, and the evolution of mean curvature flow is non-unique if the nodal set fattens, namely, it has nonempty interior.
	
	There are several similarities between the fattening of level set flow and the fattening phenomenon in this paper. First, both fattening phenomena lead to the nonuniqueness of solutions. Second, the fattening phenomena rarely happen. For the level set flow, it is known that fattening can only happen for a countably many level sets, hence it is not generic, see \cite[11.4]{Ilmanen94_EllipReg}. We prove that the fattening can only happen for a meager set of ends, c.f. Corollary \ref{Cor_Pf main thm_generic nonfattening}. Theorems \ref{Thm_Intro_Transl region} and \ref{Thm_Intro_Fattening or not} provide a new perspective on understanding the fattening of mean curvature flow.
	
	After the first version of this paper, there has been much progress on the construction of fattening level set flows (with smooth initial data). We refer the readers to recent progress such as \cite{IlmanenWhite25_Fattening, ChodoshDanielsHolgateSchulze24_MCFConical, lee2024closed, ketover2024self}.
	
	\subsection{Organization of the paper.}
	In Section \ref{Sec_Pre}, we start with a brief review of some basic notions in geometric measure theory and mean curvature flow. Then we discuss translators and their tangent flow at $-\infty$. 
	Section \ref{Sec_Moduli} is devoted to studying pl-simple translating ends over a closed self-shrinker times $\RR$. This includes the technical analysis of the operator $(\partial_z^2 + \partial_z + z^{-1}L_S)$.
	In Section \ref{Sec_Pf of Main Thm}, we associate to every pl-simple end at least one complete $\I$-minimizing translator and prove Theorem \ref{Thm_Intro_Transl region}.
	In Section \ref{Sec_Fattening}, we restrict to the case where $S$ is the Angenent torus.  Among the $\RR$-family of rotationally symmetric decomposition (\ref{Intro_R^(n+1) decomp into transl region}) of $\RR^{n+1=4}$, the topology of $T[\Sigma_a]$ as $a\to \pm\infty$ is proved to be different. This reflects the fattening phenomenon for some $T[\Sigma_a]$.

	\subsection*{Acknowledgement}	
	The authors would like to thank Professor Sigurd Angenent, Professor Xiaodong Cao, Professor Andr\'e Neves, and Professor Brian White for their interest. The second-named author would like to thank his advisor, Fernando Cod\'a Marques, for his support. We are also grateful to the reviewers for many helpful comments.

	\section{Preliminaries} \label{Sec_Pre}
	Throughout this paper, let $\RR^N$ be the $N$-dimensional Euclidean space. Let
	\begin{itemize}
		\item $\BB_r^N(x)\ $ be the open ball of radius $r$ in $\RR^N$ centered at $x$; we may omit the superscript $N$ if there's no confusion about dimension; we may write $\BB_r:= \BB_r(\mathbf{0})$ to be the ball centered at the origin $\mathbf{0}$, and write $\BB^N$ be the unit ball in $\RR^N$ centered at the origin;
		\item $\SSp^{N-1} := \partial \BB_1^N\ $ be the unit sphere;
		\item $\RR_{\geq a}:= [a, +\infty)\ $; similar notations are used for $\ZZ$ in place of $\RR$ and $>, \leq, <$ in place of $\geq$; 
		\item $\eta_{x, r}\ $ be the affine transformation of $\RR^N$, sending $y$ to $(y-x)/r$;
		\item $\scH^k\ $ be the $k$-dimensional Hausdorff measure;
		\item $g_{\Euc}\ $ be the Euclidean metric on $\RR^N$.
	\end{itemize}
	
	For a subset $E\subset \RR^N$, let $\Int(E)$ and $\Clos(E)$ be its interior and closure; $\partial E:= \Clos(E)\setminus \Int(E)$ be its topological boundary;  $\BB_r(E):= \bigcup_{x\in E} \BB_r(x)$ be its $r$-neighborhood; $(E-x)/r:= \eta_{x, r}(E)$ be its translation and dilation in $\RR^N$.
	
	For a hypersurface $\Sigma \subset \RR^N$ with a unit normal field $\nu$, let
	\begin{itemize}
		\item $A_\Sigma := \nabla \nu\ $ be the second fundamental form of $\Sigma$. 
		\item $H_\Sigma := -tr_\Sigma(A_\Sigma)\ $ be the scalar mean curvature of $\Sigma$, and $\vec{H}_\Sigma:= -H_\Sigma \nu$ be the mean curvature vector; In general, for a submanifold $\Sigma\subset \RR^N$ of any codimension, the mean curvature vector $\vec{H}_\Sigma := \nabla_{\partial_i}\partial^i$ is also well defined;
		\item $\graph_\Sigma(u):= \{x+u(x)\nu_x: x\in \Omega\}\ $ be the graph over $\Sigma$ of some function $u$ defined on a subdomain $\Omega\subset \Sigma$.
	\end{itemize}

	\subsection{Mean curvature flow and entropy} \label{Subsec_Prelim MCF and Entropy}
	Let $N\geq n>0$ be integers, $J\subset\RR$ be an interval. A family of $n$-dimensional submanifolds $\{M_t\subset \RR^N\}_{t\in J}$ is \textbf{flowing by mean curvature} if they satisfy	\[
	\left(\partial_t X(t)\right)^\perp = \vec{H}_{M_t}.
	\]	
	Here $X(t)$ is the position vector of $M_t$, $\perp$ means projection onto the normal bundle of $M_t$. 
	
	In this paper, we will be focused on the asymptotics of a mean curvature flow at $-\infty$ time. In order to study the asymptotics of mean curvature flow, Huisken \cite{Huisken90} introduced the \textbf{rescaled mean curvature flow (RMCF)}. Suppose $\{M_t\}$ is defined for $t\in(-\infty,0)$, then we can define a new flow $\tilde{M}_\tau:=e^{\tau/2}\cdot M_{-e^{-\tau}}$ for $\tau\in\RR$, satisfying the equation
	\begin{equation}
		(\partial_\tau \tilde{X})^\perp = \vec{H}_{\tilde{M}_\tau} + \frac{\tilde{X}(\tau)^\perp}{2}\,. 
	\end{equation}
	Recall that \textbf{Gaussian area} of a $n$-dimensional submanifold $S\subset\RR^N$ is defined by 
	\begin{equation}
		\displaystyle{\cF[S]:=(4\pi)^{-n/2}\int_Se^{-\frac{|x|^2}{4}}}\ d\scH^n.
	\end{equation}
	Huisken showed that any RMCF $\{\tilde{M}_\tau\}_\tau$ is a gradient flow of $\cF$. As a consequence, $\cF[\tilde{M}_\tau]$ is non-increasing in $\tau$, and this fact is known as Huisken's monotonicity. Based on the monotonicity of $\cF$, any subsequential (weak) limit of $\tilde{M}_\tau$ as $\tau\to-\infty$ should be a critical point of $\cF$, which is called a \textbf{self-shrinker}, i.e. a submanifold (integral varifold) $S$ satisfying \[
	\vec{H}_S+\frac{X^\perp}{2}=0\ .
	\]
	
	Let $S\subset\RR^N$ be a smooth self-shrinker of codimension one. By \cite{ColdingMinicozzi12_generic}, for $u\in C_c^\infty(S)$, the second variation of $\cF$ is, \[
	\delta^2 \cF[S](u, u) = C_N\cdot \int_S \left( |\nabla_S u|^2 - \left(|A_S|^2 + \frac{1}{2}\right)u^2 \right) e^{-|x|^2/4}\ d\scH^{N-1}(x),   \]
	for some constant $C_N>0$.  The Euler-Lagrangian operator of $\delta^2\cF[S]$ with respect to the Gaussian $L^2$ norm is the \textbf{Jacobi operator} of $S$,
	\begin{align}
		L_S := \Delta_S - \frac{X}{2}\cdot\nabla_S + |A_S|^2 + \frac{1}{2} \,.  \label{Pre_Gaussian Jac oper}
	\end{align}
	
	For a general submanifold $S\subset \RR^N$, $\vec{H}_S+\frac{X^\perp}{2}$ is called the \textbf{shrinker mean curvature vector} of $S$. When $S$ is a connected hypersurface with a unit normal vector field $\nu$, $\vec{H}_S+\frac{X^\perp}{2}=-(H_S-\frac{\langle X,\nu\rangle}{2})\nu$, and $(H_S-\frac{\langle X,\nu\rangle}{2})$ is called the \textbf{shrinker mean curvature}. If $(H_S-\frac{\langle X,\nu\rangle}{2})>0$ (resp. $(H-\frac{\langle X,\nu\rangle}{2})<0$), $S$ is called \textbf{shrinker mean convex} (resp. \textbf{shrinker mean concave}).
	
	Colding-Minicozzi \cite{ColdingMinicozzi12_generic} introduced a quantity which is called \textbf{entropy}. Suppose $S$ is an $n$-dimensional submanifold in $\RR^N$, the entropy is defined as
	\begin{equation}
		\lambda[S]:=\sup_{x_0\in\RR^N, t_0>0}\cF[t_0^{-1}(S-x_0)]\ .
	\end{equation}
	
	By Huisken's monotonicity formula, if $\{M_t\}_t$ is a mean curvature flow, then $\lambda[M_t]$ is non-increasing as $t$ increases. If $S$ is a self-shrinker, \cite{ColdingMinicozzi12_generic} showed that $\cF[S]=\lambda[S]$.
	
	\subsection{Geometric measure theory} \label{Subsec_Prelim GMT}
	We refer the readers to \cite{Simon83_GMT} for detailed definitions and discussions in geometric measure theory. 
	
	Suppose $N>n>0$ are integers, $\Omega\subset\RR^N$ is an open set. We use $\cI\cV_n(\Omega)$ to denote the space of integral $n$-varifolds defined on $\Omega$. $\|V\|$ denotes the associated Radon measure for $V\in\cI\cV_n(\Omega)$. For a smooth $n$-submanifold $S$, we write $|S|$ to be its associated integral $n$-varifold.
	
	We call a family of integral varifold $V_j\in\cI\cV_n(\Omega)$ \textbf{$\mbfF$-converges} to $V_\infty$, if they converges as Radon measure over $\Omega\times \mathbf{Gr}_n(\RR^N)$, where $\mathbf{Gr}_n(\RR^N)$ is the $n$-Grassmannian on $\RR^N$.
	
	There is also a geometric measure-theoretic notion of mean curvature flow, known as the \textbf{Brakke motions}, \cite{Brakke78, Ilmanen94_EllipReg}. A family of Radon measures $\{\mu_t\}_{t\geq 0}$ defined on $\Omega$ is a \textbf{Brakke motion} if for all test function $\phi\in C^2_c(\Omega)$ with $\phi\geq 0$,	\[
	\limsup_{s\to t}\frac{\mu_s(\phi)-\mu_t(\phi)}{s-t}	\leq \int (-\phi H^2+\nabla^\bot \phi\cdot \vec{H})d\mu_t,  \]
	where $\vec{H}$ is the mean curvature vector of $\mu_t$ whenever $\mu_t$ is rectifiable and has $L^2$-mean curvature in the varifold sense, otherwise the right-hand side is defined to be $-\infty$. In this paper, we only focus on \textbf{integral Brakke motion}, i.e. $\mu_t$ are associated Radon measure of integral varifolds in $\cI\cV_n(\Omega)$ for a.e. $t$.  The \textbf{support} of a Brakke motion is a closed subset of spacetime given by, \[
	\spt(\{\mu_t\}_{t\geq 0}):= \Clos\Big(\bigcup_{t\geq 0} \spt(\mu_t)\times\{t\}\Big) \subset \RR^N\times \RR.  \]
	We say that a sequence of flow $\{\mu_t^j\}_{t\geq 0}$ converges to $\{\mu_t^\infty\}_{t\geq 0}$ \textbf{in the Brakke sense}, if $\mu^j_t$ measure-converges to $\mu^\infty_t$ for all $t$, and the the associated varifolds converge for all but countably many $t$. Brakke \cite{Brakke78} proved that any sequence of integral Brakke motion with uniformly bounded area has a converging subsequence in the Brakke sense. Moreover, if the flow converges in the Brakke sense, then by Huisken's Monotonicity and avoidance principle, the support of the flow also converges locally in Hausdorff distance.
	
	$n$-varifolds are generalizations of submanifolds, so it is natural to define the translation and dilation of a varifold $V$. For any $(x_0,t_0)\in \RR^N\times(0,+\infty)$, we denote for simplicity $(V-x_0)/t_0$ to be the push forward of $V\in \cI\cV_n(\Omega)$ by $\eta_{x_0, t_0}$, in other words, for any measurable set $E\subset((\Omega-x_0)/t_0)\times \mathbf{Gr}_n(\RR^N)$, \[
	\left((V-x_0)/t_0)\right)(E) := (\eta_{x_0,t_0}{}_\sharp V)(E) = t_0^{-n} \cdot V((\eta_{x_0, t_0}^{-1}\times id_{\mathbf{Gr}_n(\RR^N)})(E)).  \]
	The definition of the Gaussian area $\cF$ and the entropy $\lambda$ are also naturally extended to varifolds. Moreover, the entropy is lower semi-continuous under varifold convergence.
	
	\begin{Prop}\label{prop:semicontinuity_entropy}
		Suppose $\{V_j\}_{j=1}^\infty$ is a sequence of integral $n$-varifolds in $\RR^N$ and $V_j$ $\mbfF$-converges to a integral $n$-varifold $V_\infty$ as $j\to\infty$. Then \[
		\liminf_{j\to\infty}\lambda[V_j]\geq \lambda[V_\infty].		\]
	\end{Prop}
	
	\begin{proof}
		It suffices to show that $\liminf_{j\to\infty}\cF[(V_j-x_0)/t_0]\geq \cF[(V_\infty-x_0)/t_0]$. Up to a translation and dilation, we only need to show $\liminf_{j\to\infty}\cF[V_j]\geq \cF[V_\infty]$. This is straightforward from the lower semi-continuity of the integrals over varifolds.
	\end{proof}

	In general, the varifolds may not be smooth, and $\mbfF$-convergence can be complicated. Thanks to Brakke-White's regularity theorem, we can better characterize convergence if the flows are regular.
	
	\begin{Thm}[\cite{White05_MCFReg}] \label{Thm_Pre_Brakke Reg}
		Suppose $\epsilon>0$, $\{\Sigma_t^j\}_{t\in (0,\epsilon)}$ is a sequence of smooth mean curvature flow in $\RR^N$ converging to a smooth mean curvature flow $\{\Sigma_t^\infty\}_{t\in (0,\epsilon)}$ with multiplicity $1$ in the Brakke sense. Then the convergence is in $C_{loc}^\infty(\RR^N\times(0,\epsilon))$.
	\end{Thm}

	\subsection{Translators} \label{Subsec_Prelim Transl}
	Now we restrict our attention to translators. Let $N\geq n >0$ be integers; parametrize $\RR^{N+1} = \RR^N\times\RR$ by $(x, z)$. We may view $z$ to be the last coordinate function, with $\partial_z$ being the coordinate vector. We may also abuse the notation to view $\partial_z = (\orig, 1)\in \RR^N\times \RR$. Throughout this paper, we only consider translators moving in the $\partial_z$ direction.
	
	A \textbf{translator}, or a translating soliton, is a submanifold $\Sigma\subset \RR^{N+1}$ satisfying 
	\begin{align}
		\vec{H}_\Sigma - \partial_z^\perp=0\,.  \label{Pre_Transl equ}
	\end{align}
	Motivated by this equation, $\vec{H}_\Sigma - \partial_z^\perp$ is also called the \textbf{translator-mean-curvature vector} for a general submanifold $\Sigma\subset \RR^{N+1}$. Translators are named due to the following fact: $\Sigma$ is a translator if and only if $\{\Sigma + t\partial_z\}_{t\in \RR}$ is a mean curvature flow in $\RR^{N+1}$. 
	
	There is a variational characterization of translators introduced by Ilmanen \cite{Ilmanen94_EllipReg}.  For $\epsilon>0$ and an $n$-dimensional submanifold (possibly with boundary) $\Sigma\subset \RR^{N+1}$, define 
	\begin{align}
		\I^{\epsilon}[\Sigma]:= \int_\Sigma e^{z/\epsilon}\ d\scH^n(x, z)\,.  \label{Pre_Transl I functional}
	\end{align}
	For simplicity, we write $\I:= \I^1$.  Clearly, this functional is invariant under translations of submanifolds in $x$-directions. Under rescalings, we have $
	\I[\Sigma] = \epsilon^{-n}\I^\epsilon[\epsilon\cdot \Sigma]$. By \cite{Ilmanen94_EllipReg}, critical points of $\I$ are translators. (Note that our functional $\cI$ differs from the one in \cite{Ilmanen94_EllipReg} by a sign in $z$, which corresponds to the opposite translating direction.)
	It's also clear from (\ref{Pre_Transl I functional}) that translators are minimal submanifolds under the conformal metric $g = e^{2z/n}g_\Euc$, and the translator-mean-curvature vector is also proportional to the mean curvature vector under $g$. In particular, the strong maximum principle of Solomon-White \cite{SolomonWhite89_Maxim} and Ilmanen \cite{Ilmanen96} also applies to translators.
	
	It is not hard to see that a complete translator can not be compact. In fact, given a complete translator $\Sigma$, for sufficiently large $z_0\gg 0$, $\Sigma\cap\{z=z_0\}\not=\emptyset$. To understand the asymptotic behavior of $\Sigma$ as $z\to+\infty$, we need a blow-down analysis. For $\tau\in \RR$, we let 
	\begin{align}
		\tilde{\Sigma}(\tau) := e^{\tau/2}\cdot(\Sigma - e^{-\tau}\partial_z)\,,   \label{Pre_RMCF assoc to transl}
	\end{align}
	Recall that by \cite{Huisken90}, $\{\tilde{\Sigma}(\tau)\}_{\tau\in \RR}$ solves the RMCF equation, 
	\begin{align}
		(\partial_\tau \tilde{X})^\perp = \vec{H}_{\tilde{\Sigma}(\tau)} + \frac{\tilde{X}(\tau)^\perp}{2}\,.  \label{Pre_RMCF equ}
	\end{align}
	
	The following lemma is proved in \cite{Ilmanen94_EllipReg}, which shows that the blow-down limit of translators splits in $\RR_z$-direction.
	\begin{Lem} \label{Lem_Pre_blow down transl split}
		Let $\Omega\subset \RR^N$ be an open subset;  $\epsilon_j\searrow 0$, $R_j\epsilon_j\nearrow +\infty$ be sequences of constants; let $V_j\in \cI\cV_n(\Omega\times (0,R_j))$ be stationary with respect to $\I^{\epsilon_j}$, $1\leq j<+\infty$. Suppose when $j\to \infty$, $\{V_j(t):= V_j + (t/\epsilon_j)\partial_z\}_{-R_j\epsilon_j< t<0}$ converges to $\{V_\infty(t)\}_{t<0}$ in Brakke sense.  
		Then $V_\infty(t)$ splits in $\RR_z$-direction for all but countably many $t<0$.
	\end{Lem}
	\begin{proof}
		The proof is in \cite[8.8]{Ilmanen94_EllipReg}. Note that in \cite{Ilmanen94_EllipReg}, the sequence $\{V_j\}_{j=1}^\infty$ is generated by the elliptic regularization, but the proof of the splitting of $V_\infty(t)$ does not rely on this assumption.
	\end{proof}
	
	By Huisken's monotonicity formula, the Gaussian area is monotone non-increasing along the RMCF $\{\tilde{\Sigma}(\tau)\}_{\tau\in \RR}$, and if \[
	\lim_{\tau\to -\infty} \cF[\tilde{\Sigma}(\tau)] < +\infty\,,  \] 
	then any subsequential measure-theoretic limit of $\tilde{\Sigma}(\tau)$ when $\tau\to -\infty$ will be a self-shrinker in $\RR^{N+1}$, known as a \textbf{tangent flow} of $\{\tilde{\Sigma}(\tau)\}$ (for simplicity, we call it the tangent flow of the translator $\Sigma$ if there's no confusion) at $-\infty$.   As a corollary of Lemma \ref{Lem_Pre_blow down transl split}, we have,

	\begin{Cor} \label{Cor_Pre_Blow down transl MCF split}
		Suppose $\Sigma\subset \RR^{N+1}$ is an $n$-dimensional translator with finite entropy. Then any tangent flow $\Gamma\in \cI\cV_n(\RR^{N+1})$ of $\Sigma$ at $-\infty$ is a self-shrinker splitting in $z$-direction, in other words, for any $\alpha\in\RR$, $\Gamma+\alpha\partial_z=\Gamma$.  Moreover, $\lambda[\Sigma] = \cF[\Gamma]$.
	\end{Cor}
	\begin{proof}
		The tangent flow being splitting follows directly from Lemma \ref{Lem_Pre_blow down transl split}.  To compute the entropy of $\Sigma$, first by Proposition \ref{prop:semicontinuity_entropy} and translation dilation invariance of the entropy, $\lambda[\Sigma]\geq \lambda[\Gamma] = \cF[\Gamma]$. On the other hand, for any $x_0\in\RR^N$, $t_0\in(0,\infty)$ and $s>0$, by Huisken's monotonicity formula,
		\[
		\cF[t_0^{-1}(\Sigma-x_0)]
		\leq 
		\cF[(t_0^2+s)^{-1/2}(\Sigma-s\partial_z-x_0)].
		\]
		Now suppose $\Gamma$ is the limit of $s_i^{-1/2}\cdot (\Sigma-s_i\partial_z)$ for a sequence $s_i\nearrow+\infty$. Then taking $s=s_i-t_0^2$ and notice that $s_i^{-1/2}x_0\to 0$, $s_i^{-1/2}t_0^2\searrow0$ and the following non-concentration of $\cF$ near infinity, 
		\begin{align}
			\limsup_{R\nearrow +\infty}\ \limsup_{\tau\to -\infty}\ \cF[\tilde{\Sigma}_\tau \setminus \BB_R] = 0,  \label{Pre_cF-nonconcentration near infty} 
		\end{align}
		we see that $\cF[t_0^{-1}(\Sigma-x_0)]\leq \cF[\Gamma]$. This implies that $\lambda[\Sigma]\leq \cF[\Gamma]$. Hence $\lambda[\Sigma]= \cF[\Gamma]$.
		
		To prove (\ref{Pre_cF-nonconcentration near infty}), notice that for every $R>1$, 
		\begin{align*}
			\cF[\tilde{\Sigma}_\tau \setminus \BB_R] & = R^n\int_{(\tilde{\Sigma}_\tau /R)\setminus \BB_1} (4\pi)^{-n/2} e^{-|y|^2(R^2-1)/4}\cdot e^{-|y|^2/4}\ d\scH^n(y) \\
			& \leq R^n e^{-(R^2-1)/4}\cdot \cF[\tilde{\Sigma}_\tau/R]\ \  \leq R^n e^{-(R^2-1)/4}\cdot \lambda[\Sigma],
		\end{align*}
		where the RHS does not depend on $\tau$ and tends to $0$ as $R\nearrow +\infty$.
	\end{proof}
	
	In the application, we shall also compute the entropy of translators with simple ends. The following lemma guarantees that Corollary \ref{Cor_Pre_Blow down transl MCF split} applies in this case.
	\begin{Lem} \label{Lem_Pre_Transl w simple end has entropy finite}
		Suppose $\Sigma\subset \RR^{N+1}$ is a translator with unique tangent flow at $-\infty$ to be $S\times\RR$, where $S$ is a smooth closed self-shrinker.  Then $\lambda[\Sigma]<+\infty$.
	\end{Lem}
	\begin{proof}
		We first claim that there exists $\bar\tau\in\RR$ and $R_0>0$ such that for all $\tau>\bar\tau$, $\tau^{-1/2}\Sigma\cap\{z=\tau^{1/2}\}\subset B_{R_0}\times \{z=\tau^{1/2}\}$. Suppose for contradiction that there exists $(x_j,z_j)\in\Sigma$ such that $|z_j|\to\infty$ and $z_j^{-1}|x_j|^2\to\infty$. Let $\lambda_j:=|x_j|$ and we define $\Sigma_j(t):=\lambda^{-1}_j(\Sigma+\lambda^2_jt \partial_z)$ to be the blow-down sequence of MCFs. Because the unique tangent flow of $\Sigma$ at $-\infty$ is $S\times\RR$, $\Sigma_j(t)$ converges to $\{\sqrt{-t}S\times\RR\}_{t\leq 0}$ in the Brakke sense as $j\to\infty$. Note that $(\lambda_j^{-1}x_j,0)\in\Sigma_j(-|x_j|^{-2}z_j)$, so by the convergence of $\Sigma_j(t)$, the spacetime points $(\lambda_j^{-1}x_j,0,-|x_j|^{-2}z_j)\to (x_\infty,0,0)\in \spt(\{\sqrt{-t}S\times\RR\}_{t\leq 0})$. But $|x_\infty|=1$, which is a contradiction because $\spt(\{\sqrt{-t}S\times\RR\}_{t\leq 0})\cap\{t\geq 0\}=\orig\times\RR\times\{0\}$.
		
		With this claim, Brakke-White regularity Theorem \ref{Thm_Pre_Brakke Reg} implies that for any $\epsilon>0$, there exists $\tau_0>0$ such that for $\tau>\tau_0$, $\tau^{-1/2}\Sigma\cap\{z=\tau^{1/2}\}$ is a graph of function $v_\tau$ over $S\times\{\tau^{1/2}\}$ inside $\RR^{N-1}\times\{\tau^{1/2}\}$, with $\|v_\tau\|_{C^{2}}\leq \epsilon$. We denote by $\widetilde\Sigma:=\Sigma\cap\{z\geq 2\tau_0\}$, and $\Sigma_\tau:=\Sigma\cap\{z=\tau\}$.
		
		By the definition of $\lambda$, it is straightforward to check that for two hypersurfaces $\Sigma_1$ and $\Sigma_2$, $\lambda[\Sigma_1\cup \Sigma_2]\leq \lambda[\Sigma_1]+\lambda[\Sigma_2]$. So it suffices to show $\lambda\left[\ \overline{\Sigma\backslash\widetilde\Sigma}\ \right]<+\infty$ and $\lambda[\widetilde\Sigma]<+\infty$. We claim that $\overline{\Sigma\backslash\widetilde\Sigma}=\Sigma\cap\{z\leq 2\tau_0\}$ is compact. Suppose by contradiction that $(x_j,z_j)\in\Sigma$ such that $\lambda_j^{2}:=|x_j|^2-z_j\to\infty$. Define $\Sigma_j(t):=\lambda_j^{-1}(\Sigma+\lambda_j^2t\partial_z)$, and again, $\Sigma_j(t)$ converges to $\{\sqrt{-t}S\times\RR\}_{t\leq 0}$ in Brakke sense. Similar to the proof of the first claim, $(\lambda_j^{-1}x_j,0,-\lambda_j^{-2}z_j)$ converges to a point in $\spt(\{\sqrt{-t}S\times\RR\}_{t\leq 0})$. But $(\lambda_j^{-1}x_j,0,-\lambda_j^{-2}z_j)$ converges to $(x_\infty,0,z_\infty)$ with $|x_\infty|^2+z_\infty=1$, and $z_\infty\geq 0$. This point is not in $\spt(\{\sqrt{-t}S\times\RR\}_{t\leq 0})$, which is a contradiction.
		
		As a consequence of the claim that $\overline{\Sigma\backslash\widetilde\Sigma}$ is compact, by \cite[Lemma 7.2]{ColdingMinicozzi12_generic}, $\lambda\left[\ \overline{\Sigma\backslash\widetilde\Sigma}\ \right]<+\infty$. It remains to show $\lambda[\widetilde\Sigma]<+\infty$. From \cite[Section 2]{baldauf-Sun2018_sharp}, if $\epsilon$ is chosen sufficiently small, when $\tau\geq 2\tau_0$, $\lambda[\Sigma_\tau]<\lambda[S]+1$, where $\Sigma_\tau$ is viewed as a hypersurface in $\RR^{N}$. Then for any $x_0\in\RR^{N}$, $z_0\in\RR$ and $t_0>0$, 
		\[\begin{split}
			&\int_{\widetilde{\Sigma}}(4\pi t_0)^{-{N-1}/2}e^{-\frac{|x-x_0|^2+(z-z_0)^2}{4t_0}}d\cH^{N}(x,z)
			\\
			&=
			\int_{\tau=2\tau_0}^\infty\int_{\Sigma_\tau}(4\pi t_0)^{-{N-1}/2}e^{\frac{-|x-x_0|^2}{4t_0}}
			\frac{1}{|\nabla_\Sigma z|}
			d\cH^{N-1}(x)
			e^{-\frac{(\tau-\tau_0)^2}{4t_0}}d\cH^1(\tau)
			\\
			&\leq 
			2(\lambda[S]+1)
			\int_{\tau=2\tau_0}^\infty
			(4\pi t_0)^{-1/2}e^{-\frac{(\tau-\tau_0)^2}{4t_0}}
			d\cH^1(\tau)\leq 2(\lambda[S]+1)
			.
		\end{split}
		\]
		Here we use the coarea formula, and we notice that $|\nabla_\Sigma z|^2=|(e_z)^\top|^2=1-|(e_z)^\bot|^2=1-H^2$, and when $\tau_0$ is chosen sufficiently large, $\widetilde{\Sigma}$ has $H<\sqrt{4/3}$, which yields $|\nabla_\Sigma z|^{-1}\leq 2$. This concludes that $\lambda[\widetilde{\Sigma}]<+\infty$.	
	\end{proof}
	
	In view of Corollary \ref{Cor_Pre_Blow down transl MCF split}, an interesting question is whether the tangent flow at $-\infty$ of a general translator is unique. To the best of the authors' knowledge, the uniqueness of cylindrical tangent flow for mean curvature flow is still widely open. Only a few special cases are proved, see \cite{CM15_Lojasiewicz, Zhu20_Lojasiewicz}.  On the other hand, in view of the work of Simon \cite{Simon89} on minimal graphs, it's plausible to conjecture that tangent flows at $-\infty$ of a general translator is unique, provided one of the tangent flow is a smooth closed or asymptotic conic self-shrinker$\times \RR$ with multiplicity $1$.


	\section{Moduli Space of pl-Simple Translating Ends} \label{Sec_Moduli}
	Hereafter in this paper, we only work in codimension $1$ case, though many of the discussions also work in higher codimensions. 
	\begin{Def} \label{Def_Moduli_Transl ends}
		Let $n\geq 2$, $S\subset \RR^n$ be a self-shrinker, $m\in \ZZ_{\geq 1}$; $\Sigma\subset \RR^n\times \RR_{>z_0}$ be a properly embedded hypersurface.  We call $\Sigma$ a \textbf{translating end} over $m|S\times \RR_+|$ if $\Sigma$ satisfies (\ref{Pre_Transl equ}) and when $R\to \infty$, $|R^{-1/2}(\Sigma - R\partial_z)|$ $\mbfF$-converges to $m|S\times \RR|$.  
		
		We call $\Sigma$ a \textbf{simple end} over $S\times \RR_+$ if $m = 1$.
	\end{Def}
	In the present paper, we shall only deal with smooth closed self-shrinkers $S\subset \RR^n$. The abundance of such self-shrinkers has been established by \cite{Angenent92_Doughnut, Moller11_ClosedShrinker, DruganLeeNguyen18_Survey_SymShrinker, KapouleasMcGrath20_DoublingShrinker, Riedler22_ClosedShrinker}.  We may discuss the case of asymptotic conic self-shrinkers in future works.
	
	Given a self-shrinker $S\subset \RR^n$ with unit normal $\nu=\nu_S$, let $\ES\subset \RR^{n+1}$ be the \textbf{auxiliary end} given by
	\begin{align}
		\ES:= \{(x, z)\in \RR^{n+1}: z>0, x/\sqrt{z}\in S\}  \label{Moduli_Auxl end}
	\end{align}
	We shall parametrize $\ES$ by $\Phi: S\times \RR_+\to \ES$, $(x,z)\mapsto (\sqrt{z}x, z)$. More generally, for $z_0>0$ and $u\in C^1_{loc}(S\times \RR_{>z_0})$, define 
	\begin{align}
		\Phi_u: S\times\RR_{>z_0} \to \RR^{n+1},\ \ (x, z)\mapsto (\sqrt{z}(x+u(x, z)\nu_x), z).  \label{Moduli_param simple end Phi_u}
	\end{align}
	For simplicity we denote $\ES[u]:= \Phi_u(S\times \RR_{>z_0})$.  Clearly, $\ES = \ES[\orig]$ and $\Phi=\Phi_\orig$. 
	\begin{Lem} \label{Lem_Moduli_Simple end <=> u to 0}
		Let $S\subset \RR^n$ be a smooth closed self-shrinker.  Then a translating end $\Sigma\subset \RR^n\times \RR_{>z_0}$ is a simple end over $S\times \RR_+$ if and only if there exists $z_0'>1$ and $u\in C^1(S\times\RR_{>z_0'})$ such that $\Sigma\cap \RR^n\times\RR_{>z_0'} = \ES[u]$, and that when $R\to +\infty$, 
		\begin{align}
			\sup_{S\times [R, 2R]} \left( |u| + |\nabla u| + R|\partial_z u|
			\right) \to 0 \,,  \label{Moduli_Simple end <=> u to 0, est}   
		\end{align}
		where $\nabla$ denote the gradient in $S$ direction.
	\end{Lem}
	\begin{proof}
		The proof is basically by changing coordinates and applying White's $\epsilon$-regularity \cite{White05_MCFReg}.  
		First recall that since $\Sigma$ satisfies the translator equation (\ref{Pre_Transl equ}), we have $\{\tilde{\Sigma}(\tau):= e^{\tau/2}(\Sigma - e^{-\tau}\partial_z)\}$ is a RMCF. 
		
		Suppose that $\Sigma$ is a simple end over $S\times \RR_+$, then by White's $\epsilon$-regularity \cite{White05_MCFReg}, when $\tau\to -\infty$, $\tilde{\Sigma}(\tau)$ $C^\infty_{loc}$-converges to $S\times \RR$. Hence for $\tau \leq \tau_0\ll -1$, by the same argument as Lemma \ref{Lem_Pre_Transl w simple end has entropy finite}, we can write \[
		\tilde{\Sigma}(\tau) \cap \RR^n\times [-1, 1] = \graph_{S\times \RR
		}(\tilde{u}(\cdot, \cdot\ ; \tau)),   \]
		for some $\tilde{u}\in C^2(S\times [-1, 1]\times (-\infty, \tau_0))$ and $\tilde{u}(\cdot, \cdot\ ; \cdot + \tau)$ $C^2$-converges to $\orig$ as $\tau\to -\infty$.  Therefore when $\tau\ll -1$, we have $\Sigma\cap \RR^{n}\times \RR_{\geq e^{-\tau}} = \ES[u]$, where \[
		u(x, e^{-\tau}) := \tilde{u}(x, 0; \tau).   \]
		The desired estimates (\ref{Moduli_Simple end <=> u to 0, est}) then follows from the $C^1$-estimates on $\tilde{u}$ (note that $|R\partial_z u(x,R)|=|\partial_\tau\tilde u(x,0;-\log R)|$ by the chain rule).

		Conversely, suppose $\Sigma = \ES[u]$ for some $u\in C^2(S\times \RR_{>z_0'})$ with estimate (\ref{Moduli_Simple end <=> u to 0, est}), then for each $\tau\leq -\log z_0'$, 
		\begin{align}
			\tilde{\Sigma}(\tau) = \left\{\left( \sqrt{1+ e^{\tau/2}\hat{z}}\cdot(x + u(x, e^{-\tau}+e^{-\tau/2}\hat{z})\nu_x), \hat{z} \right): \hat{z}\in [e^{\tau/2}z_0'-e^{-\tau/2}, +\infty) \right\}.  \label{Moduli_Coord change btwn transl and RMCF}
		\end{align}
		
		To rewrite such $\tilde{\Sigma}(\tau)$ as a graph over $S\times \RR$, we introduce the following notation.  By Lemma \ref{Lem_App_Scal graph}, there exists some $0<\vartheta_S\ll 1$ such that for every $u\in C^1(S\times [0, 1])$ with $\|u\|_{C^1}\leq \vartheta_S$, there's a unique $\check{u}\in C^1(S\times [0, 1]\times (1-\vartheta_S, 1+ \vartheta_S))$ such that for every $z\in [0, 1]$ and every $|a-1|<\vartheta_S$, we have \[
		\graph_S(\check{u}(\cdot, z, a)) = a\cdot \graph_S(u(\cdot, z)),   \]
		and estimates,
		\begin{align}
			\begin{split}
				\|\partial_a \check{u} \|_{C^0,S\times[0,1]} & \leq C_S , \\
				\|\check{u}(\cdot, \cdot, a)\|_{C^1, S\times [0, 1]} & \leq C_S (\|u\|_{C^1, S\times [0, 1]} + |a-1|), \ \ \ \ \ \forall |a-1|<\vartheta_S. 
			\end{split}  \label{Moduli_scaling graphs est}
		\end{align}
		
		With this notation, for each $\tau\ll -\log z_0$ and each $0<L<\vartheta_Se^{-\tau/2}$, \[
		\tilde{\Sigma}(\tau) \cap \RR^n\times [-L, L] = \graph_{S\times \RR
		}(\tilde{u}(\cdot, \tau)),   \]
		for some \[
		\tilde{u}(x, \hat{z}, \tau) = \check{v}(x, \hat{z}, \sqrt{1+ e^{\tau/2}\hat{z}}),   \]
		where $v(x, \hat{z}):= u(x, e^{-\tau/2}\hat{z}+e^{-\tau})$.  Then (\ref{Moduli_Simple end <=> u to 0, est}) and (\ref{Moduli_scaling graphs est}) together show that $\tilde{u}(\cdot, \cdot, \tau)$ $C^1_{loc}$-converges to $0$ as $\tau\to -\infty$. 
	\end{proof} 
	
	\begin{Def} \label{Def_Moduli_p-simple Transl end}
		We call $\ES[u]\subset \RR^{n+1}$ a \textbf{polynomially asymptotic simple translating end} (or simply a \textbf{pl-simple end}) over $S\times \RR_+$ if it is a simple end over $S\times \RR_+$ and there exists $\epsilon > 0$ such that \[
		\limsup_{R\to +\infty} \|u(\cdot, R)\|_{C^0(S)}\cdot R^\epsilon < +\infty.   \]
	\end{Def}
	\begin{Rem} \label{Rem_Moduli_C^0 vs C^k decay}
		Although we only require $C^0$-polynomial decay in the definition above, by the proof of Lemma \ref{Lem_Moduli_Simple end <=> u to 0}, classical parabolic regularity theory could upgrade this to a polynomial decay in the $C^2_\star$-norm, where \[
		\|u\|_{C^2_\star, S, R} :=  \sup_{S\times [R, 2R]} \left( |u| + |\nabla u| + R|\partial_z u| + |\nabla^2 u| + \sqrt{R}|\partial_z \nabla u| + R|\partial^2_z u| \right) .   \]
	\end{Rem}
	\begin{Rem}
		It's interesting to see whether every simple translating end over a smooth closed self-shrinker $S$ is pl-simple. It is expected so if $S$ is \textbf{integrable}, i.e., every element in $\mathrm{Ker} L_S$ is induced by a continuous family of nearby self-shrinkers.   
	\end{Rem}
	
	\begin{eg} \label{Eg_Moduli_Bowl soliton}
		Recall by \cite{AltschulerWu94_Transl, ClutterbuckSchnurerSchulze07_TranslCatenoid}, the rotationally symmetric bowl soliton in $\RR^{n+1\geq 3}$ with a tip at the origin is given by $\Sigma := \{(x, F(x)): x\in \RR^n\}$, where \[
		F(x) = \frac{|x|^2}{2(n-1)} - \log |x| + O(|x|^{-1}),   \] 
		as $x\to \infty$.
		Let $(a, b)\in \RR^n\times \RR$, then $\Sigma_{a,b}:= \Sigma + (a, b)$ is also a translator in $\RR^{n+1}$, and when $z\gg 1$, $\Sigma_{a, b}\cap \RR_{>z} = \ES[U_{a, b}]\cap \RR_{> z}$, where 
		\begin{align*}
			U_{a, b} (x, z) = \begin{cases}
				O(z^{-1}\log z), &\ \text{ if } a = \orig, \\
				O(z^{-1/2}), &\ \text{ if } a \neq \orig.
			\end{cases}
		\end{align*}
		In particular, the translations of bowl solitons in $\RR^{n+1}$ are all pl-simple ends over $\SSp^{n-1}_{\sqrt{2(n-1)}}\times \RR_+$.
	\end{eg}
	
	The goal for this section is to study the space of pl-simple translating ends over $S\times \RR$, where $S\subset \RR^n$ is a fixed closed self-shrinker.  First recall (\ref{Pre_Gaussian Jac oper}), \[
	L_S:= \Delta_S - \frac{X}{2}\cdot \nabla_S + |A_S|^2 + \frac{1}{2}  \]
	is the Jacobi operator of Gaussian energy on $S$, where $X$ is the position vector. Let $\Gamma(S):= \{\mu_1<\mu_2<\dots\nearrow +\infty\}$ be the set of eigenvalues of $-L_S$; $E_j$ be the corresponding eigenspace of $\mu_j$. 
	Let \[
	\mu^- = \mu^-(S) := \sup \left(\Gamma(S)\cap (-\infty, 0) \right).   \]
	Note that by \cite{ColdingMinicozzi12_generic}, $-1, -1/2 \in \Gamma(S)$, hence $\mu_1\leq -1$ and $\mu^-\in [-1/2, 0)$.
	
	By Appendix \ref{Sec_Append_Aux end}, the EL operator of $\I$-functional at $\ES[u]$ is given by 
	\begin{align}
		\scT(u) = -z(\partial_z^2 + \partial_z + z^{-1}L_S)u + z\scR(u) + \scT(0), \label{Moduli_Transl mean curv oper}
	\end{align}
	where the estimates of error terms $\scR$ and $\scT(0)$ are given by Lemmas \ref{Append_scT(0)} and \ref{Append_Error est II of scT}.
	In particular, the hypersurface $\ES[u]$ is a translator if and only if $\scT(u) = 0$. 
	
	The first goal of this section is to prove the following. 
	(In the following theorem, $\BB^I$ is the unit Euclidean ball in $\RR^I$.)
	\begin{Thm} \label{Thm_Moduli of transl end }
		Given a closed self-shrinker $S\subset \RR^n$, let \[
		I:= \ind_\cF(-L_S) = \sum_{\mu_j\in \Gamma(S)\cap \RR_{<0}} \dim E_j.   \]
		Then there exists a constant $C_0(S)\gg 1$ such that for every $z_0\geq C_0(S)$, there's a continuous family $\{u_{\varphi, z_0}\}_{\varphi\in \BB^I}$ of smooth functions on $S\times \RR_{> z_0}$ such that for every $\varphi\in \BB^I$,
		\begin{enumerate} [(i)]
			\item $\scT(u_{\varphi, z_0}) = 0$ on $S\times \RR_{>z_0}$, and for every $R\geq z_0$, \[
			\|u_{\varphi, z_0}\|_{C^2_\star, S, R} \leq R^{\mu^-/2};  \]
			In particular, $\ES[u_{\varphi, z_0}]$ is a pl-simple translating end over $S\times \RR_+$;
			\item If $\varphi'\neq \varphi\in \BB^I$, then,
			\begin{align*}
				C(S, z_0)^{-1}\|\varphi-\varphi'\|_{\RR^I} \leq \sup_{R\geq z_0}\|u_{\varphi, z_0} - u_{\varphi', z_0}\|_{C^2_\star, S, R}\cdot R^{-\mu^-/2} \leq C(S, z_0)\|\varphi-\varphi'\|_{\RR^I}
			\end{align*}
			\[ \limsup_{z\to +\infty} z^{-\mu_1}\cdot \|(u_{\varphi, z_0} - u_{\varphi', z_0})(\cdot, z)\|_{C^0(S)} > 0.    \]
			\item If $\Sigma_e = \ES[w]$ is another pl-simple translating end over $S\times \RR_+$, then for every $\beta\in (0, 1)$, there exists $C_1(\Sigma_e, \beta)\geq C_0$ such that $\forall z_0\geq C_1$, there exists $\psi=\psi_{w, z_0}\in \BB^I$ such that \[
			\limsup_{R\to +\infty}\|(u_{\psi, z_0} - w)(\cdot, R)\|_{C^0(S)}\cdot e^{\beta R} <+\infty.   \]
		\end{enumerate}
		Moreover, if there's a Lie group $G\subset O(n)$ acting on $S$, then there exists a $G$-invariant pl-simple end over $S\times\RR_+$.
	\end{Thm}
	\begin{Rem} \label{Rem_Moduli_Asymp rate for simple ends}
		The asymptotic rate $\mu^-/2$ in Theorem \ref{Thm_Moduli of transl end } (i) \& (ii) can be replaced by any constant $\gamma\in (\mu^-, 0)$, while the constant $C_0(S)$ also depend on $\gamma$.
	\end{Rem}
	
	Sections \ref{Subsec_Moduli_L^2 est}-\ref{Subsec_Moduli_C^2,alpha est} are devoted to the analysis of linear operator $T_S:= \partial_z^2 + \partial_z + z^{-1}L_S$ on $S\times \RR_+$,  based on which a fixed point argument is carried out in Section \ref{Subsec_Moduli_Pf of Main Thm} to prove Theorem \ref{Thm_Moduli of transl end }.  In Section \ref{Subsec_Moduli_RR action}, we introduce a natural one-sided deformation action on the space of pl-simple translating end, which plays an important role in Section \ref{Sec_Pf of Main Thm} to prove the uniqueness of the translator associated with a generic end, as well as in Section \ref{Sec_Fattening} to construct examples such that the uniqueness fails.

	\subsection{$L^2$ estimates} \label{Subsec_Moduli_L^2 est}
	We shall study the $L^2$ analysis of the operator 
	\begin{align}
		T_S :=  \partial_z^2 + \partial_z + z^{-1}L_S,  \label{Moduli_Oper T_S}  
	\end{align}
	on $S\times \RR_+$, which by (\ref{Moduli_Transl mean curv oper}) is the principal part of the translator mean curvature.
	
	Define the Gaussian-$L^2$ norm for $\phi$ on $S$ to be \[
	\|\phi\|_{L^2(S)}^2 := \int_S \phi(x)^2 e^{-|x|^2/4}\ dx;   \] 
	For $\gamma\in \RR$ and $\Omega\subset S\times \RR_+$, define $L^2_\gamma(\Omega)$ be the space of locally $L^2$ functions on $\Omega$ with finite $\|\cdot\|_{L^2_\gamma, \Omega}$ norm, where
	\begin{align}
		\|u\|_{L^2_\gamma, \Omega}^2 := \int_\Omega u(x, z)^2 z^{-2\gamma-1}e^{-|x|^2/4}\ dxdz.
	\end{align}
	And for simplicity, $\|u\|_{L^2_\gamma, S, z_0}:= \|u\|_{L^2_{\gamma}, S\times \RR_{> z_0}}$.  Heuristically, $\|u\|_{L^2_\gamma, S, z_0}<+\infty$ means $|u|\lesssim z^\gamma$.
	
	Recall $\Gamma(S):= \{\mu_1<\mu_2< \dots \nearrow +\infty\}$ is the set of eigenvalues of the Jacobi operator $-L_S$ for Gaussian energy of $S$; $E_j$ is the Gaussian-$L^2$-eigenspace of $\mu_j$.  For $\gamma\in \RR$, we denote \[
	E_{<\gamma}:= \bigoplus_{\mu_j<\gamma} E_j,   \] 
	and $\Pi_{<\gamma}$ to be the $L^2(S)$-orthogonal projection onto $E_{<\gamma}$. Note that when $\gamma\leq \mu_1$, $E_{<\gamma} = \orig$.
	\begin{Lem} \label{Lem_Moduli_T_S u = f}
		For every $\Lambda>1$ and $\sigma>0$, there exists $C(\Lambda, \sigma)>1$ with the following properties.
		Suppose $S\subset \RR^n$ is a closed self-shrinker with $\mu_1(-L_S)\geq -\Lambda$; $\gamma\in \RR_{\geq -\Lambda}$ such that $dist_\RR(\gamma, \Gamma(S))\geq \sigma$; $z_0\geq C(\Lambda, \sigma)$.  Then for every $f\in L^2_{loc}(S\times \RR_{>z_0})$ with $\|f\|_{L^2_{\gamma-1}, S, z_0}<+\infty$, every $\phi\in L^2(S)$ and every $\varphi\in E_{<\gamma}$, there exists a unique solution $u\in L^2_\gamma(S\times \RR_{>z_0})$ to 
		\begin{align}
			\begin{cases}
				(\partial_z^2 + \partial_z + z^{-1}L_S)u = f,\ &\ \text{ on }S\times\RR_{>z_0}, \\
				u(\cdot, z_0) = \phi,\ \ \  \Pi_{<\gamma}\partial_zu(\cdot, z_0) = \varphi & \ \text{ on }S.
			\end{cases}  \label{Moduli_L^2 est_L^tau u = f}
		\end{align}
		and it satisfies, 
		\begin{align}
			\begin{split}
				&\ \|u\|_{L^2_\gamma, S, z_0} \\
				\leq &\ C(\Lambda, \sigma)\left(\|f\|_{L^2_{\gamma - 1}, S, z_0}+ \|\phi\|_{L^2(S)}z_0^{-\gamma-1/2} + \left(\|\varphi\|_{L^2(S)} + (2+ |\gamma|)\|\Pi_{<\gamma}\phi\|_{L^2(S)}\right)\cdot z_0^{-\gamma}\right). 
			\end{split} \label{Moduli_L^2 est_Sol L^tau u = f w est}
		\end{align}
	\end{Lem}
	
	To prove the Lemma, we first need to study $T_S$ restricted to each eigenspace $E_j$.
	For $\mu\in\RR$, let $\cL_\mu:= \partial_z^2 + \partial_z - \mu/z$ be a differential operator on $C^2(\RR_+)$.  An easy computation shows that $\cL_\mu$ is the EL operator of 
	\begin{align}
		\int_0^{+\infty} (\dot{u}^2 + \frac{\mu}{z}u^2)e^z\ dz = \int_0^{+\infty} |(u e^{z/2})'|^2 + (\frac{1}{4}+ \frac{\mu}{z})(u e^{z/2})^2\ dz.  \label{Moduli_L^2 est_Quad form of cL_mu}  
	\end{align}
	The equality follows from the integration by parts. Hence $-\cL_\mu$ is positive on $W^{1,2}_0(\RR_{>z_0})$ provided $z_0> (-4\mu)^+$. Here $a^+:= \sup\{a, 0\}$.
	
	\begin{Lem} \label{Lem_Moduli_Sol cL_mu = 0 w est}
		For every $\Lambda>1$, there exists $C(\Lambda)>4$ such that for every $|\mu|<\Lambda$,  there exists $b\in C^\infty(\RR_{\geq C(\Lambda)})$ solving $\cL_\mu b = 0$ on $\RR_{>C(\Lambda)}$ and satisfying \[
		\frac{1}{2} \leq 1-z^{-1/2} \leq \frac{b(z)}{z^\mu} \leq 1 + z^{-1/2} \leq 2, \ \ \ \text{ on }[C(\Lambda), +\infty).   \]
	\end{Lem}
	\begin{proof}
		First note that for every $\alpha\in \RR$, \[
		\cL_\mu z^\alpha = z^{\alpha-1}(\alpha - \mu + \frac{\alpha(\alpha-1)}{z}).  \]
		Take $\alpha = \mu$ and $\mu-1/2$ we get \[
		\cL_\mu (z^{\mu-1/2} \pm z^\mu) = z^{\mu-2}(-\frac{\sqrt{z}}{2} \pm \mu(\mu-1) + \frac{(2\mu-1)(2\mu-3)}{4\sqrt{z}}) \leq 0,   \]
		provided $z\geq C(\Lambda)$.  Hence one can solve $\cL_\mu b = 0$ with barriers $z^\mu \pm z^{\mu-1/2}$ on $[C(\Lambda)+4, +\infty)$ and get desired solution. 
	\end{proof}
	
	For $z_0>0$, $\alpha\in \RR$ and $u\in L^2_{loc}(\RR_+)$, let 
	\begin{align}
		\|u\|_{z_0; \alpha} := \left(\int_{z_0}^{+\infty} |u(s)|^2s^{-2\alpha - 1}\ ds \right)^{1/2}.  \label{Moduli_L^2 est_|u|_(z, alpha)}
	\end{align}
	
	\begin{Lem} \label{Lem_Moduli_Sol cL_mu = f w est}
		For every $\Lambda>1$, there exists $C(\Lambda)>0$ such that for every $\alpha\neq \mu-1$ and every \[
		z_0\geq \max\{C(\Lambda), \frac{|(2\alpha+2)(2\alpha +3)|}{|\alpha+1 - \mu|}\},  \]
		\begin{enumerate}[(i)]
			\item If $\mu > \alpha+1$ and $\mu>-\Lambda$, then for every $f\in L^2_{loc}(\RR_{>z_0})$ with $\|f\|_{z_0;\alpha}< +\infty$, there exists a unique solution $u$ to 
			\begin{align}
				\begin{cases}
					\cL_\mu u = f,\ &\ \text{ on }\RR_{>z_0}; \\
					u(z_0) = 0,\ &  \|u\|_{z_0, \alpha+1}<+\infty.
				\end{cases}  \label{Moduli_L^2 est_cL_mu = f}
			\end{align}
			Moreover, it satisfies the estimate 
			\begin{align}
				\|u\|_{z_0; \alpha+1} \leq \frac{2}{\mu - \alpha - 1}\|f\|_{z_0, \alpha}.  \label{Moduli_L^2 est_Sol cL_mu = f, mu>alpha+1}
			\end{align}
			\item If $\mu< \alpha+1$ and $|\mu|<\Lambda$, then for every $\|f\|_{z_0;\alpha}< +\infty$, the unique solution to $\cL_\mu u = f$ with initial data $u(z_0) =  0$, $u'(z_0) = c_0$ satisfies the estimate
			\begin{align}
				\|u\|_{z_0; \alpha+1}^2 \leq \frac{C(\Lambda)}{(\alpha +1-\mu)^2}\|f\|_{z_0, \alpha}^2 + \frac{C(\Lambda)}{\alpha + 1 -\mu}c_0^2z_0^{-2\alpha -2}.  \label{Moduli_L^2 est_Sol cL_mu = f, mu<alpha+1}
			\end{align}
		\end{enumerate}
	\end{Lem}
	We assert that the constant here does NOT depend on $\mu$. 
	\begin{proof}
		(i) Take $C(\Lambda)>4\Lambda$, the by (\ref{Moduli_L^2 est_Quad form of cL_mu}), $-\cL_\mu$ is a positive operator on $W_0^{1,2}(\RR_{\geq z_0})$.  Hence for every $R\gg z_0$, there's a unique solution $u_R$ to 
		\begin{align}
			\begin{cases}
				\cL_\mu u_R = f,\ &\ \text{ on }(z_0, R); \\
				u_R(z_0) = u_R(R) = 0.\ &  {}
			\end{cases}  \label{Moduli_L^2 est_cL_mu = f on (z_0, R)}
		\end{align}
		Thus, 
		\begin{align*}
			\int_{z_0}^R f(s)\cdot u_R(s)s^{-2\alpha -2}\ ds  
			= & \int_{z_0}^R u_R\cdot \cL_\mu u_R \cdot s^{-2\alpha -2}\ ds \\
			= & \int_{z_0}^R s^{-2\alpha-2}(-\dot{u}_R^2 + (\frac{u_R^2}{2})' - \frac{\mu}{s}u_R^2) + (\alpha+1)(u_R^2)'\cdot s^{-2\alpha -3}\ ds \\
			= & \int_{z_0}^R -\dot{u}_R^2\cdot s^{-2\alpha-2} + (\alpha+1-\mu + \frac{(\alpha+1)(2\alpha+3)}{s})u_R^2\cdot s^{-2\alpha -3}\ ds \\
			\leq & \int_{z_0}^R \frac{\alpha+1-\mu}{2}u_R^2\cdot s^{-2\alpha -3}\ ds, 
		\end{align*}
		where we use the choice of $z_0$ and $\mu>\alpha+1$ in the last inequality. Hence by H{\"o}lder's inequality, we have, \[
		\|u_R\|_{z_0, \alpha+1} \leq \frac{2}{\mu-\alpha-1}\|f\|_{z_0, \alpha}.   \] 
		Take $R\to \infty$, elliptic estimates guarantees that $u_R\to u$ in $W^{1,2}_{loc}$ and $u$ solves (\ref{Moduli_L^2 est_cL_mu = f}) with estimate (\ref{Moduli_L^2 est_Sol cL_mu = f, mu>alpha+1}).
		
		To see the uniqueness of the solution to (\ref{Moduli_L^2 est_cL_mu = f}), it suffices to show that when $f=0$, any solution to (\ref{Moduli_L^2 est_cL_mu = f}) must be $0$. This is followed by multiplying the equation with $u(z)z^{-2\alpha-2}\eta$ and doing the integration by parts as above, where $\eta\in C_c^\infty(\RR)$ is a cut-off function approximating $1$.
		
		(ii) Let $b\in C^\infty(\RR_{\geq z_0})$ be the solution of $\cL_\mu b = 0$ given by Lemma \ref{Lem_Moduli_Sol cL_mu = 0 w est}. Then, \[
		\cL_\mu u = (e^{-z}b^{-1})\partial_z(e^zb^2\partial_z)(b^{-1}u).     \]
		Thus if $u$ solves $\cL_\mu u = f$ on $\RR_{>z_0}$ with $u(z_0) = 0$, $u'(z_0) = c_0$, then,
		\begin{align*}
			|u(z)| & = \left|b(z)\cdot\left(\int_{z_0}^z f(s)b(s)e^s\ ds\int_s^z e^{-\tau}b(\tau)^{-2}\ d\tau \ +\ c_0e^{z_0}b(z_0)\int_{z_0}^z e^{-\tau}b(\tau)^{-2}\ d\tau \right)\right| \\
			& \leq C(\Lambda)z^\mu\cdot \left(\int_{z_0}^z |f(s)|s^{-\mu}\ ds \ +\ |c_0|z_0^{-\mu}\right)\,.
		\end{align*}
		
		Hence,
		\begin{align*}
			&\ \int_{z_0}^{+\infty} |u(z)|^2z^{-2\alpha-3}\ dz \\
			\leq &\  C(\Lambda)\int_{z_0}^{+\infty} z^{2(\mu-\alpha-1)-1}(\int_{z_0}^z |f(s)|s^{-\mu}\ ds + |c_0|z_0^{-\mu})^2\ dz \\
			\leq &\ C(\Lambda)\int_{z_0}^{+\infty} z^{2(\mu-\alpha-1)-1}\left((\int_{z_0}^z |f(s)|^2 s^{-\mu-\alpha}\ ds)(\int_{z_0}^z s^{-\mu+\alpha}\ ds) + c_0^2z_0^{-2\mu}\right)\ dz \\
			\leq &\ \frac{C(\Lambda)}{(\alpha+1-\mu)^2}\int_{z_0}^{+\infty} |f(s)|^2 s^{-2\alpha-1}\ ds + \frac{C(\Lambda)}{\alpha + 1 - \mu}c_0^2 z_0^{-2\alpha -2}.   
		\end{align*}
	\end{proof}
	
	\begin{Lem} \label{Lem_Moduli_Sol cL_mu = 0 fast decay}
		For every $\Lambda>0$, there exists $C(\Lambda)>1$ such that the following holds.
		Let $\mu\geq -\Lambda$, $z_0 \geq C(\Lambda)$, then there's a unique solution $w_{z_0}\in C^\infty(\RR_{\geq z_0})$ to 
		\begin{align*}
			\begin{cases}
				\cL_\mu w_{z_0} = 0,\ &\ \text{ on }\RR_{>z_0}, \\
				w_{z_0}(z_0) = 1,\ & \int_{z_0}^{+\infty} w_{z_0}(s)^2e^s\ ds <+\infty.
			\end{cases}
		\end{align*}
		Moreover, $w_{z_0}$ satisfies the estimate, 
		\begin{align*}
			\begin{split}
				0 < w_{z_0}(z) & \leq e^{z_0-z}\cdot \left(z/z_0\right)^{1+ (-\mu)^+},\ \ \ \ \ \forall z\geq z_0; \\
				|w_{z_0}'(z_0)| & \leq 2 + \mu^+ .
			\end{split}
		\end{align*}
	\end{Lem}
	\begin{proof}
		For the uniqueness, consider multiplying the equation by $w e^s\eta^2$ for some cut-off $\eta$ approximating $1$ and do the integration by parts. Use (\ref{Moduli_L^2 est_Quad form of cL_mu}) to conclude the uniqueness.
		
		For the existence and derivative estimate, first note that $\forall \alpha\in \RR$,
		\begin{align*}
			\cL_\mu (e^{-z}z^\alpha) = e^{-z}z^{\alpha-1}(-\alpha -\mu + \frac{\alpha(\alpha-1)}{z}).
		\end{align*}
		Hence when $\mu\leq 0$, let $a_\pm(z):= e^{-z+z_0}(z/z_0)^{-\mu \pm 1}$, we have for $z_0\geq 100(1+\Lambda^2)$, \[
		\pm \cL_\mu a_\pm \leq 0,   \]
		on $\RR_{\geq z_0}$.  That means $a_\pm$ serve as barriers for equation $\cL_\mu w = 0$, and thus there exists a solution $a_- \leq w_{z_0} \leq a_+$ on $\RR_{\geq z_0}$ with $w_{z_0}(z_0) = 1$ and $a_-'(z_0) \leq w'_{z_0}(z_0) \leq a_+'(z_0)$.
		
		When $\mu>0$, consider instead $a_+(z):= e^{-z + z_0}$ and $a_-(z):= e^{-z+z_0}(z/z_0)^{-\mu}$, which are still super- and sub-solutions of $\cL_\mu w = 0$. The same proof as above gives the desired decaying estimate and derivative estimate.
	\end{proof}
	
	\begin{proof}[Proof of Lemma \ref{Lem_Moduli_T_S u = f}.]
		Let $\psi_1, \psi_2, \psi_3, \dots$ be a family of $L^2(S)$-orthonormal eigenfunctions of $-L_S$, with corresponding eigenvalues $\bar{\mu}_1 < \bar{\mu}_2 \leq \bar{\mu}_3\leq \dots$. Let 
		\begin{align*}
			u_j(z) & := \int_S u(x, z)\psi_j(x)e^{-|x|^2/4}\ dx, &\ 
			f_j(z) & := \int_S f(x, z)\psi_j(x)e^{-|x|^2/4}\ dx, \\
			\phi_j & :=  \int_S \phi(x)\psi_j(x)e^{-|x|^2/4}\ dx, &\ 
			\varphi_j & :=  \int_S \varphi(x)\psi_j(x)e^{-|x|^2/4}\ dx .
		\end{align*} 
		Then equation (\ref{Moduli_L^2 est_L^tau u = f}) reduces to the following for each $j\geq 1$,
		\begin{align*}
			\begin{cases}
				\cL_{\bar{\mu}_j}u_j = f_j,\ &\ \text{ on }\RR_{>z_0}, \\
				u_j(z_0) = \phi_j;\ \ \  u_j'(z_0) = \varphi_j & \ \text{ provided }\bar{\mu}_j<\gamma.
			\end{cases}
		\end{align*}
		By Lemmas \ref{Lem_Moduli_Sol cL_mu = f w est} and \ref{Lem_Moduli_Sol cL_mu = 0 fast decay}, when $z_0$ is large, there's a unique solution $u_j$ to this equation, which satisfies 
		\begin{align*}
			\|u_j - \phi_j\cdot w_{z_0, j}\|_{z_0; \gamma}^2 \leq 
			C(\sigma, \Lambda)\left(\|f_j\|_{z_0; \gamma-1}^2 + \chi_j\cdot (\varphi_j - \phi_jw_{z_0, j}'(z_0))^2\cdot z_0^{-2\gamma}\right);  
		\end{align*}
		where $w_{z_0, j}$ is given by Lemma \ref{Lem_Moduli_Sol cL_mu = 0 fast decay}; $\chi_j = 1$ if $\bar{\mu}_j < \gamma$ and $\chi_j = 0$ if $\bar{\mu}_j > \gamma$.
		Also $|w_{z_0, j}'(z_0)|\leq 2+ \bar{\mu}_j^+$, and by the decay estimate on $w_{z_0, j}$, we have \[
		\|w_{z_0, j}\|_{z_0; \gamma}^2 \leq \int_{z_0}^\infty e^{2(z_0-z)}(z/z_0)^{2+ 2(-\bar{\mu}_j)^+}\cdot z^{-2\gamma -1}\ dz \leq C(\Lambda)z_0^{-2\gamma - 1}.   \]
		These altogether gives desired estimates on $\|u_j\|_{z_0; \gamma}$, and taking sum of which over $j$ proves Lemma \ref{Lem_Moduli_T_S u = f}.
	\end{proof}

	We finish this subsection by pointing out the following quantization of the asymptotic rate for solutions to $T_S u = 0$.
	\begin{Cor} \label{Cor_Moduli_Sharp asymp for T_S u = 0}
		Let $S, z_0$ be the same as in Lemma \ref{Lem_Moduli_T_S u = f}.  Let $u\in L^2_{loc}(S\times \RR_{\geq z_0-1})$ be the solution of $(\partial_z^2 + \partial_z + z^{-1}L_S) u = 0$ such that the asymptotic rate at infinity satisfies, \[
		\cA\cR_\infty(u) := \inf\{\gamma\in \RR: \|u\|_{L^2_\gamma, S, z_0}< +\infty\} \in (-\infty, +\infty).   \]
		Here, we use the convention that $\inf \emptyset = +\infty$ and $\inf \RR = -\infty$.  Then there exists $l\geq 1$ such that $\cA\cR_\infty(u) = \mu_l\in \Gamma(S)$, and there exist an eigenfunction $\orig \neq \psi\in E_l$ and $\epsilon>0$ such that \[
		\|u-z^{\mu_l}\psi(x)\|_{L^2_{\mu_l-\epsilon}, S, z_0} < +\infty \,.   \]
	\end{Cor}
	\begin{proof}
		First note that for every $\gamma<\gamma'\in \RR$ with $[\gamma, \gamma']\cap \Gamma(S) = \emptyset$, by Lemma \ref{Lem_Moduli_T_S u = f} we know that 
		\begin{align*}
			&\ \{v\in L^2_\gamma(S\times \RR_{\geq z_0}): T_S v = 0, v(\cdot, z_0)=u(\cdot, z_0)\} \\
			\subset &\ \{v\in L^2_{\gamma'}(S\times \RR_{\geq z_0}): T_S v = 0, v(\cdot, z_0)=u(\cdot, z_0)\}     
		\end{align*}
		are both finite dimensional affine subspace of $L^2_{loc}(S\times \RR_{\geq z_0})$ of the same dimension.  Hence, they are the same.  This immediately implies $\cA\cR_\infty(u) = \mu_l\in \Gamma(S)$ for some $l\geq 1$.
		
		To prove the decay estimate, notice that for every $j\geq 1$ and every $\psi\in E_j$, $u_j(z):= \langle u(\cdot, z), \psi\rangle_{L^2(S)}$ satisfies $\cL_{\mu_j}u_j = 0$ on $\RR_{\geq z_0}$.  The space of all solutions to this ODE is a two-dimensional vector space generated by $b_j$ in Lemma \ref{Lem_Moduli_Sol cL_mu = 0 w est} and $w_{z_0, j}$ in Lemma \ref{Lem_Moduli_Sol cL_mu = 0 fast decay}.  That means there exist some $c_j(\psi)\in \RR$ such that $|u_j - c_jb_j| = O(e^{-z/2})$.  And since $b_j\sim z^{\mu_j}$, by definition of $\cA\cR_\infty(u)$, we must have $c_j = 0$ if $\mu_j>\cA\cR_\infty(u)$ and $c_l(\psi) \neq 0$ for some $\orig \neq\psi\in E_l$.  Then the desired estimate follows from Lemmas \ref{Lem_Moduli_Sol cL_mu = 0 w est} and \ref{Lem_Moduli_Sol cL_mu = 0 fast decay}.
	\end{proof}

	\subsection{$C^0$ estimate} \label{Subsec_Moduli_C^0 est}
	Consider the operator of the following general form in this section. For $u\in S\times (a_1, a_2)$, let 
	\begin{align}
		\scL u := -\text{div}_S \left(B_1[u]\right) - \partial_z(B_2[u]) - z\partial_z u + B_3[u],  \label{Moduli_C^0 est_general oper div form}
	\end{align}
	where $B_j[u]:= B_j(x, z, u, \nabla_S u, z\partial_z u)$, and $B_j(x, z, p, \xi, \eta)$ satisfies the following,
	\begin{align}
		\begin{split}
			\xi\cdot B_1 + z^{-1}\eta\cdot B_2 & \geq \lambda(|\xi|^2 + z^{-1}\eta^2) -\Lambda p^2, \\
			|B_1| + z^{-1/2}|B_2| + |B_3| & \leq \Lambda(|\xi|+z^{-1/2}|\eta|+ |p|),
		\end{split} \label{Moduli_C^0 est_Coef condition}
	\end{align}
	for some constant $0<\lambda<\Lambda<+\infty$.  
	
	Note that $-zT_S = -(z\partial_z^2 + z\partial_z + L_S)$ satisfies the conditions (\ref{Moduli_C^0 est_Coef condition}) on the coefficients. By Lemma \ref{Append_Error est III of scT}, for $u_\pm$ satisfying \[
	|u_\pm| + |\nabla u_\pm| + |z\partial_z u_\pm| + |\nabla^2 u_\pm| + |\sqrt{z}\nabla u_\pm| + |z\partial_z^2 u_\pm| \leq \delta,   \]
	the operator $\tilde{\scT}(u_+-u_-):= (1+b)^{-1}(\scT(u_+)-\scT(u_-))$ ($b$ is a small function specified in Lemma \ref{Append_Error est III of scT}) also satisfies (\ref{Moduli_C^0 est_Coef condition}), provided $\delta\ll1$ and $z\gg 1$, where the constants $\lambda$, $\Lambda$ both depend only on $S$.  Therefore, the following several Lemmas apply to both cases.
	
	\begin{Lem} \label{Lem_Moduli_C^0 est}
		Let $u$ be the solution of $\scL u = f$ on $S\times (R, 4R)$, $R\geq 1$, where $\scL$ satisfies (\ref{Moduli_C^0 est_general oper div form}) and (\ref{Moduli_C^0 est_Coef condition}).  Then we have\[
		\|u\|_{L^\infty, S\times (2R, 3R)}\leq C(S, \lambda, \Lambda)\left( R^{-1/2}\|u\|_{L^2, S\times (R, 4R)} + \|f\|_{L^\infty, S\times (R, 4R)} \right) .   \]
	\end{Lem}
	\begin{proof}
		It suffices to prove the upper bound for $u$; then the lower bound follows from the upper bound for $-u$. Consider $u_R(x, \hat{z}):= u(x, R\hat{z})$ defined on $S\times (1,4)$ solving \[
		-\text{div}_S(\tilde{B}_{1,R}[u_R]) - \frac{1}{R}\partial_{\hat{z}}(\tilde{B}_{2, R}[u_R]) - \hat{z}\partial_{\hat{z}} u_R + \tilde{B}_{3, R}[u_R] = f_R,   \]
		where $\tilde{B}_{j, R}[u_R](x, \hat{z})=\tilde{B}_{j, R}(x, \hat{z}, u_R, \nabla u_R, \hat{z}\partial_{\hat{z}} u_R) := B_j(x, z, u, \nabla u, z\partial_z u)|_{z=R\hat{z}}$, and by (\ref{Moduli_C^0 est_Coef condition}), $\tilde{B}_{j,R}(x, \hat{z}, p, \xi, \eta)$ satisfies the estimates, 
		\begin{align*}
			\begin{split}
				\xi\cdot \tilde{B}_{1,R} + R^{-1}\hat{z}^{-1}\eta\cdot \tilde{B}_{2,R} & \geq \frac{\lambda}{4}(|\xi|^2 + R^{-1}\eta^2) - \Lambda p^2, \\
				|\tilde{B}_{1,R}| + R^{-1/2}|\tilde{B}_{2,R}| + |\tilde{B}_{3,R}| & \leq 4\Lambda(|\xi|+R^{-1/2}|\eta|+ |p|).
			\end{split} 
		\end{align*}
		
		Let $p\geq 1$, $\phi\in C^\infty_c(1,4)$ be a non-negative cut-off function; $k:=\|f\|_{L^\infty, S\times (R, 4R)}$, $\bar{u}:= \sup\{u, 0\} +k$. Multiplying the equation by $(\bar{u}^p - k^p) \cdot \phi^2$ and integration by part provide
		\begin{align*}
			\int_{S\times (1,4)} |\nabla_S (\bar{u}^{(p+1)/2} \phi)|^2 & \leq \int_{S\times (1,4)} C(\lambda, \Lambda)(1+\beta^2)(\phi^2 + \dot{\phi}^2)\bar{u}^p(\bar{u}+|f_R|) \\
			& \leq \int_{S\times (1,4)} C(\lambda, \Lambda)(1+\beta^2)(\phi^2 + \dot{\phi}^2)\bar{u}^{p+1} .
		\end{align*}
		Then we apply the Sobolev inequality on $S$ and run the Moser iteration \cite{GilbargTrudinger01} to get $C^0$ estimate \[
		\sup_{S\times (2,3)} u_R \leq C(S, \lambda, \Lambda)\left(\|u_R\|_{L^2, S\times (1,4)} + \|f_R\|_{L^\infty, S\times (1,4)} \right) .    \]
		Then scaling back gives the upper bound for $u$.
	\end{proof}
	\begin{Rem} \label{Rem_Moduli_C^0 est, bdy case}
		A similar argument shows that if $u(\cdot, R) = 0$ and $u$ solves $\scL u = f$ on $S\times [R, 4R)$, $R\geq 1$ with $\scL$ satisfying (\ref{Moduli_C^0 est_general oper div form}) and (\ref{Moduli_C^0 est_Coef condition}), then we have the boundary estimate \[
		\|u\|_{L^\infty, S\times [R, 3R)}\leq C(S, \lambda, \Lambda)\left( R^{-1/2}\|u\|_{L^2, S\times (R, 4R)} + \|f\|_{L^\infty, S\times (R, 4R)} \right) .   \] 
	\end{Rem}
	
	We shall also need the following Harnack inequality for $\scL$: 
	\begin{Lem} \label{Lem_Moduli_Harnack Ineq}
		Let $u>0$ be a positive solution to $\scL u = 0$ on $S\times (R, 8R)$, $R\geq 1$, where $\scL$ satisfies (\ref{Moduli_C^0 est_general oper div form}) and (\ref{Moduli_C^0 est_Coef condition}). Then \[
		\sup_{S\times [5R, 6R]} u \leq C(S, \lambda, \Lambda) \inf_{S\times [2R, 3R]} u .   \]
	\end{Lem}
	\begin{proof}
		Let $U_R(\tau, x, \hat{z}):= u(x, R\tau + \sqrt{R}\hat{z})$, defined on $(1, 7)\times S\times (0, 4)$, satisfying 
		\begin{align}
			-\text{div}_S(\cB_{1,R}[U_R]) - \frac{1}{\sqrt{R}}\partial_{\hat{z}}(\cB_{2,R}[U_R]) - (\tau+\frac{\hat{z}}{\sqrt{R}})\cdot \partial_\tau U_R + \cB_{3, R}[U_R] = 0,  \label{Moduli_C^0 est_Parab equ U_R}   
		\end{align}
		where $\cB_{j, R}[U_R] (\tau, x, \hat{z}) = \cB_{j, R}(\tau,x, \hat{z}, U_R, \nabla U_R, \partial_{\hat{z}}U_R)$ satisfies,
		\begin{align*}
			\begin{split}
				\cB_{j, R}(\tau, x, \hat{z}, p, \xi, \hat{\eta}) & := B_j(x, R\tau+\sqrt{R}\hat{z}, p, \xi, (\sqrt{R}\tau + \hat{z})\hat{\eta}), \\
				\xi\cdot \cB_{1, R} + \hat{\eta}\cdot \frac{1}{\sqrt{R}}\cB_{2, R} & \geq \lambda(|\xi|^2 + |\hat{\eta}|^2) -\Lambda |p|^2, \\
				|\cB_{1, R}| + \frac{1}{\sqrt{R}}|\cB_{2, R}| + |\cB_{3, R}| &\leq 16\Lambda(|\xi|+ |\hat{\eta}| + |p|).
			\end{split}
		\end{align*}
		This indicates that (\ref{Moduli_C^0 est_Parab equ U_R}) is a uniform parabolic equation in $-\tau$ direction. Hence by \cite{AronsonSerrin67_QuasilinearParabEqu}, $U_R$ satisfies the parabolic Harnack inequality, \[
		\sup_{(4, 6)\times S\times (1, 3)} U_R \leq C(S, \lambda, \Lambda) \inf_{(1,3)\times S\times (1,3)} U_R.   \]
		This directly implies the Harnack inequality of $u$.
	\end{proof}

	\subsection{$C^{2,\alpha}$ estimate} \label{Subsec_Moduli_C^2,alpha est}
	We first derive a long-term $C^{2,\alpha}$ estimate, which is governed by almost parabolicity of $T_S$ when $z$ is large.  For $R>1$ and $\alpha\in (0, 1)$, we define the $C^{k,\alpha}_{\star}$ norm for functions on $S\times [R, 2R]$ by,
	\begin{align}
		\begin{split}
			[f]_{\alpha; S, R}^\star := &\ \sup\left\{\frac{|f(x, z)- f(x', z')|}{|x- x'|^\alpha + R^{-\alpha/2}|z-z'|^\alpha}: (x, z), (x', z')\in S\times [R, 2R]\right\}; \\
			\|f\|_{C^\alpha_\star; S, R} := &\ \sup_{S\times [R, 2R]} |f| + [f]_{\alpha; S, R}^\star; \\
			\|u\|_{C^{2,\alpha}_\star, S, R}:= &\ \Big(\|u\|_{C^\alpha_\star; S, R} + \|\nabla u\|_{C^\alpha_\star; S, R} + R\|\partial_z u\|_{C^\alpha_\star; S, R} \\
			&\ + \|\nabla^2 u\|_{C^\alpha_\star; S, R} + \sqrt{R}\|\partial_z \nabla u\|_{C^\alpha_\star; S, R} + R\|\partial_z^2 u\|_{C^\alpha_\star; S, R} \Big).
		\end{split} \label{Moduli_C^k,alpha_star norms}
	\end{align}
	This norm is chosen to be compatible with the spacetime $C^{2,\alpha}$ norm of a transformation of $u$ (specified in the proof of Lemma \ref{Lem_Moduli_C^2,alpha_star est, interior}) that satisfies a parabolic equation with uniform bounds on the coefficients, which enables us to prove the long-term estimate.
	
	Since $S$ is a fixed self-shrinker, we may omit the subscript $S$ for simplicity.

	\begin{Lem} \label{Lem_Moduli_C^2,alpha_star est, interior}
		Let $u$ be the solution of $(\partial_z^2 + \partial_z + z^{-1}L_S)u = f$ on $S\times (R, 8R)$, $R>16$.  Then 
		\begin{align*}   
			\|u\|_{C^{2,\alpha}_\star, S, 3R} \leq  C(S, \alpha)\left(\|u\|_{C^0, S\times (R,8R)} + \sup_{R<\rho< 4R}\rho\|f\|_{C^\alpha_\star, S, \rho} \right). 
		\end{align*}
	\end{Lem}
	\begin{proof}
		Let $u_R(\tau, x, \hat{z}):= u(x, R\tau + \sqrt{R}\hat{z})$ be defined on $\Omega:= (1,7)\times S \times (0, 4)$.  Then $u_R$ solves \[
		\partial_\tau u_R + \partial_{\hat{z}}^2 u_R + \frac{1}{\tau + \hat{z}/\sqrt{R}}L_S u_R = R f_R,   \]
		on $\Omega$.  Let $\Omega':= (1, 6)\times S\times (1, 3)$. By interior parabolic Schauder estimate \cite{Knerr80_SpacialParabSchauder} we have, 
		\begin{align}
			\|u_R\|_{2+\alpha, \Omega'}^* \leq C(S)(\|u_R\|_{C^0, \Omega} + \|R f_R\|^*_{\alpha, \Omega}),  \label{Moduli_Parab spacial Schauder est}   
		\end{align}
		where the H\"older norm $\|\cdot\|^*_{\alpha, \Omega'}$ is taking difference only in space, i.e.,
		\begin{align*}
			\|v\|^*_{\alpha, \Omega'} & := \|v\|_{C^0, \Omega'} + \sup_{(\tau, x, \hat{z})\neq (\tau, x', \hat{z}') \in \Omega'} \frac{|v(\tau, x, \hat{z}) - v(\tau, x', \hat{z}')|}{|x-x'|^{\alpha} + |z - z'|^\alpha} ; \\
			\|v\|^*_{2+\alpha, \Omega'} & := \|v\|^*_{\alpha, \Omega'} + \|\nabla_{S\times \RR} v\|^*_{\alpha, \Omega'} + \|\nabla^2_{S\times \RR} v\|^*_{\alpha, \Omega'} + \|\partial_\tau v\|^*_{\alpha, \Omega'}.
		\end{align*}
		It's easy to check that,
		\begin{align*}
			\|u_R\|^*_{2+\alpha, \Omega'} \geq C'(S)\|u\|_{C^{2,\alpha}_\star, S, 3R}\,, &\ & \|f_R\|^*_{\alpha, \Omega} \leq C''(S)\sup_{\rho\in(R,4R)}\|f\|_{C^\alpha_\star, S,\rho}.
		\end{align*}
		Therefore, the Lemma follows immediately from (\ref{Moduli_Parab spacial Schauder est}).
	\end{proof}
	
	We shall also estimate the following short-term $C^{2,\alpha}$ norm, governed by the ellipticity of the equation.
	\begin{align}
		\begin{split}
			\|u\|_{C^d_\sharp, S, R} & := \sum_{0\leq k+l\leq d}\sup_{S\times [R, R+1]} R^{-l/2}|\partial_z^k \nabla^l_S u|; \\
			\|u\|_{C^{d, \alpha}_\sharp, S, R} & := \|u\|_{C^d_\sharp, S, R}\ + \\ &\    \sup_{\substack{(x, z)\neq (x', z')\in S\times [R, R+1]\\ |x-x'|\leq R^{-1/2}}} \sum_{0\leq k+l \leq d} \frac{R^{-l/2}|\partial_z^k\nabla^l_S u(x, z) - \partial_z^k\nabla^l_S u(x', z')|}{R^{\alpha/2}|x-x'|^\alpha + |z-z'|^\alpha}.
		\end{split} \label{Moduli_C^k,alpha_sharp norms}
	\end{align}
	It's not hard to check that for $R\geq 1$ and $\alpha\in (0,1)$, we have
	\begin{align}
		\|u\|_{C^{2, \alpha}_\sharp, S, R} \leq \|u\|_{C^{2,\alpha}_\star, S, R}.   \label{Moduli_C^k_sharp < C^k_star}
	\end{align}
	
	We have the short-term Schauder estimates,
	\begin{Lem} \label{Lem_Moduli_short term Schauder est, int}
		Let $u$ be the solution of $(\partial_z^2 + \partial_z + z^{-1}L_S)u = f$ on $S\times (R-1, R+2)$, $R\geq 2$.  Then \[
		\|u\|_{C^{2,\alpha}_\sharp, S, R}\leq  C(S, \alpha)\left(\|u\|_{C^0, S\times (R-1, R+2)} + \sup_{\rho\in (R-1, R+1)}\|f\|_{C^\alpha_\sharp, S, \rho}\right).   \]
	\end{Lem}
	\begin{Lem} \label{Lem_Moduli_short term Schauder est, bdy}
		Let $u$ be the solution of $(\partial_z^2 + \partial_z + z^{-1}L_S)u = f$ on $S\times [R, R+2)$, $R\geq 2$. Suppose $u(\cdot, R) \equiv 0$.  Then \[
		\|u\|_{C^{2,\alpha}_\sharp, S, R}\leq  C(S, \alpha)\left(\|u\|_{C^0, S\times [R, R+2)} + \sup_{\rho\in [R, R+1)}\|f\|_{C^\alpha_\sharp, S, \rho}\right).   \]
	\end{Lem}
	
	\begin{proof}[Proof of Lemmas \ref{Lem_Moduli_short term Schauder est, int} and \ref{Lem_Moduli_short term Schauder est, bdy}]
		Apply classical interior and boundary elliptic estimates \cite{GilbargTrudinger01} to the function $\hat{u}_R(x, z):= u(x/\sqrt{R}, z+R)$ that is defined on $(x, z)\in \sqrt{R}\cdot S \times (-1,2)$ or $\sqrt{R}\cdot S \times [0,2)$. 
	\end{proof}
	
	\begin{Rem} \label{Rem_Moduli_No long term bdy C^k est}
		There's no analogous boundary estimate as Lemma \ref{Lem_Moduli_C^2,alpha_star est, interior}.  This is because such long term estimate in Lemma \ref{Lem_Moduli_C^2,alpha_star est, interior} originates from the almost parabolicity of $\partial_z^2 + \partial_z + z^{-1}L_S$ when $z$ is large. But $S\times \{R\}$ is the terminal time slice for the parabolic equation on $S\times (R, 8R]$, and hence the value of the solution can not be arbitrarily prescribed.
	\end{Rem}
	
	\begin{Rem} \label{Rem_Moduli_Schauder est work for general oper}
		By Lemma \ref{Append_Error est II of scT} and (\ref{Moduli_C^k_sharp < C^k_star}), the same proof above also gives long-term and short-term Schauder estimates of the same form for the operator $\tilde{\scT}(u_+-u_-):= -z^{-1}\cdot(\scT(u_+)-\scT(u_-))$, provided $\|u_\pm\|_{C^{2,\alpha, S, R}}\leq \delta_S\ll1$ and $R\geq z_0(S)\gg 1$.
	\end{Rem}

	\subsection{Proof of Theorem \ref{Thm_Moduli of transl end }} \label{Subsec_Moduli_Pf of Main Thm}
	For $\alpha\in (0, 1)$, $\gamma\in \RR$ and $z_0\geq 1$, we define $\mathfrak{X}_{\gamma}^\alpha(S, z_0)$ to be the space of $C^{2,\alpha}_{loc}$ functions $u$ on $S\times \RR_{\geq z_0/12}$ such that
	\begin{align*}
		\|u\|_{\mathfrak{X}_\gamma^\alpha, S, z_0} := \sup_{R\geq z_0/2} R^{-\gamma}\|u\|_{C^{2,\alpha}_\star, S, R} \leq 1,
	\end{align*}
	where $\|\cdot\|_{C^{2,\alpha}_\star}$ norm is defined in (\ref{Moduli_C^k,alpha_star norms}).
	
	By Lemmas \ref{Append_scT(0)} and \ref{Append_Error est II of scT}, the translating equation (\ref{Pre_Transl equ}) for $\ES[u]$ is equivalent to \[
	\scT(u) =: -zT_S u + z\scR(u) + \scT(0) = 0;   \]
	where $T_S := \partial_z^2 + \partial_z + z^{-1}L_S$, $\scR(0) = 0$ and there exists some geometric constant $\delta_S\in (0, 1)$ such that if $\|v_\pm\|_{C^2_\star, R} \leq \delta_S$. Then
	\begin{align}
		\begin{split}
			\|\scT(0)\|_{C^\alpha_\star, R} & \leq C(S, \alpha)R^{-1}; \\
			\|\scR(v_+) - \scR(v_-)\|_{C^\alpha_\star, R} & \leq C(S, \alpha)R^{-1}\left(R^{-1}+\|v_\pm\|_{C^{2,\alpha}_\star, R}\right) \cdot \|v_+ - v_-\|_{C^{2,\alpha}_\star, R}\,.
		\end{split} \label{Moduli_C^alpha_star error est scR^+-scR^-}
	\end{align}
	
	Throughout this and next subsections, let $\eta\in C^\infty(\RR)$ be a non-decreasing cut off function such that $\eta(t) = 0$ for $t\leq 1/2$, $\eta(t) = 1$ for $t\geq 3/4$ and $|\eta'|\leq 8$. For $a>0$, let $\eta_a(t):= \eta(t/a)$. 
	Recall $E_{<0}$, $\Pi_{<0}$ are introduced at the beginning of Section \ref{Subsec_Moduli_L^2 est}.
	\begin{Lem} \label{Lem_Moduli_contraction map}
		For every $\alpha\in (0,1)$, $\gamma \in (\mu^-, 0)$, there exists $C(S, \gamma, \alpha)>1$ and $\delta_0(S, \gamma, \alpha)\in (0, 1)$ such that for every $z_0\geq C(S, \gamma, \alpha)$ and every $\varphi\in E_{<0}$ with 
		\begin{align}
			\|\varphi\|_{L^2(S)}\leq \delta_0(S, \gamma, \alpha)\cdot z_0^{\gamma},  \label{Moduli_smallness of u_z(z_0)}
		\end{align}
		There's a unique solution $u_\varphi \in \mathfrak{X}^\alpha_\gamma(S, z_0)$ to the equation,
		\begin{align}
			\begin{cases}
				(\partial_z^2 + \partial_z + z^{-1}L_S )u = \eta_{z_0}(z)\cdot\left(\scR(u) + z^{-1}\scT(0)\right), \ &\ \text{ on }S\times \RR_{>z_0/12};\\
				u(\cdot, z_0/12) = 0, \ \ \  \Pi_{<0}\partial_z u(\cdot , z_0/12) = \varphi, \ &\ \text{ on }S.
			\end{cases} \label{Moduli_Transl equ w cut-off and prescib bdy deriv}
		\end{align}
		Furthermore, it has the estimate
		\begin{align}
			\|u_\varphi\|_{\mathfrak{X}^\alpha_\gamma, S, z_0}\leq C(S,\alpha, \gamma)\left( z_0^{\gamma} + \|\varphi\|_{L^2(S)}z_0^{-\gamma} \right)\leq \frac{1}{2}\,. \label{Moduli_Sol |u_varphi|_mathfr < 1/2}
		\end{align}
		In particular, such $u_\varphi$ satisfies $\scT(u_\varphi) = 0$ on $S\times \RR_{> z_0}$.

		Moreover, if $\varphi_+ \neq \varphi_- \in E_{<0}$ both satisfy (\ref{Moduli_smallness of u_z(z_0)}), then for every $\bar{z}\geq z_0$,
		\begin{align}
			&\ C(S,\alpha,\gamma, \bar{z})^{-1} \leq \frac{\sup_{R\geq \bar{z}}\|u_{\varphi_+}-u_{\varphi_-}\|_{C^2_\star, S, R}\cdot R^{-\gamma}}{\|\varphi_+-\varphi_-\|_{L^2(S)} } \leq C(S,\alpha,\gamma, \bar{z}) ; \label{Moduli_varphi <-> u_varphi biLip} \\ 
			&\ \limsup_{R\to +\infty}\|(u_{\varphi_+} - u_{\varphi_-})(\cdot, R)\|_{C^0, S}\cdot R^{-\mu_1} >0. \label{Moduli_L^2_mu_1 Diff of u_varphi w different initial}
		\end{align}
	\end{Lem}
	
	The reason that we add a cut-off term in the equation (\ref{Moduli_Transl equ w cut-off and prescib bdy deriv}) is that we don't have a good long-term boundary estimate, see Remark \ref{Rem_Moduli_No long term bdy C^k est}.
	
	\begin{proof}
		\noindent  \textbf{Step 1.} Consider the map from $\mathfrak{X}_\gamma^\alpha(S, z_0)$, $\scU_\varphi: v \mapsto u:= \scU_\varphi (v)$ solving 
		\begin{align*}
			\begin{cases}
				(\partial_z^2 + \partial_z + z^{-1}L_S )u = \eta_{z_0}(z)\cdot\left(\scR(v) + z^{-1}\scT(0)\right), \ &\ \text{ on }S\times \RR_{>z_0/12};\\
				u(\cdot, z_0/12) = 0, \ \ \  \Pi_{<0}\partial_z u(\cdot , z_0/12) = \varphi, \ &\ \text{ on }S.
			\end{cases}
		\end{align*}
		We first verify that $\scU_\varphi$ maps $\mathfrak{X}_\gamma^\alpha(S, z_0)$ to itself, provided $z_0\gg 1$ and $\delta_0\ll1$.  For $v\in \mathfrak{X}^\alpha_\gamma(S, z_0)$, by (\ref{Moduli_C^alpha_star error est scR^+-scR^-}) and definition of $\mathfrak{X}^\alpha_\gamma(S, z_0)$, for every $R\geq z_0/2$,
		\begin{align*}
			\|\scR (v)\|_{C_\star^\alpha, R} + \|z^{-1}\scT(0)\|_{C_\star^\alpha, R} \leq C(S, \alpha)R^{-1+2\gamma}.
		\end{align*}
		Hence by Lemma \ref{Lem_Moduli_T_S u = f}, $u:= \scU_\varphi (v)$ satisfies \[
		\|u\|_{L^2_\gamma; S,z_0/12} \leq C(S, \gamma)(\|\eta_{z_0}\cdot\left(\scR(v) + z^{-1}\scT(0)\right)\|_{L^2_{\gamma -1}; S, z_0/12} + \|\varphi\|_{L^2(S)}z_0^{-\gamma}) \leq C(S, \alpha, \gamma)(z_0^\gamma + \delta_0);   \]
		By Lemma \ref{Lem_Moduli_C^0 est}, for every $R\geq z_0/6$, 
		\begin{align*}  
			\|u\|_{C^0, S, R} & \leq C(S, \gamma)(R^{-1/2}\|u\|_{L^2, S\times [R/2, 4R]} + R\|\eta_{z_0}\cdot\left(\scR(v) + z^{-1}\scT(0)\right)\|_{C^0, S\times[R/2, 4R]} ) \\
			& \leq C(S, \alpha, \gamma)(z_0^\gamma + \delta_0)R^\gamma. 
		\end{align*}
		By Lemma \ref{Lem_Moduli_C^2,alpha_star est, interior},  for every $R\geq z_0/2$,
		\begin{align*}
			\|u\|_{C^{2,\alpha}_\star, S, R} & \leq C(S, \alpha, \gamma)\left(\|u\|_{C^0, S\times [R/3, 3R]} + \sup_{R/3 \leq \rho\leq 4R/3} \rho\|\eta_{z_0}\cdot\left(\scR(v) + z^{-1}\scT(0)\right)\|_{C^\alpha_\star, \rho} \right) \\
			& \leq C(S, \alpha, \gamma)(z_0^\gamma + \delta_0)R^\gamma.
		\end{align*}
		In other words, \[
		\|u\|_{\mathfrak{X}_\gamma^\alpha, S, z_0} \leq C(S, \alpha, \gamma)(z_0^\gamma + \delta_0).   \]
		Hence by taking $z_0\gg 1$ and $\delta_0\ll1$, we have $\|\scU_\varphi (v)\|_{\mathfrak{X}_\gamma^\alpha, S, z_0} <1$, i.e. $\scU_\varphi$ maps $\mathfrak{X}^\alpha_\gamma(S, z_0)$ to itself. \\
		\textbf{Step 2.} By a similar approach as above, we see that for every $v_\pm \in \mathfrak{X}^\alpha_\gamma(S, z_0)$ and every $\varphi_\pm \in E_{<0}$ satisfying (\ref{Moduli_smallness of u_z(z_0)}), we have,
		\begin{align}
			\begin{split}
				&\ \|\scU_{\varphi_+}(v_+) - \scU_{\varphi_-}(v_-)\|_{\mathfrak{X}^\alpha_\gamma, S, z_0} \\
				\leq &\ C(S, \alpha, \gamma)\left( z_0^\gamma\cdot \|v_+ - v_-\|_{\mathfrak{X}^\alpha_\gamma, S, z_0} + z_0^{-\gamma}\|\varphi_+ - \varphi_-\|_{L^2(S)} \right).  
			\end{split}  \label{Moduli_scU_varphi(v)|^+_- contraction}
		\end{align}
		Hence for fixed $\varphi$, when $z_0\gg 1$,  $\scU_\varphi$ contracts $\|\cdot\|_{\mathfrak{X}_\gamma^\alpha, S, z_0}$ norm. Define the iterating sequence $u^{(l)}:= \scU_\varphi{ }^l(\orig)$, $\{u^{(l)}\}_{l\geq 1}$ is a Cauchy sequence in $\|\cdot\|_{\mathfrak{X}^\alpha_\gamma, S, z_0}$ norm.  Thus when $l\to \infty$, $u^{(l)}$ locally $C^{2,\alpha}_\star$-converges to some $u_\varphi$ on $S\times \RR_{\geq z_0/2}$, which solves the equation, \[
		(\partial_z^2 + \partial_z + z^{-1}L_S )u_\varphi = \eta_{z_0}(z)\cdot\left(\scR(u_\varphi) + z^{-1}\scT(0)\right)   \]
		on $S\times \RR_{\geq z_0/2}$.  
		By the classical elliptic interior and boundary estimates \cite{GilbargTrudinger01}, \[
		\|u^{(l)}\|_{C^{2,\alpha}_\star, S\times [z_0/12, 3z_0/4]} \leq C(S, \gamma, \alpha, z_0)(\|u^{(l)}\|_{L^2(S\times [z_0/12, z_0])} + \|u^{(l-1)}\|_{C^{2,\alpha}_\star, z_0/2} + \|\varphi\|_{L^2(S)}),   \]
		which is uniformly bounded. Hence, up to a subsequence, $u^{(l)}$ $C^2(S\times [z_0/12, 3z_0/4])$-converges to some $u^{(\infty)}$ which agrees with $u_\varphi$ on $S\times [z_0/2, 3z_0/4]$, and solves, 
		\begin{align*}
			\begin{cases}
				(\partial_z^2 + \partial_z + z^{-1}L_S )u = \eta_{z_0}(z)\cdot\left(\scR(u) + z^{-1}\scT(0)\right), \ &\ \text{ on }S\times (z_0/12, 3z_0/4];\\
				u(\cdot, z_0/12) = 0, \ \ \  \Pi_{<0}\partial_z u(\cdot , z_0/12) = \varphi, \ &\ \text{ on }S.
			\end{cases}
		\end{align*}
		This proves the existence of a solution. To see the uniqueness, notice that if $u_\varphi^\pm \in \mathfrak{X}^\alpha_\gamma(S, z_0)$ are both fixed point of $\scU_\varphi$, then $\scU_\varphi$ contracting $\|\cdot\|_{\mathfrak{X}^\alpha_\gamma(S, z_0)}$ norm guarantees that $u_\varphi^+ = u_\varphi^-$ on $S\times \RR_{\geq z_0/2}$.  Then since $(\partial_z^2 + \partial_z + z^{-1}L_S )u_\varphi^\pm = 0$ on $S\times [z_0/12, z_0/2]$ both with $0$-boundary value, we know that $u_\varphi^+ = u_\varphi^-$ by the positivity of $-T_S$ when $z_0\geq C(S)$. Also by taking $z_0\gg 1$, the RHS of (\ref{Moduli_varphi <-> u_varphi biLip}) follows directly from (\ref{Moduli_scU_varphi(v)|^+_- contraction}).\\
		\textbf{Step 3.} To prove LHS of (\ref{Moduli_varphi <-> u_varphi biLip}), let $\varphi_+\neq \varphi_- \in E_{<0}$. Notice that $v:= u_{\varphi_+} - u_{\varphi_-}$ satisfies an equation of form,
		\begin{align}
			\begin{cases}
				(\partial_z^2 + \partial_z + z^{-1}L_S )v = \eta_{z_0}(z)\cdot\tilde{\scR}(v), \ &\ \text{ on }S\times \RR_{>z_0/12};\\
				v(\cdot, z_0/12) = 0, \ &\ \text{ on }S.
			\end{cases} \label{Moduli_Equ of u_varphi_pm}
		\end{align}
		where by Lemma \ref{Append_Error est II of scT}, $\tilde{\scR}(v)$ is a second order linear operator in $v$ of the following form, \[
		z\tilde{\scR}(v) = \bar{\cE_1}\cdot z\partial_z^2v + (\bar{\cE_2} + x^S)\cdot \partial_z\nabla v + \bar{\cE_3}\cdot \nabla^2v + \bar{\cE_4}\cdot z\partial_z v + \bar{\cE_5}\cdot \nabla v + \bar{\cE_6}v,   \]
		where for $1\leq l\leq 6$, $\bar{\cE_l}$ satisfies the $C^\alpha_\star$-estimates when $R\gg 1$, 
		\begin{align*}
			\|\bar{\cE_l}\|_{C^\alpha_\star, S, R} \leq C(S, \alpha)(R^{-1} + \|u_{\varphi_\pm}\|_{C^{2,\alpha}_\star, S, R}) \leq C(S, \gamma, \alpha)\cdot R^\gamma.
		\end{align*}
		Consider the barrier function $U(x, z):= z^{\mu_1+\gamma/4}\psi_1(x)$, and we recall that $\psi_1>0$ is the first eigenfunction of $-L_S$. A direct computation shows that 
		\begin{align}
			\begin{split}
				&\ (\partial_z^2 + \partial_z + z^{-1}L_S )U - \eta_{z_0}(z)\cdot\tilde{\scR}(U) \\
				\leq &\ z^{\mu_1+\gamma/4 -1}\psi(x)\left(\gamma/4 + \frac{(\mu_1+\gamma/4)(\mu_1+\gamma/4 -1)}{z} + C(S, \gamma, \alpha)z^\gamma\right) \leq 0, 
			\end{split} \label{Moduli_Barrier U = z^(mu_1+gamma/4)}
		\end{align}
		provided $z_0\gg 1$. Hence by maximum principle, 
		\begin{align*}
			\|u_{\varphi_+}-u_{\varphi_-}\|_{C^0, S\times [z_0/12, \bar{z}]} \leq C(S, \bar{z})\|u_{\varphi_+}-u_{\varphi_-}\|_{C^0, S\times \{\bar{z}\}}.
		\end{align*}
		Since $\varphi_\pm = \Pi_{<0}\partial_z u_{\varphi_\pm}(\cdot, z_0/12)$, by classical boundary elliptic estimates \cite{GilbargTrudinger01}, we have \[
		\|\varphi_+-\varphi_-\|_{L^2(S)}\leq C(S, \bar{z})\|u_{\varphi_+}-u_{\varphi_-}\|_{C^1, S\times [z_0/12, z_0/4]} \leq C(S, \bar{z})\|u_{\varphi_+}-u_{\varphi_-}\|_{C^0, S\times [z_0/12, \bar{z}]}.   \]
		These two together prove the LHS of (\ref{Moduli_varphi <-> u_varphi biLip}). \\
		\textbf{Step 4.} To prove (\ref{Moduli_L^2_mu_1 Diff of u_varphi w different initial}), suppose $\varphi_+, \varphi_- \in E_{<0}$ such that \[     
		\limsup_{R\to +\infty} R^{-\mu_1}\|(u_{\varphi_+} - u_{\varphi_-})(\cdot, R)\|_{C^0, S} = 0.   \] 
		As in Step 3, $v:= u_{\varphi_+} - u_{\varphi_-}$ still satisfies an equation of the form (\ref{Moduli_Equ of u_varphi_pm}). Therefore, by the similar proof of Lemma \ref{Lem_Moduli_C^2,alpha_star est, interior}, we have, \[
		\limsup_{R\to +\infty} R^{-\mu_1}\|u_{\varphi_+} - u_{\varphi_-}\|_{C^{2,\alpha}_\star, S, R} = 0.   \]
		Hence there exists $M(u_{\varphi_\pm})>1$ such that $\forall R\geq z_0/2$,  \[
		\|\tilde{\scR}(v)\|_{C^\alpha_\star, S, R}\leq M\cdot R^{\mu_1+\gamma - 1}.  \]
		
		Let $w\in L^2_{\mu_1+\gamma/2}$ be the unique solution to 
		\begin{align*}
			\begin{cases}
				(\partial_z^2 + \partial_z + z^{-1}L_S) w = \eta_{z_0}(z)\cdot \tilde{\scR}(v), &\  \text{ on } S\times \RR_{> z_0/12}; \\
				w(\cdot, z_0/12) = 0, &\ \text{ on }S,
			\end{cases}
		\end{align*}
		given by Lemma \ref{Lem_Moduli_T_S u = f}.  Then together with the $C^0$ estimate Lemma \ref{Lem_Moduli_C^0 est}, for every $R\geq z_0/6$, \[
		\|w\|_{C^0, S, R} \leq C(S, \gamma, \alpha)M\cdot R^{\mu_1+\gamma/2}.   \] 
		Thus $\|(v-w)(\cdot, R)\|_{C^0, S}\cdot R^{-\mu_1} \to 0$ as $R\to +\infty$.  
		
		On the other hand, since 
		\begin{align*}
			\begin{cases}
				(\partial_z^2 + \partial_z + z^{-1}L_S)(v- w) \equiv 0, &\  \text{ on } S\times \RR_{> z_0/12}; \\
				(v-w)(\cdot, z_0/12) = 0, &\ \text{ on }S. 
			\end{cases}
		\end{align*}
		Thus by Corollary \ref{Cor_Moduli_Sharp asymp for T_S u = 0}, $\cA\cR_\infty(v-w)=-\infty$. Then the uniqueness assertion in Lemma \ref{Lem_Moduli_T_S u = f} forces $v-w \equiv 0$, which gives an improved $C^0$-decaying estimate for $v$, \[
		\|v\|_{C^0, S, R} \leq C(S, \gamma, \alpha)M\cdot R^{\mu_1+\gamma/2}, \ \ \ \ \ \forall R\geq z_0.   \] 
		
		Consider the barrier function $U(x, z):= z^{\mu_1+\gamma/4}\psi_1(x)$ which satisfies (\ref{Moduli_Barrier U = z^(mu_1+gamma/4)}). Since \[
		\lim_{z\to \infty} v(\cdot, z)/U(\cdot, z) = 0,  \] 
		by (\ref{Moduli_Equ of u_varphi_pm}) and maximum principle, we have $v \equiv 0$ on $S\times \RR_{\geq z_0/12}$, and thus $\varphi_+ - \varphi_- = 0$. This finishes the proof of (\ref{Moduli_L^2_mu_1 Diff of u_varphi w different initial}).
	\end{proof}
	\begin{Rem} \label{Rem_Moduli_|u_(varphi^pm)| lower bd}
		By Step 3 and 4 in the proof above, we see that the LHS of (\ref{Moduli_varphi <-> u_varphi biLip}) and (\ref{Moduli_L^2_mu_1 Diff of u_varphi w different initial}) do not rely on the smallness assumption (\ref{Moduli_smallness of u_z(z_0)}) of $\varphi_\pm$. In fact, it holds for every pair $u_\pm\in \mathfrak{X}^\alpha_\gamma(S, z_0)$ solving (\ref{Moduli_Transl equ w cut-off and prescib bdy deriv}), with $\varphi_\pm:= \Pi_{<0} \partial_z u_\pm(\cdot, z_0/12)$.  This will be used to deduce the continuity of $\RR$-action in Section \ref{Subsec_Moduli_RR action}.
	\end{Rem}
	
	We now assert that any pl-simple translating end over $S\times \RR_+$ has a uniform asymptotic rate bound.
	\begin{Lem} \label{Lem_Moduli_Sharp asymp rate for pl-simple end}
		Let $\ES[u] \subset \RR^n\times \RR_{>z_0}$ be a pl-simple translating end over $S\times \RR_+$.  Then \[
		\limsup_{R\to +\infty} \|u\|_{C^{2,\alpha}_\star, S, R}\cdot R^{-\mu^-} < +\infty.   \]
	\end{Lem}
	\begin{proof}
		Let \[
		\gamma_0(u):= \inf\{\gamma<0: \limsup_{R\to +\infty} \|u\|_{C^{2,\alpha}_\star, S, R}\cdot R^{-\gamma} < +\infty\}.   \]
		Note that by Definition \ref{Def_Moduli_p-simple Transl end} and Remark \ref{Rem_Moduli_C^0 vs C^k decay}, $0> \gamma_0\geq -\infty$.  And if $\gamma_0 < \mu^-$ then we are immediately done.  Now assume $\gamma_0\in [\mu^-, 0)$.  Let $\gamma_1 \in (2\gamma_0, \gamma_0)\setminus \Gamma(S)$.  
		Then by (\ref{Moduli_C^alpha_star error est scR^+-scR^-}) for $R\geq z_1(u)\gg z_0$, \[
		\|\scR(u) + z^{-1}\scT(0)\|_{C^\alpha_\star, S, R}\leq R^{-1+\gamma_1}.   \]
		By Lemma \ref{Lem_Moduli_T_S u = f}, there exists $w\in L^2_{\gamma_1}(S\times \RR_{\geq z_1})$ solving,
		\begin{align*}
			\begin{cases}
				(\partial_z^2 + \partial_z + z^{-1}L_S) w = \scR(u) + z^{-1}\scT(0), &\ \text{ on }S\times \RR_{>z_1}, \\
				w(\cdot, z_1) = 0, &\ \text{ on }S.
			\end{cases}
		\end{align*}
		And as before, the elliptic estimates Lemmas \ref{Lem_Moduli_T_S u = f}, \ref{Lem_Moduli_C^0 est} and \ref{Lem_Moduli_C^2,alpha_star est, interior} gives for $R\geq 6z_1$,
		\begin{align}
			\|w\|_{C^{2,\alpha}_\star, S, R} \leq C(S, \gamma_1, \alpha)R^{\gamma_1}.  \label{Moduli_error, rate|w|<gamma_1}   
		\end{align}
		Therefore, by the choice of $\gamma_0$ and $\gamma_1$, we have,
		\begin{align}
			\limsup_{R\to +\infty} \|u-w\|_{C^{2,\alpha}_\star, S, R}\cdot R^{-\gamma} & = +\infty, &\ & \forall \gamma<\gamma_0 ; \label{Moduli_main, rate|u-w|>gamma_0-} \\
			\limsup_{R\to +\infty} \|u-w\|_{C^{2,\alpha}_\star, S, R}\cdot R^{-\gamma} & = 0, &\ & \forall \gamma>\gamma_0 . \label{Moduli_main, rate|u-w|<gamma_0+}
		\end{align}
		Since the choice of $w$ guarantees $T_S(u-w) = 0$, by (\ref{Moduli_main, rate|u-w|<gamma_0+}), Corollary \ref{Cor_Moduli_Sharp asymp for T_S u = 0}, Lemmas \ref{Lem_Moduli_C^0 est} and \ref{Lem_Moduli_C^2,alpha_star est, interior} we have, 
		\begin{align}
			\limsup_{R\to +\infty} \|u-w\|_{C^{2,\alpha}_\star, S, R} \cdot R^{-\mu^-} < +\infty.  \label{Moduli_main, rate|u-w|=mu^-}   
		\end{align}
		Then combined with (\ref{Moduli_main, rate|u-w|>gamma_0-}), we know that $\gamma_1< \gamma_0\leq \mu^-$.  And therefore the Lemma follows from (\ref{Moduli_error, rate|w|<gamma_1}) and (\ref{Moduli_main, rate|u-w|=mu^-}).
	\end{proof}
	\begin{Rem}
		A similar argument of this gives that for every pl-simple translating end $\ES[u]$ over $S\times \RR_+$, we have the asymptotic rate of $u$ satisfies, \[
		\cA\cR_\infty(u):= \inf\{\gamma<0: \limsup_{R\to +\infty} \|u\|_{C^{2,\alpha}_\star, S, R}\cdot R^{-\gamma} < +\infty\} \in [-\infty, -1]\cup \Gamma(S).  \]
	\end{Rem}

	The following Lemma guarantees that heuristically, those $u_\varphi$ constructed in Lemma \ref{Lem_Moduli_contraction map} cover all pl-simple translating ends over $S\times \RR_+$, up to an exponentially decaying error.
	\begin{Lem} \label{Lem_Moduli_general transl end is one of the constructed}
		For every $\alpha\in (0,1)$, $\beta\in (1/2, 1)$ and $\gamma \in (\mu^-, 0)$, there exists $C(S, \gamma, \alpha, \beta)>1$ and $\delta_1(S, \gamma, \alpha, \beta)\in (0,\delta_0)$, where $\delta_0$ is given by Lemma \ref{Lem_Moduli_contraction map}, such that for every $z_0\geq C(S, \gamma, \alpha, \beta)$, if $w\in C^{2,\alpha}_{loc}(S\times \RR_{\geq z_0/12})$ satisfies $\scT(w) = 0$ and, \[
		\sup_{R\geq z_0/12}  R^{-\gamma}\cdot\|w\|_{C^{2,\alpha}_\star, R} \leq \delta_1,   \]
		Then there exists a unique $u\in \mathfrak{X}^\alpha_\gamma(S, z_0)$ solving,
		\begin{align*}
			\begin{cases}
				(\partial_z^2 + \partial_z + z^{-1}L_S )u = \eta_{z_0}(z)\cdot\left(\scR(u) + z^{-1}\scT(0)\right), \ &\ \text{ on }S\times \RR_{>z_0/12};\\
				u(\cdot, z_0/12) = 0, \ \ \  \|\Pi_{<0}\partial_z u(\cdot , z_0/12)\|_{L^2(S)} \leq \delta_0 z_0^\gamma, \ &\ \text{ on }S.
			\end{cases} 
		\end{align*}
		and satisfying for every $z\geq z_0$, 
		\begin{align}
			\|(u-w)(\cdot, z)\|_{C^0(S)}\leq  z_0^\gamma e^{-\beta (z-z_0)} . \label{Moduli_sol u exp asymp to w}
		\end{align}
	\end{Lem}
	
	To prove the Lemma, recall the $C^{d,\alpha}_\sharp$ norms are defined in (\ref{Moduli_C^k,alpha_sharp norms}). By Lemma \ref{Append_Error est II of scT}, we have
	\begin{align}
		\|\scR(u^+) -\scR(u^-)\|_{C^\alpha_\sharp; S, R} \leq C(S,\alpha)\left(R^{-1}+\|u^\pm\|_{C^{2,\alpha}_\star, S, R}\right)\cdot\|u^+ - u^-\|_{C^{2,\alpha}_\sharp; S, R}  \label{Moduli_C^alpha_sharp error est scR^+-scR^-}
	\end{align}
	
	\begin{proof}[Proof of Lemma \ref{Lem_Moduli_general transl end is one of the constructed}] 
		Define \[
		\mathfrak{Y}^\alpha_{\gamma, \beta}(S, z_0):= \left\{v\in \mathfrak{X}^\alpha_\gamma(S, z_0): \sup_{R\geq z_0} e^{\beta(R-z_0)}\|v- w\|_{C^{2,\alpha}_\sharp, S, R}\leq z_0^\gamma \right\}.    \]
		Consider the map $\scV_w$ defined on $\mathfrak{Y}^\alpha_{\gamma, \beta}(S, z_0)$, $v\mapsto u:= \scV_w(v)$ solving,
		\begin{align}
			\begin{cases}
				(\partial_z^2 + \partial_z + z^{-1}L_S )u = \eta_{z_0}(z)\cdot\left(\scR(v) + z^{-1}\scT(0)\right), \ &\ \text{ on }S\times \RR_{>z_0/12};\\
				u(\cdot, z_0/12) = 0, \ &\ \text{ on }S; \\
				\|u-w\|_{S, z_0/12; \mu_1-1} <+\infty.
			\end{cases}  \label{Moduli_Def scV_w by solving equ}
		\end{align}
		\textbf{Step 1.} Let's first verify that $\scV_w$ is a well-defined map from $\mathfrak{Y}^\alpha_{\gamma, \beta}(S, z_0)$ to itself.
		For every $v\in \mathfrak{Y}^\alpha_{\gamma, \beta}(S, z_0)$, by (\ref{Moduli_C^alpha_sharp error est scR^+-scR^-}), for every $R\geq z_0$,
		\begin{align}
			\|\scR(v) - \scR(w)\|_{C^\alpha_\sharp, S, R} \leq C(S, \gamma, \alpha)\cdot R^\gamma \|v-w\|_{C^{2,\alpha}_\sharp, S, R} \leq C(S, \gamma, \alpha)\cdot z_0^{2\gamma} e^{-\beta(R-z_0)}.  \label{Moduli_|scR(v)-scR(w)|^sharp_alpha < e^-R}
		\end{align}
		Also, by (\ref{Moduli_C^alpha_star error est scR^+-scR^-}) and assumption on $v$ and $w$, for every $R\in [z_0/12, z_0]$ we have 
		\begin{align}
			\|\eta_{z_0}\cdot\scR(v)\|_{C^\alpha_\star, S, R} + \|\scR(w)+z^{-1}\scT(0)\|_{C^\alpha_\star, S, R} \leq C(S, \alpha, \gamma)R^{-1+2\gamma}.  \label{Moduli_|scR(v) + scR(w)|^star_alpha < R^-1+2gamma}
		\end{align}
		And notice that equation (\ref{Moduli_Def scV_w by solving equ}) is equivalent to,
		\begin{align*}
			\begin{cases}
				T_S (u - w) = \eta_{z_0}\cdot\left(\scR(v) - \scR(w)\right) - (1-\eta_{z_0})(\scR(w)+z^{-1}\scT(0)), \ &\ \text{ on }S\times \RR_{>z_0/12};\\
				(u - w)(\cdot, z_0/12) = -w(z_0/12), \ &\ \text{ on }S; \\
				\|u-w\|_{S, z_0/12; \mu_1-1} <+\infty.
			\end{cases} 
		\end{align*}   
		Thus by Lemma \ref{Lem_Moduli_T_S u = f}, equation (\ref{Moduli_Def scV_w by solving equ}) has a unique solution $u$ satisfying 
		\begin{align}
			\begin{split}
				&\ \|u-w\|_{L^2_{\mu_1-1}, S, z_0/12}^2 \\
				\leq &\ C(S)\Big(\|w(\cdot, z_0/12)\|^2_{C^0(S)}\cdot z_0^{-2\mu_1+1} + \|\eta_{z_0}(\scR(v)-\scR(w))\|^2_{L^2_{\mu_1-2}, S, z_0/12} \\
				&\ \ \ \ \ \ +\|(1-\eta_{z_0})(\scR(w)+z^{-1}\scT(0))\|^2_{L^2_{\mu_1-2}, S, z_0/12}\Big) \\
				\leq &\ C(S, \gamma, \alpha, \beta)\Big(\delta_1^2 z_0^{2\gamma-2\mu_1 +1} + \int_{S\times \RR_{>2z_0}} |\scR(v) - \scR(w)|^2 z^{-2(\mu_1-2)-1}\ dxdz \\
				&\ \ \ \ \ \  + \|\eta_{z_0}|\scR(v)|+|\scR(w)|+|z^{-1}\scT(0)|\|^2_{L^2, S\times (z_0/12, 2z_0)}z_0^{-2(\mu_1 -2)-1} \Big) \\
				\leq &\ C(S, \gamma, \alpha, \beta)\left( \delta_1^2 z_0^{-1} + e^{-\beta z_0}z_0^{-2\gamma +1} + z_0^{2\gamma} \right)\cdot z_0^{2\gamma-2\mu_1+2}.
			\end{split} \label{Moduli_L^2_(mu_1-1) bd for u-v in frakY}
		\end{align}
		Thus by $C^0$-estimate Lemma \ref{Lem_Moduli_C^0 est}, for every $R\geq z_0/6$,
		\begin{align*}
			\|u\|_{C^0, S, R} \leq &\ C(S)\left(R^{-1/2}\||u-w| + |w|\|_{L^2, S\times (R/2, 4R)} + R\|\eta_{z_0}(\scR(v)+z^{-1}\scT(0))\|_{C^0, S\times (R/2, 4R)}\right) \\
			\leq &\ C(S, \gamma, \alpha, \beta) (\delta_1 + z_0^\gamma)R^\gamma,
		\end{align*}
		where the last inequality hold if $z_0\gg 1$.  Then as before, by Lemma \ref{Lem_Moduli_C^2,alpha_star est, interior}, we derive for $R\geq z_0/2$, 
		\begin{align*}
			\|u\|_{C^{2,\alpha}_\star, S, R} \leq C(S, \gamma, \alpha, \beta)(\delta_1+z_0^\gamma)R^\gamma < R^\gamma, 
		\end{align*}
		provided $z_0\gg 1$ and $\delta_1\ll1$.  This means $u\in \mathfrak{X}^\alpha_\gamma(S, z_0)$.  
		
		Also, when $z\geq 3z_0/4\gg 1$, by (\ref{Moduli_C^alpha_star error est scR^+-scR^-}), (\ref{Moduli_|scR(v)-scR(w)|^sharp_alpha < e^-R}) and the definition of $\mathfrak{X}^\alpha_\gamma(S, z_0)$, 
		\begin{align*}
			|T_S (u - w)| & = |\eta_{z_0}\cdot (\scR(v) - \scR(w))| \\
			& \leq C(S, \gamma, \alpha)z_0^{2\gamma}e^{-\beta(z- z_0)} \leq -C(S, \gamma, \alpha, \beta)(\delta_1 + z_0^\gamma)z_0^\gamma \cdot T_S(e^{-\beta(z-z_0)}\psi_1).  
		\end{align*}
		By maximum principle and the uniqueness assertion in Lemma \ref{Lem_Moduli_T_S u = f}, we thus have the following refined $C^0$ estimate for $z\geq 3z_0/4$,
		\begin{align*}
			|u-w|(x, z) \leq C(S, \gamma, \alpha, \beta)(\delta_1 + z_0^\gamma)z_0^\gamma\cdot e^{-\beta(z-z_0)}\psi_1(x).
		\end{align*}
		Then by (\ref{Moduli_|scR(v)-scR(w)|^sharp_alpha < e^-R}) and the short term $C^{2,\alpha}_\sharp$ estimate Lemma \ref{Lem_Moduli_short term Schauder est, int}, we have for every $R\geq z_0$,
		\begin{align*}
			\|u-w\|_{C^{2,\alpha}_\sharp, S, R}\leq C(S, \gamma, \alpha, \beta)(\delta_1+z_0^\gamma) z_0^\gamma \cdot e^{-\beta(R-z_0)}.
		\end{align*}
		By taking $z_0\gg 1$, this means $u\in \mathfrak{Y}^\alpha_{\gamma, \beta}(S, z_0)$. In other words, $\scV_w$ maps $\mathfrak{Y}^\alpha_{\gamma, \beta}(S, z_0)$ to itself. \\
		\textbf{Step 2.} Repeat the Step 2 in the proof of Lemma \ref{Lem_Moduli_contraction map}, one can show that if $z_0\gg 1$ and $\delta_1\ll1$, then $\scV_w$ contracts $\|\cdot\|_{\mathfrak{X}^\alpha_{\mu_1-1}(S, z_0)}$-distance.  Hence by iteration starting from $w$ and taking limit, we shall find a unique fixed point $u_w$ of $\scV_w$, which is the desired solution to the equation in Lemma \ref{Lem_Moduli_general transl end is one of the constructed}; the desired estimate on $\Pi_{<0}\partial_z u(\cdot, z_0/12)$ follows from (\ref{Moduli_L^2_(mu_1-1) bd for u-v in frakY}), the $C^0$-boundary estimate Remark \ref{Rem_Moduli_C^0 est, bdy case} and the short term boundary estimate Lemma \ref{Lem_Moduli_short term Schauder est, bdy}. 
	\end{proof}

	\begin{proof}[Proof of Theorem \ref{Thm_Moduli of transl end }]
		Recall $\mu^- := \sup(\Gamma(S)\cap \RR_{<0})$.  Fix the choice $\gamma:= \mu^-/2$ in Lemma \ref{Lem_Moduli_contraction map}; Identify $\BB^I$ as of ball of radius $\delta_0 z_0^{\gamma}/2$ centered at $\orig$ in $(E_{<0}, \|\cdot\|_{L^2(S)})$.  For every $\varphi\in \BB^I$, let $u_{\varphi, z_0}$ be the unique solution to (\ref{Moduli_Transl equ w cut-off and prescib bdy deriv}). (i), (ii) in Theorem \ref{Thm_Moduli of transl end } follow from Lemma \ref{Lem_Moduli_contraction map}. Note that (ii) guarantees that $u_{\varphi, z_0}$ varies continuously with respect to $\varphi$. 
		(iii) follows from Lemma \ref{Lem_Moduli_Sharp asymp rate for pl-simple end} and \ref{Lem_Moduli_general transl end is one of the constructed} by taking $C_1\gg 1$ in Theorem \ref{Thm_Moduli of transl end } (iii).
		
		When $S$ is acted by some Lie group $G\subset O(n)$, then the eigenfunction with the lowest eigenvalue, $\psi_1\in E_{<0}$ is also $G$-invariant. Thus the uniqueness assertion in Lemma \ref{Lem_Moduli_contraction map} implies that $u_{c\psi_1, z_0}$ is $G$-invariant for every small $c\in \RR$.
	\end{proof}

	\subsection{$\RR$-action on translating ends} \label{Subsec_Moduli_RR action}
	We finish this Section by introducing an $\RR$-action on the space of pl-simple translating ends. 
	
	\begin{Lem} \label{Lem_Moduli_one-sided transl end}
		Let $\alpha\in (0, 1)$, $\gamma\in (\mu^-, 0)$. Then there exists $\underline{z}(S, \alpha, \gamma)\gg 1$ and $\delta_2(S, \alpha, \gamma)\in (0, 1)$ with the following property.  
		
		Let $\delta\in (0, \delta_2)$, $z_0\geq \delta^{1/(2\gamma)}\cdot \underline{z}$, $a\in [-\delta\cdot z_0^{-\mu_1}, \delta\cdot z_0^{-\mu_1}]$. Let $u_0\in C^\infty(S\times \RR_{>z_0/2})$ be the solution of $\scT(u_0) = 0$ on $S\times \RR_{\geq z_0}$, with $\|u_0\|_{\mathfrak{X}^\alpha_\gamma; S, z_0} \leq 1$.  Then there exists a solution $u_a$ of $\scT(u_a) = 0$ on $S\times \RR_{>z_0}$ with the following estimate: for every $R\geq z_0$, 
		\begin{align}
			\|u_a - u_0 - az^{\mu_1}\psi_1(x)\|_{C^{2,\alpha}_\star, S, R} \leq (1+|a|)\sqrt{\delta}\cdot \left(\frac{R}{z_0}\right)^{\gamma/2} R^{\mu_1}.  \label{Moduli_est one-sided deformtion}
		\end{align}
		Here recall that $\psi_1$ is the first $L^2(S)$-unit eigenfunction of $-L_S$. The collection of such solutions is denoted by $\varpi(a)[u_0]$.
		Moreover, if there's a Lie group $G\subset O(n)$ acting on $S$ and $u_0$ is $G$-invariant, then so is $u_a$ for every $a$.	
	\end{Lem}
	It's worth mentioning that if we restrict \eqref{Moduli_est one-sided deformtion} to $\{z=z_0\}$ slice and take $a=\pm \delta z_0^{-\mu_1}$, then \[
	u_a(\cdot, z_0) = u_0(\cdot, z_0) \pm \delta\psi_1\cdot (1+ O(\sqrt{\delta})).   \]
	Also, if we define the space of pl-simple translating end classes over $S\times \RR_+$ to be
	\begin{align}
		\cE_{pl}(S):= \{u\in C^\infty(S\times \RR_{>z_0}): z_0>1, \ES[u] \text{ is a pl-simple end over }S\times \RR_+\}/\sim, \label{Moduli_Def_Poly asymp transl ends}
	\end{align}
	where $u_1\sim u_2$ if and only if \[
	\limsup_{z\to \infty}\|(u_2 - u_1)(\cdot, z)\|_{C^0(S)}\cdot e^{z/2} = 0.   \]
	Under this notation, Theorem \ref{Thm_Moduli of transl end } (i) and (ii) imply that for $z_0\gg 1$, $\{u_{\varphi, z_0}\}_{\varphi\in \BB^I}$ determines an embedding of $\BB^I$ into $\cE_{pl}(S)$, and (iii) guarantees that whenever $z_j\nearrow +\infty$, $\{[u_{\varphi, z_j}]\}_{\varphi\in \BB^I}$ exhausts $\cE_{pl}(S)$. While Lemmas \ref{Lem_Moduli_one-sided transl end} and \ref{Lem_Moduli_Sharp asymp rate for pl-simple end} guarantee that for every equivalent class $[u]\in \cE_{pl}(S)$ and every $a\in \RR$, there's a unique $\varpi(a)[u] \in \cE_{pl}(S)$ such that for every $v\in \varpi(a)[u]$, \[
	\lim_{z\to +\infty} z^{-\mu_1}(v-u)(\cdot, z) = a\psi_1.   \]
	In particular, since $\psi_1>0$, this well defines an $\RR$-action on $\cE_{pl}(S)$ such that for every $a\neq 0$, if $v \in \varpi(a)[u]$, then $\ES[v]\cap \ES[u]\cap \RR^n\times \RR_{>\bar{z}} = \emptyset$ for some large $\bar{z}(u, v)\gg 1$.

	\begin{proof}[Proof of Lemma \ref{Lem_Moduli_one-sided transl end}.]
		Let $\Psi_1(x, z):= z^{\mu_1}\psi_1(x)$.  Define \[
		\mathfrak{Z}^\alpha_\gamma(S, z_0, a):= \{u\in C^{2,\alpha}_{loc}(S\times \RR_{\geq z_0/12}): \|u - u_0 - a\Psi_1\|_{\mathfrak{X}^\alpha_{\mu_1+\gamma/2}, z_0} \leq (1+|a|)\sqrt{\delta}\cdot z_0^{-\gamma/2}\}.   \]
		Note that for every $v\in \mathfrak{Z}^\alpha_\gamma(S, z_0, a)$ and every $R\geq z_0/2$, we have,
		\begin{align}
			\begin{split}
				\|v\|_{C^{2,\alpha}_\star, S, R} & \leq \|v - u_0 - a\Psi_1\|_{\mathfrak{X}^\alpha_{\mu_1+\gamma/2}, z_0} R^{\mu_1+\gamma/2} + \|u_0\|_{C^{2,\alpha}_\star, S, R} + |a|\cdot\|\Psi_1\|_{C^{2,\alpha}_\star, S, R} \\
				& \leq C(S, \alpha)(1 + \delta z_0^{-\gamma})R^\gamma \ll1,
			\end{split}  \label{Moduli_Equ_A prior est for v in mathf(Z)}
		\end{align}
		provided $\delta_2\ll1$ and $\underline{z}\gg 1$.  Thus by Lemma \ref{Append_Error est II of scT}, we have for every $v\in \mathfrak{Z}^\alpha_\gamma(S, z_0, a)$ and every $R\geq z_0/2$, 
		\begin{align}
			\|\scR(v) - \scR(u_0)\|_{C^\alpha_\star, S, R}\leq C(S, \alpha)(1 + \delta z_0^{-\gamma})(1+|a|)\cdot R^{\mu_1+\gamma-1}, \label{Moduli_Equ_est for |scR(v)-scR(u_0)| w v in mathf(Z)}
		\end{align}
		while for every $v_\pm \in\mathfrak{Z}^\alpha_\gamma(S, z_0, a)$ and every $R\geq z_0/2$,
		\begin{align}
			\|\scR(v_+) - \scR(v_-)\|_{C^\alpha_\star, S, R}\leq C(S, \alpha)(1 + \delta z_0^{-\gamma})\|v_+ - v_-\|_{\mathfrak{X}^\alpha_{\mu_1+\gamma/2}, z_0}\cdot R^{\mu_1+3\gamma/2 -1}.\label{Moduli_Equ_est for |scR(v_+)-scR(v_-)| w v_pm in mathf(Z)}
		\end{align}
		
		Let $\scW_a$ be the map defined on $\mathfrak{Z}^\alpha_\gamma(S, z_0, a)$ given by $u:= \scW_a(v)$ solving 
		\begin{align}
			\begin{cases}
				(\partial_z^2 + \partial_z + z^{-1}L_S)(u-u_0) = \eta_{z_0}\cdot (\scR(v)-\scR(u_0))\ &\ \text{ on }S\times \RR_{>z_0/12}; \\
				(u-u_0)(\cdot, z_0/12) = 0, \ &\ \text{ on }S; \\
				\|u-u_0- a\Psi_1\|_{L^2_{\mu_1+\gamma/2}; S, z_0/12} < +\infty.
			\end{cases} \label{Moduli_One-sided transl end Equ}
		\end{align}
		Similar as before, we shall verify that $\scW_a$ maps $\mathfrak{Z}^\alpha_\gamma(S, z_0, a)$ to itself and contracts $\|\cdot\|_{\mathfrak{X}_{\mu_1+\gamma/2}^\alpha}$-distance, provided $\delta_2\ll1$ and $\underline{z}\gg 1$.  In fact, by Lemma \ref{Lem_Moduli_T_S u = f} and (\ref{Moduli_Equ_est for |scR(v)-scR(u_0)| w v in mathf(Z)}), the equation (\ref{Moduli_One-sided transl end Equ}) has a unique solution $u$ satisfying, 
		\begin{align*}
			&\ \|u - u_0 - a\Psi_1\|_{L^2_{\mu_1+\gamma/2}, S, z_0/12} & & \\
			\leq &\ C(S, \gamma, \alpha)\Big( \|\eta_{z_0}\cdot (\scR(v)-\scR(u_0))\|_{L^2_{\mu_1+\gamma/2-1}, S, z_0/12} & & + |a|\cdot\|T_S \Psi_1\|_{L^2_{\mu_1+\gamma/2-1}, S, z_0/12} \\
			&\ & & + \|a\Psi_1(\cdot, z_0/12)\|_{L^2(S)}\cdot z_0^{-\mu_1 - \gamma/2-1/2} \Big) \\
			\leq &\ C(S, \gamma, \alpha)(1+\delta z_0^{-\gamma})(1+|a|)\cdot z_0^{\gamma/2} . & &
		\end{align*}
		Then by Lemma \ref{Lem_Moduli_C^0 est} and \ref{Lem_Moduli_C^2,alpha_star est, interior}, for $R\geq z_0/2$ we have 
		\begin{align*}
			\|u-u_0 - a\Psi_1\|_{C^{2,\alpha}_\star, S, R} & \leq C(S, \gamma, \alpha)(1+\delta z_0^{-\gamma})(1+|a|)z_0^{\gamma/2}\cdot R^{\mu_1 + \gamma/2}  \\
			& < (1+|a|)\sqrt{\delta} z_0^{-\gamma/2}\cdot R^{\mu_1 + \gamma/2},   
		\end{align*}
		where the last inequality holds if $\underline{z}\gg 1$ and $\delta_2\ll1$, since we are taking $z_0\geq \delta^{1/(2\gamma)}\underline{z}$.  This means $\scW_a(v) = u \in \mathfrak{Z}^\alpha_\gamma(S, z_0, a)$.  Similarly using (\ref{Moduli_Equ_est for |scR(v_+)-scR(v_-)| w v_pm in mathf(Z)}), one can show that $\scW_a$ contracts $\|\cdot\|_{\mathfrak{X}_{\mu_1+\gamma/2}^\alpha}$ provided $z_1\gg 1$ and $\delta_1\ll1$.  Then, iterating $\scW_a$ starting from $u_0 + a\Psi_1$ and taking limit produce a fixed point $u_a$ of $\scW_a$, which is what we want in Lemma \ref{Lem_Moduli_one-sided transl end}.
		
		Finally, if there's a Lie group $G\subset O(n)$ acting on $S$, then by the uniqueness of the first eigenfunction up to a constant, $\psi_1$ is $G$-invariant. Therefore, if $u_0$ is also $G$-invariant, then by the uniqueness of fixed points of $\scW_a$, $u_a$ should also be $G$-invariant.
	\end{proof}
	
	The proof of the above lemma also gives the continuity of the $\RR$-action in the following sense.
	\begin{Cor} \label{Cor_Moduli_RR-action conti}
		Let $z_0\geq C(S)\gg 1$, $\{u_{\varphi, z_0}\}_{\varphi\in \BB^I}$ be a fixed family given by Theorem \ref{Thm_Moduli of transl end } (See its proof at the end of Section \ref{Subsec_Moduli_Pf of Main Thm}). Then,
		\begin{enumerate} [(i)]
			\item for every $\varphi\in \BB^I$ and every $|a|\ll1$ (may depend on $\varphi$), $\exists\ \varphi_a\in \BB^I$ such that $u_{\varphi_a, z_0}\in \varpi(a)[\varphi]$;
			\item for $1\leq j\leq \infty$, let $\varphi^j$ be points in $\BB^I$, $a_j$ be real numbers such that $|a_j|\leq 1$, $\varphi^j\to \varphi^\infty$ in $\BB^I$ and $a_j\to a_\infty$. Also suppose that for $1\leq j< \infty$, $\exists\ \varphi^j_a\in \BB^I$ such that $u_{\varphi^j_a, z_0}\in \varpi(a_j)[u_{\varphi^j, z_0}]$. 
			
			Then $\varphi^j_a$ converges to some $\psi$ in $\BB^I$ if and only if there exists some $\varphi^\infty_a\in \BB^I$ such that $u_{\varphi^\infty_a, z_0}\in \varpi(a_\infty)[u_{\varphi^\infty, z_0}]$. When this happens, $\psi = \varphi^\infty_a$.
		\end{enumerate}
	\end{Cor}
	\begin{proof}
		By the proof of Lemma \ref{Lem_Moduli_one-sided transl end} above, we see that as long as $u_{\varphi, z_0}$ solves (\ref{Moduli_Transl equ w cut-off and prescib bdy deriv}), so do their one-sided deformations with probably a different $\varphi$. Therefore, (i) follows from the LHS of (\ref{Moduli_varphi <-> u_varphi biLip}) and Remark \ref{Rem_Moduli_|u_(varphi^pm)| lower bd}, while (ii) follows directly from (\ref{Moduli_L^2_mu_1 Diff of u_varphi w different initial}), (\ref{Moduli_est one-sided deformtion}) and the uniqueness assertion in Lemma \ref{Lem_Moduli_contraction map}.
	\end{proof}
	
	The following lemma guarantees that every simple translating end on one side is given by the $\varpi(a)$-action.
	\begin{Lem} \label{Lem_Moduli_General one-sided transl end}
		Let $\ES[u_0]$ be a pl-simple translating end, $\ES[u_+]$ be a simple translating end, both over $S\times \RR_+$.  Suppose $u_+\geq u_0$ near infinity. Then there exists $a\geq 0$ such that $u_+\in \varpi(a)[u_0]$.
	\end{Lem}
	\begin{proof}
		Let $3\gamma\in (\mu^-, 0)$ be fixed such that $(\mu_1, \mu_1-\gamma]\cap \Gamma(S)=\emptyset$, and $\delta(S, \gamma)\ll1$ TBD. By Lemma \ref{Lem_Moduli_Simple end <=> u to 0} and \ref{Lem_Moduli_Sharp asymp rate for pl-simple end} and taking $z_0(u_0, u_+, \delta)\gg 1$, we may assume that $u_+> u_0$ on $S\times \RR_{\geq z_0}$ and for every $R\geq z_0$,
		\begin{align*}
			\|u_0\|_{C^{2,1/2}_\star, S, R}\leq R^{2\gamma} \leq \delta, &\ & \|u_+\|_{C^{2,1/2}_\star, S, R}\leq \delta.
		\end{align*}
		We can impose the further assumption that,
		\begin{align}
			\limsup_{R\to +\infty} \|(u_+-u_0)(\cdot, R)\|_{C^0, S} R^{-\mu_1-\gamma/2} = +\infty.  \label{Moduli_|u_+-u_0| decay not too fast}
		\end{align}
		Otherwise, by the following Lemma \ref{Lem_Pf main thm_fast poly decay => exp decay}, we have $u_+\in [u_0]$.\\
		\textbf{Claim 1.} For $\delta(S, \gamma)\ll1$, \[
		\limsup_{R\to \infty} \|(u_+ - u_0)(\cdot, R)\|_{C^0, S}\cdot R^{-\mu_1+\gamma} <+\infty .   \]
		\begin{proof}[Proof of Claim 1.]
			Recall that $0 = \scT(u_+) - \scT(u_0) = -zT_S(u_+-u_0) + \tilde{\scR}(u_+-u_0)$, where $\tilde{\scR}$ is some second order linear differential operator, \[
			\tilde{\scR}(v) = \tilde{\cE}_1\cdot z\partial^2_z v + \tilde{\cE}_2\cdot\sqrt{z}\partial_z\nabla v + \tilde{\cE}_3 \cdot\nabla^2 v + \tilde{\cE}_4\cdot z\partial_z v + \tilde{\cE}_5\cdot \nabla v + \tilde{\cE}_6 v,   \]
			with pointwise estimates on the coefficients, for every $1\leq l\leq 6$, \[
			|\tilde{\cE}_l| \leq C(S)\delta.   \]
			Consider $U(x, z):= ((z/z_0)^{\mu_1-\gamma}- e^{-(z-z_0)/2})\psi_1$,  where recall that $\psi_1$ is the $L^2$-unit first eigenfunction of $-L_S$. A direct calculation provides $U(\cdot, z_0) = 0$ and,
			\begin{align}
				\begin{split}
					-z T_S(U) + \tilde{\scR}(U) & \leq \left(\frac{z}{z_0}\right)^{\mu_1-\gamma}\cdot \left( \gamma - \frac{(\mu_1-\gamma)(\mu_1-\gamma-1)}{z} +C(S)\delta \right)\psi_1 \\
					&\ + z e^{-(z-z_0)/2}\cdot \left( -\frac{1}{2} + \frac{1}{4} -\frac{\mu_1}{z} + C(S)\delta \right) \\
					& <0,
				\end{split} \label{Moduli_T+error(U)<0 in one-sided deform}
			\end{align}
			where the last inequality holds if $\delta(S, \gamma)\ll1$. 
			
			If Claim 1 fails, then there are $R_j\nearrow +\infty$ such that $\|(u_+-u_0)(\cdot, 2 R_j)\|_{C^0, S} \geq j R_j^{\mu_1-\gamma}$, and by Harnack inequality Lemma \ref{Lem_Moduli_Harnack Ineq}, $(u_+-u_0)(\cdot, R_j)\geq c(S, z_0)j\cdot U(\cdot, R_j)$.  Then by (\ref{Moduli_T+error(U)<0 in one-sided deform}) and maximum principle, \[
			u_+ - u_0 \geq c(S, z_0)j\cdot U, \ \ \ \text{ on }z\in [z_0, R_j].   \]
			This is impossible when $j\to \infty$ since $U>0$ when $z\gg z_0$ and the RHS tends to infinity.
		\end{proof}
		
		By combining Claim 1 with the Parabolic Schauder Estimates in the proof of Lemma \ref{Lem_Moduli_C^2,alpha_star est, interior}, we get that for every $R\gg  4z_0$, \[
		\|u_+-u_0\|_{C^{2,1/2}_\star, S, R}\leq C(u_+, u_0, z_0)R^{\mu_1-\gamma}.   \]
		Therefore by Lemma \ref{Append_Error est II of scT}, \[
		\|\tilde{\scR}(u_+-u_0)\|_{C^{1/2}_\star, S, R} \leq C(S,\gamma)R^{\mu_1+\gamma}.    \]
		By Lemmas \ref{Lem_Moduli_T_S u = f}, \ref{Lem_Moduli_C^0 est} and \ref{Lem_Moduli_C^2,alpha_star est, interior}, there exists $v$ solving $T_S v = z^{-1}\tilde{\scR}(u_+-u_0)$ on $S\times \RR_{\geq 10z_0}$ and satisfying \[
		\|v\|_{C^{2,1/2}_\star, S, R} \leq C(S, u_+, u_0) R^{\mu_1 + \gamma/2}.   \]
		While by the assumption (\ref{Moduli_|u_+-u_0| decay not too fast}) and Claim 1, $w:= u_+-u_0-v$ solves $T_S (w) = 0$ and $\cA\cR_\infty(w)\in [\mu_1+\gamma/2, \mu_1 -\gamma]$.  By Corollary \ref{Cor_Moduli_Sharp asymp for T_S u = 0}, since $[\mu_1+\gamma/2, \mu_1-\gamma]\cap \Gamma(S) = \{\mu_1\}$, we have $\|w-az^{\mu_1}\psi_1\|_{L^2_{\mu_1-\epsilon}, S, 10z_0}<+\infty$ for some $a>0$ and $0<\epsilon<|\gamma|/2$. Let $u_a\in \varpi(a)[u_0]$ be given by Lemma \ref{Lem_Moduli_one-sided transl end}, then the above estimate implies $\|u_+ - u_a\|_{L^2_{\mu_1-\epsilon}}<+\infty$, which could also be upgraded to a $C^0$-decaying estimate by Lemma \ref{Lem_Moduli_C^0 est}.   
		By Theorem \ref{Thm_Moduli of transl end } (ii), we have $u_+\in \varpi(a)[u_0]$. 
	\end{proof}
	
	\begin{Lem} \label{Lem_Pf main thm_fast poly decay => exp decay}
		Let $\beta\in (0, 1)$. There exists $\delta_3(S, \beta)>0 $, and $C(S, \beta)\geq 2$ such that if $z_0\geq C$, $u_\pm \in C^2_{loc}(S\times \RR_{\geq z_0})$ satisfies $\scT(u_\pm) = 0$, $\sup_{R\geq z_0}\|u_\pm\|_{C^2_\star, S, R} \leq \delta <\delta_3$ and 
		\begin{align}
			\liminf_{z\to +\infty} \|(u_+ - u_-)(\cdot, z)\|_{C^0, S} \cdot z^{-\mu_1 + \beta} < +\infty.  \label{Pf main thm_|u_+ - u_-| fast poly decay}
		\end{align}
		Then for every $R\geq 3z_0$ we have, \[
		\|u_+ - u_-\|_{C^{2,1/2}_\sharp, S, R} \leq C(S, \beta)\cdot \delta\cdot e^{-\beta(R-3z_0)}.   \]
	\end{Lem}
	\begin{proof}
		Let $v:= u_+ - u_-$. By Lemma \ref{Append_Error est II of scT}, for $\delta_3(S, \beta)\ll1$, 
		\begin{align}
			0 = \scT(u_+) - \scT(u_-) = -z T_S (v) + \tilde{\scR}(v),  \label{Pf main thm_T_S(u_+-u_-) +error = 0}   
		\end{align}
		where recall $T_S:= \partial_z^2 + \partial_z + z^{-1}L_S$; $\tilde{\scR}$ is some second order linear differential operator, \[
		\tilde{\scR}(v) = \tilde{\cE}_1\cdot z\partial^2_z v + \tilde{\cE}_2\cdot\sqrt{z}\partial_z\nabla v + \tilde{\cE}_3 \cdot\nabla^2 v + \tilde{\cE}_4\cdot z\partial_z v + \tilde{\cE}_5\cdot \nabla v + \tilde{\cE}_6 v,   \]
		with pointwise estimates on the coefficients, for every $1\leq l\leq 6$, \[
		|\tilde{\cE}_l| \leq C(S)(\delta + z_0^{-1}).   \]
		
		For every $\epsilon\in (0, 1)$, consider the barrier functions $W_\epsilon(x, z):= (e^{-\beta(z-z_0)} + \epsilon \cdot(z/z_0)^{\mu_1-\beta/2})\bar{\psi}_1$, where recall $\bar{\psi}_1:= \psi_1/\inf \psi_1$ is the normalized first eigenfunction of $-L_S$. Then we have,
		\begin{align}
			\begin{split}
				-z T_S (W_\epsilon) + \tilde{\scR}(W_\epsilon) & \geq \delta z e^{-\beta (z-z_0)}\cdot \left( \beta - \beta^2 + \frac{\mu_1}{z} -C(S)(\delta + z_0^{-1}) \right)\bar{\psi}_1 \\
				&\ + \epsilon z^{\mu_1-\beta/2}\left(\frac{\beta}{2} - \frac{(2\mu_1-\beta)(2\mu_1-\beta-2)}{4z} - C(S)(\delta + z_0^{-1}) \right)\bar{\psi}_1 \\
				& \geq 0,
			\end{split} \label{Pf main thm_W_epsilon super sol}
		\end{align}
		where the last inequality follows by taking $\delta <\delta_3(S, \beta)\ll1$ and $z_0\geq C(S, \beta)\gg 1$.
		
		By (\ref{Pf main thm_|u_+ - u_-| fast poly decay}), there are $z_j\to +\infty$ such that $\|(u_+ - u_-)(\cdot, z_j)\|_{C^0, S} \leq z_j^{\mu_1-\beta/2}/j$. Then we have for $j>1/\epsilon$, 
		\begin{align*}
			\|v(\cdot, z_0)\|_{C^0, S} \leq \delta \leq \inf_S W_\epsilon(\cdot, z_0); &\ & \|v(\cdot, z_j)\|_{C^0, S} \leq z_j^{\mu_1-\beta/2}/j \leq \inf_S W_\epsilon(\cdot, z_j).
		\end{align*}
		Then by (\ref{Pf main thm_T_S(u_+-u_-) +error = 0}), (\ref{Pf main thm_W_epsilon super sol}) and maximum principle, first let $j\to \infty$ and then $\epsilon\searrow 0$, we get pointwise estimate, \[
		|v(\cdot, z)| \leq C(S)\delta \cdot e^{-\beta(z-z_0)}, \ \ \ \ \ \forall z\geq z_0.   \]
		
		To get $C^{2,1/2}_\sharp$ estimate, first note that since $\scT(u_\pm) = 0$ and $\|u_\pm\|_{C^2_\star, S\times \RR_{\geq z_0}}\leq \delta$, by the same argument as in Lemma \ref{Lem_Moduli_C^2,alpha_star est, interior}, using the parabolic Schauder Estimate \cite{Knerr80_SpacialParabSchauder}, we know that for every $R\geq 2z_0$, \[
		\|u_\pm\|_{C^{2, 1/2}_\star, S, R}\leq C(S)\delta .  \]
		Then by (\ref{Pf main thm_T_S(u_+-u_-) +error = 0}) and Lemma \ref{Append_Error est II of scT}, the proof of elliptic Schauder estimate Lemma \ref{Lem_Moduli_short term Schauder est, int} works here to give \[
		\|v\|_{C^{2,1/2}_\sharp, S, R} \leq C(S)\|v\|_{C^0, S\times [R-1, R+2]} \leq C(S)\delta\cdot e^{-\beta(R-3z_0)}, \ \ \ \ \forall R\geq 3z_0,   \]
		provided $z_0\geq C(S, \beta)\gg 1$.
	\end{proof}

	\section{Complete $\I$-minimizing Translators} \label{Sec_Pf of Main Thm}
	Let $S\subset \RR^n$ be an $(n-1)$-dimensional closed self-shrinker. Recall that translators are critical points of Ilmanen's functional \[
	\I[\Sigma] = \I^{1}[\Sigma]:= \int_\Sigma e^{z}\ d\scH^n.    \] 
	The first goal of this Section is to prove:
	\begin{Thm} \label{Thm_Existence of Transl w prescib end}
		Let $\beta\in (0, 1)$, $S\subset \RR^n$ be a closed smooth embedded self-shrinker, $\varepsilon(z)$ be a decreasing positive function on $\RR_+$ such that $\varepsilon(z)\to 0$ when $z\to +\infty$. Then there exists $\Lambda_0(\beta, S, \varepsilon)\gg 1$ with the following significance:
		
		If $z_0\geq 2$, $\Sigma_e = \ES[w]\cap \RR_{\geq z_0}$ is a simple translating end over $S\times \RR_+$, with \[
		\|w\|_{C^2_\star, S, R} \leq \varepsilon(R),   \]
		for every $R\geq z_0$. Then there exists an $n$-dimensional $\I$-minimizing translator $\Sigma \subset \RR^{n+1}$, possibly with $(n-7)$-dimensional singularities, \textbf{exponentially asymptotic} to $\Sigma_e$ near infinity in the following sense: $\Sigma \cap \RR^n\times \RR_{\geq \Lambda_0 z_0} = \ES[u]$ for some $u\in C^2_{loc}(S\times \RR_{\geq\Lambda_0 z_0})$ satisfying 
		\begin{align}
			\|u - w\|_{C^2_\sharp, S, R}  \leq C(\beta, S, \varepsilon) e^{-\beta (R-\Lambda_0 z_0)},\ \ \ \forall R\geq \Lambda_0 z_0.  \label{Pf main thm_transl exp asymp to ends}   
		\end{align}
		In particular, $\{\Sigma + te_{n+1}\}_{t\in \RR}$ has the tangent flow at $-\infty$ to be $S\times \RR$.
		
		Moreover, if $\Sigma_e$ is invariant with respect to some closed subgroup $G$ of the isometry group of $\RR^n$, then $\Sigma$ can be chosen $G$-invariant as well.
	\end{Thm}
	
	Recall that the norms we are using here are 
	\begin{align*}
		\|u\|_{C^2_\star, \Omega} & := \sup_{(x,z)\in \Omega} \left( |u| + |\nabla_S u| + |z\partial_z u| + |\nabla^2_S u| + |\sqrt{z}\partial_z\nabla_S u| + |z\partial^2_z u| \right); \\
		\|u\|_{C^2_\star, S, R} & := \|u\|_{C^2_\star, S\times [R, 2R]}; \\
		\|u\|_{C^2_\sharp, \Omega} & := \sum_{0\leq k+l\leq 2} \sup_{(x, z)\in \Omega} z^{-l/2}|\partial_z^k \nabla_S^l u| ; \\
		\|u\|_{C^2_\sharp, S, R} & := \|u\|_{C^2_\sharp, S\times [R, R+1]}.
	\end{align*}
	
	To prove Theorem \ref{Thm_Existence of Transl w prescib end}, we first recall the following convergence, that known as the elliptic regularization of mean curvature flow.
	\begin{Lem} \label{Lem_Pf main thm_ellip reg smooth approx MCF}
		Let $\{S_j\}_{j\geq 1}$ be a family of closed hypersurfaces in $\RR^{n}\times\{0\}$ which $C^2$-converges to $S_\infty$; $\epsilon_j\searrow 0$.  Let $P_j\in \mbfI^n(\RR^{n+1})$ be a minimizer of $\I^{\epsilon_j}$ among \[
		\{T\in \mbfI^n(\RR^{n+1}): \partial T = [S_j]\},   \]
		Let $\cM_j:= \{\|P_j(t)\|:= \|P_j\| + (t/\epsilon_j)\partial_z\}_{t\geq 0}$ be the associated mean curvature flow of $P_j$ in $\RR^{n+1}$, $\{S_\infty(t)\}_{t\in [0, T_m)}$ be the mean curvature flow of $S_\infty$ before first singular time $T_m$.
		
		Then as $j\to \infty$,
		\begin{enumerate} [(i)]
			\item $\cM_j$ converges to $\cM_\infty:= \{S_\infty(t)\times \RR\}_{t\in [0, T_m)}$ in $C^\infty_{loc}(\RR^n\times \RR\times (0, T_m))$;
			\item $\spt(\cM_j)\to \spt(\cM_\infty)$ locally in Hausdorff distance in $\RR^n\times \RR_{\geq 0}\times [0, T_m)$. In other words, there exists a compact exhaustion $K_1\subset K_2\subset \dots \nearrow \RR^n\times \RR_{\geq 0}\times [0, T_m)$ and $\epsilon_j\searrow 0$ such that for every $j\geq 1$, 
			\begin{align*}
				\spt(\cM_j)\cap K_j \subset \BB_{\epsilon_j}(\spt(\cM_\infty)), & \ & \spt(\cM_\infty)\cap K_j \subset \BB_{\epsilon_j}(\spt(\cM_j)). 	
			\end{align*}
		\end{enumerate}
	\end{Lem}
	\begin{proof}
		(i) follows from \cite{Ilmanen94_EllipReg} and White's $\epsilon$-regularity Theorem \cite{White05_MCFReg}. (ii) follows from avoidance principle.
	\end{proof}
	
	\begin{proof}[Proof of Theorem \ref{Thm_Existence of Transl w prescib end}.]
		For every $R> 2z_0$, let $\mathfrak{T}_R(S, \varepsilon, z_0)$ be the space of pairs $(\Sigma_e, \Sigma_R)$ such that,
		\begin{enumerate}[(i)]
			\item $\Sigma_e = \ES[w]\cap \RR^n\times \RR_{\geq z_0}$ is a simple translating end over $S\times \RR_+$ with \[
			\|w\|_{C^2_\star, S, R} \leq \varepsilon(R), \ \ \ \ \ \forall R\geq z_0;   \]
			\item $\Sigma_R$ is the support of an integral current $[\Sigma_R]$ with boundary to be the $R$-slice of $\Sigma_e$, i.e. $\partial [\Sigma_R] = [\Sigma_e\cap \RR^n\times \{R\}]$; moreover, such that $[\Sigma_R]$ is $\I$-minimizing among \[
			\left\{ T\in \mbfI_n(\RR^{n+1}): \partial T = [\Sigma_e\cap \RR^n\times \{R\}] \right\}.   \]
		\end{enumerate}
		For each simple translating end $\Sigma_e$ as in the Theorem and each $R>2z_0$, the existence of $(\Sigma_e, \Sigma_R)\in \mathfrak{T}_R(S, \varepsilon, z_0)$ is derived by minimizing $\I$-functional and taking limit using Federer-Fleming Compactness Theorem \cite{Simon83_GMT}. Our principal goal is to show that, when sending $R\nearrow +\infty$, $\Sigma_R$ will stay close to $\Sigma_e$ outside a uniform large compact subset, and therefore the subsequential limit of $\Sigma_R$ will be a desired translator exponentially asymptotic to $\Sigma_e$. The way to achieve this is inspired by the work of Chan \cite{Chan97} who constructed some $I$-parameter family of minimizing hypersurfaces asymptotic to a given strictly minimizing hypercone.
		
		For each pair $(\Sigma_e, \Sigma_R)\in \mathfrak{T}_R(S, \varepsilon, z_0)$ and every $\delta>0$, define 
		\begin{align*}
			\rho(\Sigma_R, z_0, \delta) := \inf\big\{\rho\in (z_0, R]: &\ \Sigma_R\cap \RR^n\times\{z\} \subset \sqrt{z}\BB^n_\delta(S)\times\{z\}, \ \ \ \forall \rho\leq z\leq R; \\
			&\ \Sigma_R\cap \RR^n\times [\rho, R/2] = \ES[u_R] \text{ with }\|u_R\|_{C^2_\star, S\times [\rho, R/2]}\leq \delta \big\}, 
		\end{align*} 
		which is the smallest height above which $\Sigma_R$ is $\delta$-close to $\ES$.
		We define \[
		\rho_R(S, \varepsilon, z_0, \delta):= \sup\{\rho(\Sigma_R, z_0, \delta): (\Sigma_e, \Sigma_R)\in \mathfrak{T}(S, \varepsilon, z_0)\}.   \]
		\textbf{Claim 1.} For every $\varepsilon$, $z_0\geq 2$ and $\delta>0$, \[
		\limsup_{R\to +\infty} \rho_R(S, \varepsilon, z_0, \delta)/R = 0.   \]
		\begin{proof}[Proof of Claim 1.] 
			Suppose for contradiction that there exist some $z_0\geq 2$, $\delta>0$, $R_j\nearrow +\infty$ and $(\Sigma^j_e, \Sigma^j_{R_j})\in \mathfrak{T}(S, \varepsilon, z_0)$ such that $\rho_j/R_j \to c>0$ as $j\to \infty$, where $\rho_j:= \rho(\Sigma^j_{R_j}, z_0, \delta)$.
			
			Consider the rescaled mean curvature flow
			\begin{align*}
				\left\{\tilde{\Sigma}_{R_j}^j(\tau) := \frac{e^{\tau/2}}{\sqrt{R_j}}\left(\Sigma_{R_j}^j - R_je^{-\tau} e_{n+1} \right)\right\}_{\tau \geq 0},  \ \ \text{ defined on }\bigcup_{\tau\geq 0}\RR^n\times \left(-\infty, \sqrt{R_j}(e^{\tau/2}-e^{-\tau/2})\right) \times \{\tau\}; 
			\end{align*}
			Note that $\{\tilde{\Sigma}_{R_j}^j(\tau)\}_{\tau\geq 0}$ is the associated RMCF of $\tilde{\Sigma}_{R_j}^j(0)$, which is a minimizer of $\I^{1/R_j}$.  And since \[
			\partial \tilde{\Sigma}_{R_j}^j(0) = \left\{(x/\sqrt{R_j}, 0): (x, R_j)\in \Sigma_e^j\right\},   \]
			which is $\varepsilon(R_j)$-$C^2$-close to $S\times \{0\}$ by definition. We then know from Lemma \ref{Lem_Pf main thm_ellip reg smooth approx MCF} that when $R_j\to \infty$, $\{\tilde{\Sigma}_{R_j}^j(\tau)\}_{\tau\geq 0}$ $C^\infty_{loc}$-converges to the static rescaled mean curvature flow $\{S\times \RR\}$ on $\tau\geq \log 2$ and locally distance converges to $\{S\times \RR\}$ on $\RR^n\times \RR_{\geq 0}\times [0, +\infty)$. But this in turn implies $\rho_j/R_j \to 0$, which is a contradiction.
		\end{proof}
		\noindent   \textbf{Claim 2.} There exist $\delta_4(S, \varepsilon)\in (0, 1)$ and $\Lambda_1(S, \varepsilon)>1$ such that for every $0<\delta< \delta_4(S, \varepsilon)$, $z_0\geq 2$ and $R > 2\cdot max\{z_0, \rho_R(S, \varepsilon, z_0, \delta), \Lambda_1\}$. Let $\bar{\rho}_R:= max\{\rho_R(S, \varepsilon, z_0, \delta), \Lambda_1\}$ and let $(\Sigma_e, \Sigma_R)\in \mathfrak{T}_R(S, \varepsilon, z_0)$ such that 
		\begin{align*}
			\Sigma_e\cap \RR^n\times \RR_{\geq z_0} =: \ES[w], &\ & \Sigma_R\cap \RR^n\times [\bar{\rho}_R, R/2] =: \ES[u_R].
		\end{align*}
		Let $\bar{\psi}_1(x):= \psi_1(x)/\inf_S \psi_1 > 0$ be the normalized first eigenfunction of $-L_S$.  Then we have,
		\begin{align}
			|(u_R-w)(x, z)| & \leq \left(\delta + \varepsilon(\bar{\rho}_R)\right) \cdot(z/\bar{\rho}_R)^{\mu_1-1}\bar{\psi}_1(x), & \ & \forall (x, z)\in S\times (\bar{\rho}_R, R/2);  \label{Pf main thm_Equ_fast polyn decay est}
		\end{align}
		\begin{proof}[Proof of Claim 2.]
			Consider for every $\rho\in [\bar{\rho}_R, R]$, the barrier functions \[
			W^\pm_{\rho, \delta}(x, z):= w(x, z)\pm \left(\delta + \varepsilon(\bar{\rho}_R)\right)\cdot \left(z/\rho\right)^{\mu_1-1}\cdot\bar{\psi}_1(x).   \]
			Note that we have for every $r\geq \rho$, \[
			\|W^\pm_{\rho, \delta}\|_{C^2_\star, S, r} \leq \|w\|_{C^2_\star, S, r} + C(S)\left(\delta + \varepsilon(\bar{\rho}_R)\right) \leq \kappa(S),   \]
			where $\kappa(S)\ll 1$ is a constant TBD, and since by definition $\|w\|_{C^2_\star, S, r}\leq \varepsilon(r) \to 0$ when $r\to +\infty$, the last inequality follows by taking $r\geq \bar{\rho}_R \geq \Lambda_1(S, \varepsilon, \kappa)\gg 1$ and $\delta<\delta_4(S, \varepsilon, \kappa)\ll1$.   By Lemma \ref{Append_Error est II of scT}, we have pointwise estimate on $S\times \RR_{\geq\rho}$,
			\begin{align*}
				|\scR(W^\pm_{\rho, \delta}) - \scR(w)| \leq C(S)z^{-1}\kappa(S)\cdot |W^\pm_{\rho, \delta} - w|.
			\end{align*}
			Hence, by $\scT(w) \equiv 0$, we have,
			\begin{align}
				\begin{split}
					\pm\scT(W^\pm_{\rho, \delta}) & = - z(\partial_z^2 + \partial_z + z^{-1}L_S)\left( \left(\delta + \varepsilon(\bar{\rho}_R)\right)\cdot \left(z/\rho\right)^{\mu_1-1}\cdot\bar{\psi}_1(x) \right) \pm z\cdot \left(\scR(W^\pm_{\rho,\delta}) - \scR(w)\right) \\
					& \geq \left(\delta + \varepsilon(\bar{\rho}_R)\right)\rho^{1-\mu_1}\left(1 - \frac{(\mu_1-1)(\mu_1-2)}{z} \right)z^{\mu_1-1}\bar{\psi}_1 \\
					&\  - C(S)\kappa(S)\left(\delta + \varepsilon(\bar{\rho}_R)\right)\cdot (z/\rho)^{\mu_1-1}\bar{\psi}_1 \\
					& > 0.
				\end{split}  \label{Pf main thm_Equ_1st barrier functions}
			\end{align}
			where the last inequality follows by taking $\kappa(S)\ll1$ and $z\geq \bar{\rho}_R \geq \Lambda_1(S, \varepsilon, \kappa)\gg 1$.
			Under Fermi coordinates of $S\subset \RR^n$, let 
			\begin{align*}
				\Omega_{\rho, \delta}:= \{(\sqrt{z}x, z): z\geq\rho; &\ x = (x_0, y)\in \RR^n \text{ with } x_0\in S,\\
				&\ |y-w(x_0, z)|\leq \left(\delta + \varepsilon(\bar{\rho}_R)\right) \cdot(z/\rho)^{\mu_1-1}\bar{\psi}_1(x_0)\}     
			\end{align*}
			be the neighborhood of $\Sigma_e$ on $\RR^n\times \RR_{>\rho}$ bounded by $\ES[W^\pm_{\rho,\delta}]$.  By (\ref{Pf main thm_Equ_1st barrier functions}), $\partial\Omega_{\rho,\delta}\cap \RR^n\times \RR_{>\rho}$ has translator-mean-curvature pointing inward on $\RR^n\times \RR_{>\rho}$. By definition, $\Omega_{\rho, \delta}\cap \RR^n\times \{\rho\} \supset \sqrt{\rho}\BB_\delta (S)\times\{\rho\}$.  Thus by the strong maximum principle \cite{SolomonWhite89_Maxim}, \[
			\inf\{\rho\in [\bar{\rho}_R, R]: \Sigma_R\cap \RR^n\times \RR_{>\rho}\subset \Omega_{\rho,\delta}\} = \bar{\rho}_R.   \]
			Therefore, $u_R$ satisfies the decay estimate (\ref{Pf main thm_Equ_fast polyn decay est}).
		\end{proof}
		
		\noindent  \textbf{Claim 3.} For every $\delta\in (0, \delta_4)$, there exists $\Lambda_2(S, \varepsilon, \delta)<+\infty$ such that \[
		\sup_{z_0\geq 2}\limsup_{R\to +\infty} \frac{\rho_R(S, \varepsilon, z_0, \delta)}{z_0} < \Lambda_2(S, \varepsilon, \delta).   \]
		\begin{proof}[Proof of Claim 3.]
			It suffices to show that the LHS is finite.  Suppose for contradiction, there exist $\delta\in (0,\delta_4)$, $z_j\geq 2$ and $R_j\to +\infty$ such that $\rho_{R_j}(S, \varepsilon, z_j, \delta)>j z_j$.  Then there exist pairs $(\Sigma_e^j, \Sigma^j_{R_j})\in \mathfrak{T}_{R_j}(S, \varepsilon, z_j)$ such that $\rho_j:= \rho(\Sigma^j_{R_j}, z_j, \delta)\geq j z_j$. In particular, $\bar{\rho}_j := \max\{\rho_j, \Lambda_1(S, \varepsilon)\} = \rho_j \to +\infty$ as $j\to \infty$.  Denote as before, 
			\begin{align*}
				\Sigma_e^j\cap \RR^n\times \RR_{\geq z_j} =: \ES[w_j], &\ & \Sigma_{R_j}^j\cap \RR^n\times [\rho_j, R_j/2] =: \ES[u_{R_j}].
			\end{align*}
			
			Consider for every $T>1$, the rescaled mean curvature flows 
			\begin{align*}
				\{\bar{\Sigma}^j(\tau) & := \cR_{\tau - \log(\rho_j)} (\Sigma_{R_j}^j)\}_{\tau \in (-T, T]} \ \ \ \text{ defined on }\RR^n\times (-\infty, (R_j\cdot e^{-T/2} - \rho_j\cdot e^{T/2})/\sqrt{\rho_j}); \\
				\{\bar{\Sigma}_e^j(\tau) & := \cR_{\tau - \log(\rho_j)} (\Sigma_e^j)\}_{\tau \in (-T, T]} \ \ \ \ \ \text{ defined on }\RR^n\times (\frac{z_j\cdot e^{T/2}}{\sqrt{\rho_j}} - e^{-T/2}\sqrt{\rho_j}, +\infty),
			\end{align*}
			where $\cR_\tau (\Sigma):= e^{\tau/2}(\Sigma - e^{-\tau} e_{n+1})$. Note that by Claim 1 and the contradiction assumption, $\rho_j/R_j + z_j/\rho_j \to 0$ when $j\to \infty$, thus the region where both flows are defined exhausts $\RR^n\times \RR$.  Also notice that for every $\tau\in (-T, 0]$,
			\begin{align*}
				&\ \bar{\Sigma}^j(\tau) \cap \RR^n \times \left[\sqrt{\rho_j}(e^{\tau/2} -e^{-\tau/2}), (R_j\cdot e^{-T/2}/2 - \rho_j\cdot e^{T/2})/\sqrt{\rho_j} \right)\\
				= &\ \left\{\left(\sqrt{1+ \frac{e^{\tau/2}\hat{z}}{\sqrt{\rho_j}}}\cdot\left( x+ u_{R_j}(x, z) \nu_x \right), \hat{z}\right): \hat{z}\in \left[\sqrt{\rho_j}(e^{\tau/2} -e^{-\tau/2}), \frac{R_j\cdot e^{-T/2} - 2\rho_j\cdot e^{T/2}}{2\sqrt{\rho_j}} \right) \right\}; \\
				&\ \bar{\Sigma}^j_e(\tau) \cap \RR^n \times \left(\frac{z_j\cdot e^{T/2}}{\sqrt{\rho_j}} - e^{-T/2}\sqrt{\rho_j}, +\infty \right) \\
				= &\ \left\{\left(\sqrt{1+ \frac{e^{\tau/2}\hat{z}}{\sqrt{\rho_j}}}\cdot\left( x+ w_j(x, z) \nu_x \right), \hat{z}\right): \hat{z}> \frac{z_j\cdot e^{T/2}}{\sqrt{\rho_j}} - e^{-T/2}\sqrt{\rho_j} \right\},
			\end{align*}
			where $z:= e^{-\tau}\rho_j + e^{-\tau/2}\sqrt{\rho_j}\hat{z}$ and $\nu_x$ is the unit normal field of $S$ at $x$.
			For every $\tau < 0$ and every $\hat{z}\in \RR$, 
			\begin{align*}
				&\ \limsup_{j\to \infty} dist_H(\bar{\Sigma}^j(\tau)\cap \RR^n\times \{\hat{z}\}, \bar{\Sigma}_e^j(\tau)\cap \RR^n\times \{\hat{z}\}) \\
				\leq &\ \limsup_{j\to \infty} C(S)\|(u_{R_j}-w_j)(\cdot, e^{-\tau}\rho_j + e^{-\tau/2}\sqrt{\rho_j}\hat{z})\|_{C^0, S}\\ 
				\leq &\ \limsup_{j\to \infty} C(S, \varepsilon, \beta)\cdot \left(\frac{e^{-\tau}\rho_j + e^{-\tau/2}\sqrt{\rho_j}\hat{z}}{\rho_j}\right)^{\mu_1-1} \leq C(S, \varepsilon, \beta)e^{-(\mu_1-1)\tau},
			\end{align*}
			where the second inequality follows from (\ref{Pf main thm_Equ_fast polyn decay est}).
			Since $w_j$ are uniformly bounded by $\varepsilon$, we know that $\bar{\Sigma}_e^j(\tau) \to S\times \RR$ in $C^\infty_{loc}(\RR^{n+1})$, we thus have by Lemma \ref{Lem_Pre_blow down transl split}, when $j\to \infty$ and $T\nearrow +\infty$, $\{\bar{\Sigma}^j(\tau)\}_{\tau}$ subconverges (in the Brakke sense) to some rescaled mean curvature flow $\{\hat{\Sigma}(\tau)\times \RR\}_{\tau\in \RR}$, where $\hat{\Sigma}(\tau) = \graph_S(\hat{v}(\cdot, \tau))$ satisfies \[
			\|\hat{v}(\cdot, \tau)\|_{C^0(S)} \leq C(S, \varepsilon, \beta)e^{-(\mu_1-1)\tau},\ \ \ \forall \tau\ll 0.   \]
			By analogy with \cite{CCMS20_GenericMCF}, the only such ancient rescaled mean curvature flow is the static $\{S\times \RR\}_\tau$, hence by Brakke regularity, $\{\bar{\Sigma}^j(\tau)\}_{\tau\in [-1,1]}$ $C^\infty_{loc}$-subconverges to $\{S\times \RR\}_\tau$ on $\RR^{n+1}$.  This mean when $j\to \infty$, \[
			\|u_{R_j}\|_{C^2_\star, S\times [\rho_j/2, 2\rho_j]} \to 0,   \] 
			which violates the definition of $\rho_j = \rho(\Sigma_{R_j}^j, z_j, \delta)$.
			\end{proof}
			By Claims 1-3, for each $\delta = (0, \delta_4)$, $z_0\geq 2$, and each $\Sigma_e = \ES[w]\cap \RR^n\times \RR_{\geq z_0}$ as in the Theorem and every $R\gg 1$, we can minimize $\I$-functional to find $(\Sigma_e, \Sigma_R)\in \mathfrak{T}_R(S, \varepsilon, z_0)$ such that \[
			\Sigma_R\cap \RR^n\times [\Lambda_2 \cdot z_0, R/2] = \ES[u_R],   \]
			and,
			\begin{align*}
				|(u_R - w)(\cdot, z)| & \leq (\delta + \varepsilon(\Lambda_2 z_0))\cdot\left(\frac{z}{\Lambda_2 z_0}\right)^{\mu_1-1}\cdot\bar{\psi}_1, \ \ \ \ \ \forall z\in [\Lambda_2 z_0, R/2]; \\
				\|u_R\|_{C^2_\star, S\times [\Lambda_2 z_0, R/2]} & \leq \delta.
			\end{align*}
			Take $R\to +\infty$, by Federer-Fleming Compactness Theorem \cite{Simon83_GMT}, $[\Sigma_R]$ flat-converges to some $\I$-minimizing boundary $[\Sigma]$, where $\Sigma \cap \RR^n\times \RR_{\geq \Lambda_2 z_0} = \ES[u]$ for some $u\in C^2_{loc}(S\times \RR_{\geq \Lambda_2 z_0})$ satisfying,
			\begin{align*}
				|(u - w)(\cdot, z)| & \leq (\delta + \varepsilon(\Lambda_2 z_0))\cdot\left(\frac{z}{\Lambda_2 z_0}\right)^{\mu_1-1}\cdot\bar{\psi}_1, \ \ \ \ \ \forall z\in [\Lambda_2 z_0, +\infty); \\
				\|u\|_{C^2_\star, S\times [\Lambda_2 z_0, +\infty)} & \leq \delta.
			\end{align*}
			Then the exponentially decaying estimate (\ref{Pf main thm_transl exp asymp to ends}) is a direct consequence of Lemma \ref{Lem_Pf main thm_fast poly decay => exp decay}, by taking $\Lambda_0(S, \varepsilon, \beta)\gg \Lambda_2$ and fix the choice of $0<\delta(S, \varepsilon, \beta) \ll \delta_4(S, \varepsilon)$. 
		\end{proof}

		In general, we call a translator $\Sigma$ \textbf{exponentially asymptotic} to a simple translating end $\ES[w]$ over $S\times \RR_+$ if $\Sigma\cap \RR^n\times \RR_{\geq z_0} = \ES[u]$ for some $z_0\gg 1$, and \[
		\limsup_{R\to +\infty}\|u-w\|_{C^2_\sharp, S, R}\cdot e^{R/2} = 0.   \]
		As a refinement of Theorem \ref{Thm_Existence of Transl w prescib end}, we establish the following dichotomy of translators that are exponentially asymptotic to a pl-simple end.
		\begin{Thm} \label{Thm_Pf main thm_Transl region}
			Let $\Sigma_e\subset \RR^{n+1}$ be a pl-simple translating end over $S\times \RR_+$.  Then there exists a closed subset $T[\Sigma_e]\subset \RR^{n+1}$ which contains all $\I$-minimizing translators exponentially asymptotic to $\Sigma_e$. And one of the following holds,
			\begin{enumerate}[(i)]
				\item $T[\Sigma_e]$ is the support of an $\I$-minimizing translator (and hence is the unique $\I$-minimizing translator exponentially asymptotic to $\Sigma_e$);
				\item $\Int(T[\Sigma_e])\neq \emptyset$ and $\partial T[\Sigma_e] = T^+ \sqcup T^-$, where each of $T^\pm$ is the support of an $\I$-minimizing translator exponentially asymptotic to $\Sigma_e$.
			\end{enumerate}
			
			Also, both $T[\Sigma_e]$ in case (i) and $T^\pm$ in case (ii) satisfy the uniform estimate (\ref{Pf main thm_transl exp asymp to ends}) provided $\Sigma_e$ satisfies the assumption in Theorem \ref{Thm_Existence of Transl w prescib end}. In particular, for a fixed $z_0\gg 1$, if $\{\Sigma_\varphi = \ES[u_{\varphi, z_0}]\}_{\varphi\in \BB^I}$ is the family of pl-simple end constructed in Theorem \ref{Thm_Moduli of transl end }, then $T[\Sigma_\varphi]$ is upper-semi-continuous in $\varphi$ in the following sense: if $\varphi_j\to \varphi_\infty$ in $\BB^I$, then there are $R_j\to +\infty$ such that \[
			\BB_{1/R_j}(T[\Sigma_{\varphi_\infty}]) \supset T[\Sigma_{\varphi_j}]\cap \BB_{R_j}.  \]
			
			Moreover, let $\varpi$ be the $\RR$-action on the space of pl-simple ends introduced in Section \ref{Subsec_Moduli_RR action}, then $\{T[\varpi(a)[\Sigma_e]]\}_{a\in \RR}$ is a decomposition of $\RR^{n+1}$, i.e. 
			\begin{align}
				\RR^{n+1} = \coprod_{a\in \RR} T[\varpi(a)[\Sigma_e]].   \label{Pf main thm_Equ_R^n+1 decomp into T_a}
			\end{align}
		\end{Thm}
		
		For simplicity, for every pl-simple translating end $\Sigma_e$ over $S\times \RR_+$, we write $\cT[\Sigma_e]$ to be the set of all $\I$-minimizing translators exponentially asymptotic to $\Sigma_e$.  Note that since $S$ is closed and connected, any such $T\in \cT[\Sigma_e]$ has connected support, and there's a unique connected component $\Omega_T$ of $\RR^{n+1}\setminus \spt(T)$, referred to as the \textbf{inner region} bounded by $T$, such that $\Omega_T\cap \RR^n\times \{z\}$ is bounded for $z\gg 1$. Here we abuse the notation to identify an integral cycle $T$ with its support, since everything is multiplicity one and smooth away from a low-dimensional closed subset.
		
		\begin{proof}[Proof of Theorem \ref{Thm_Pf main thm_Transl region}.] We start with an observation that disjoint ends have disjoint $\I$-minimizing translators exponentially asymptotic to them.\\
			\textbf{Claim 1.} For every pl-simple end $\Sigma_e\in \cE_p(S)$ and every real number $a>0$, if $T_0\in \cT[\Sigma_e]$ and $T_a\in \cT[\varpi(a)[\Sigma_e]]$, then $\Omega_{T_a} \supset \Clos(\Omega_{T_0})$.
			\begin{proof}[Proof of Claim 1.]
				Suppose $\Sigma_e \cap \RR^n\times \RR_{\geq z_0}= \ES[w_0]$ and $w_a$ is given by Lemma \ref{Lem_Moduli_one-sided transl end} so that $\ES[w_a]\in \varpi(a)[\Sigma_e]$. Then by the definition of exponentially asymptotic, $T_0 \cap \RR^n\times \RR_{\geq z_0'} =: \ES[u_0]$ and $T_a \cap \RR^n\times \RR_{\geq z_0'} =: \ES[u_a]$ satisfy 
				\begin{align*}
					&\ |(u_a - u_0)(\cdot, z)\cdot z^{-\mu_1} - a\psi_1| \\
					= &\ |(w_a - w_0)(\cdot, z)\cdot z^{-\mu_1} - a\psi_1| + |(u_0 - w_0)(\cdot, z)\cdot z^{-\mu_1}| + |(u_a - w_a)(\cdot, z)\cdot z^{-\mu_1}| \to  0.          
				\end{align*}
				And since $a>0$ and $\psi_1>0$, this implies $u_a > u_0$ for $z\geq z_0''\gg 1$, which means $\Clos(\Omega_{T_0})\cap \RR^n\times \RR_{\geq z_0''} \subset \Omega_{T_a}$.  But since $\partial \Omega_{T_0}$ and $\partial \Omega_{T_a}$ are both $\I$-minimizing, by comparing $\I$-functional and strong maximum principle \cite{SolomonWhite89_Maxim, Ilmanen96}, this implies $\Clos(\Omega_{T_0})\subset \Omega_{T_a}$.
			\end{proof}
			
			Now let $\Sigma_e\in \cE_p(S)$ be a pl-simple end in the Theorem. For every $a\in \RR$, we pick $T_a\in \cT[\varpi(a)[\Sigma_e]]$ to be an $\I$-minimizing translator constructed in Theorem \ref{Thm_Existence of Transl w prescib end}, and let $\Omega_a:= \Omega_{T_a}$ be the inner region bounded by $T_a$. Recall that Claim 1 guarantees that $\{\Omega_a\}_{a\in \RR}$ is monotone increasing in $a$. 
			
			Define the closed subset \[
			\Omega^+ = \Omega^+[\Sigma_e]:= \bigcap_{a>0} \Clos(\Omega_a) = \bigcap_{a>0} \Omega_a.  \]
			First notice that by Claim 1, $\Omega^+[\Sigma_e]$ is independent of the choice of $T_a$. Also, $\Omega^+$ is the limit of $\Omega_a$ in the measure sense, hence has $\I$-minimizing boundary by the compactness of minimizing boundary \cite{Simon83_GMT}. And since we have $C^2$-estimates for the graphical functions of $\partial \Omega_a = T_a$ over $\ES$ outside a uniform compact subset in Lemma \ref{Lem_Moduli_one-sided transl end} and Theorem \ref{Thm_Existence of Transl w prescib end}, combined with Lemma \ref{Lem_Pf main thm_fast poly decay => exp decay}, we have $\partial \Omega^+\in \cT[\Sigma_e]$. Moreover, if $\Sigma_e$ satisfies the assumption in Theorem \ref{Thm_Existence of Transl w prescib end}, then $\partial\Omega^+$ satisfies the uniform estimate (\ref{Pf main thm_transl exp asymp to ends}).
			Similarly, \[
			\Omega^- = \Omega^-[\Sigma_e]:= \bigcup_{a<0} \Omega_a,  \]
			is an open subset with connected $\I$-minimizing boundary, which falls in $\cT[\Sigma_e]$. 
			
			Let $T[\Sigma_e]:= \Omega^+[\Sigma_e]\setminus \Omega^-[\Sigma_e]$. Since $\Omega^- \subset \Omega^+$, by strong maximum principle \cite{SolomonWhite89_Maxim, Ilmanen96}, either $\partial\Omega^- = \partial\Omega^+$, in which case (i) holds in Theorem \ref{Thm_Pf main thm_Transl region}, or $\Clos(\Omega^-) \subset \Int(\Omega^+)$, which means (ii) holds. Here we used the fact that $\partial\Omega_{\pm}$ are both connected. This is because $\partial\Omega_{\pm}$ are $\mathcal I$-minimizing and asymptotic to a connected end. Again by Claim 1, for every $T\in \cT[\Sigma_e]$, $T\subset \Omega^+$ and $T\cap \Omega^- = \emptyset$, which means $T\subset T[\Sigma_e]$. The upper-semi-continuity of $T[\Sigma_\varphi]$ in $\varphi$ follows directly from (\ref{Pf main thm_transl exp asymp to ends}), Theorem \ref{Thm_Moduli of transl end } (ii) and that $T[\Sigma_\varphi]$ contains all $\I$-minimizing translator exponentially asymptotic to $\Sigma_\varphi$.
			
			To prove (\ref{Pf main thm_Equ_R^n+1 decomp into T_a}), first by Claim 1, $\{T[\varpi(a)[\Sigma_e]]\}_{a\in \RR}$ are pairwise disjoint; And by its definition, \[
			\bigcup_{a\in \RR} T[\varpi(a)[\Sigma_e]] = \bigcup_{a>0} \Omega_a \setminus \bigcap_{a<0} \Omega_a.    \]
			Therefore, it suffices to show that $\Omega_{+\infty} := \bigcup_{a>0} \Omega_a = \RR^{n+1}$ and $\Omega_{-\infty}:= \bigcap_{a<0} \Omega_a = \emptyset$. In the following, we shall only prove one of them, while the other one is similar.
			
			Suppose for contradiction that $\Omega_{+\infty} \neq \RR^{n+1}$, then again by the compactness theorem of minimizing boundary \cite{Simon83_GMT}, $\partial\Omega_{+\infty}$ is an $\I$-minimizing translator in $\RR^{n+1}$.  By Lemma \ref{Lem_Pre_blow down transl split}, suppose $\{\sqrt{-t}V\}_{t<0}$ is a tangent flow of $\partial\Omega_{+\infty}$ at $-\infty$, where $V$ splits in $z$-direction.  Since $\Omega_{+\infty} \supset \Omega_1$, whose boundary has a simple end over $S\times \RR_+$, we then know that the self-shrinking integral varifold $V$ is supported in $\Clos(D)\times \RR$, where $D$ is the unbounded component of $\RR^n\setminus S$. By strong maximum principle \cite{SolomonWhite89_Maxim}, $V = m|S\times \RR| + V_+$ for some $m\in \ZZ_{\geq 1}$ and $V_+$ is another self-shrinking integral varifold supported in $D\times \RR$.  By the avoidance principle of MCF, $V_+ = 0$ (see also \cite[Appendix C]{CCMS20_GenericMCF}).
			Also, by Corollary \ref{Cor_Pre_Blow down transl MCF split} and Lemma \ref{Lem_Pre_Transl w simple end has entropy finite}, the entropy $\lambda[\partial \Omega_a] = \lambda[S\times \RR]$ for every $a\in \RR$.  Then the entropy of $V$ satisfies \[
			m\lambda[S\times\RR] = \lambda[V] = \lambda[\partial \Omega_{+\infty}] \leq \liminf_{a\to +\infty}\lambda[\partial\Omega_a] = \lambda[S\times \RR].   \] 
			This forces $m=1$, and hence $\partial\Omega_{+\infty}$ is also an $\I$-minimizing translator with simple end over $S\times \RR_+$. Then by Lemma \ref{Lem_Moduli_General one-sided transl end}, there exists $a\geq 0$ such that $\partial \Omega_{+\infty} \in \varpi(a)[\Sigma_e]$.  However, by Claim 1, this means $\partial \Omega_{+\infty} \subset \Int(\Omega_{a+1})\subset \Int(\Omega_{+\infty})$, which is impossible.  This finishes the proof of $\Omega_{+\infty} = \RR^{n+1}$.   
		\end{proof}

		\begin{Def}
			Case (ii) in Theorem \ref{Thm_Pf main thm_Transl region} is called \textbf{fattening}. In contrast, case (i) in Theorem \ref{Thm_Pf main thm_Transl region} is called \textbf{non-fattening}
		\end{Def}
		
		In Section \ref{Sec_Fattening}, we will show that fattening indeed happens. Nevertheless, the following Corollary suggests that fattening is a rare phenomenon. 
		
		\begin{Cor} \label{Cor_Pf main thm_generic nonfattening}
			For a fixed $z_0\geq C_0(S)$, let $\{\Sigma_\varphi:=\ES[u_{\varphi, z_0}]\}_{\varphi\in \BB^I}$ be the family of pl-simple end constructed in Theorem \ref{Thm_Moduli of transl end }. Then there exists a meager subset $Z\subset \BB^I$ such that $\forall \varphi\in \BB^I\setminus Z$, $T[\Sigma_\varphi]$ is non-fattening.
		\end{Cor}
		\begin{proof}
			Define \[
			Z_j:= \{\varphi\in \BB^I: T[\Sigma_\varphi]\cap \overline{\BB_j^{n+1}} \text{ contains a ball of radius }1/j \}.  \]
			Then by upper-semi-continuity of $T[\Sigma_\varphi]$, $Z_j\subset \BB^I$ is a closed subset.  Moreover, let $\varpi$ be the restriction on $\BB^I$ of the $\RR$-action introduced in Section \ref{Subsec_Moduli_RR action}.  By (\ref{Pf main thm_Equ_R^n+1 decomp into T_a}), for every $\varphi\in \BB^I$, $\{a\in \RR: \varpi(a)[\varphi]\in Z_j\}$ is a countable subset of $\RR$. Hence by Corollary \ref{Cor_Moduli_RR-action conti}, $Z_j$ has no interior, and $Z:= \bigcup_{j\geq 1}Z_j$ is the desired meager subset.
		\end{proof}
		
		Since the complete translators constructed in Theorem \ref{Thm_Pf main thm_Transl region} have finite entropy, they can possibly arise as the type-II blow-up limit of mean curvature flows of closed hypersurfaces. We propose the following conjecture.
		
		\begin{Con}
			The complete translators constructed in Theorem \ref{Thm_Pf main thm_Transl region} are the type-II blow-up limit of mean curvature flows of closed hypersurfaces.
		\end{Con}
		
		In \cite{AAG95_RotationMCF, AV97_DegenerateNeckpinches}, there are explicit examples of mean curvature flow of closed surfaces that have type-II singularities, and the type-II blow-up limit is the bowl soliton. This implies that if the self-shrinker $S$ is the round sphere $\SSp^2_2$, then the translators in Theorem \ref{Thm_Pf main thm_Transl region} can arise as the type-II blow-up limit of mean curvature flows of closed hypersurfaces.

		\section{Topology Change and Fattening}\label{Sec_Fattening}
		
		Theorem \ref{Thm_Existence of Transl w prescib end} establishes the existence of a complete embedded translator that is asymptotic to a given simple translating end. Meanwhile, Theorem \ref{Thm_Pf main thm_Transl region} suggests a possible phenomenon that is called fattening, and the fattening implies non-uniqueness of the complete embedded translators that are exponentially asymptotic to the same simple end.
		
		In this section, we show that fattening can indeed happen, and it is due to the topology change of a family of mean curvature flows. Recall that the Angenent torus $S_A\subset \RR^3$ introduced in \cite{Angenent92_Doughnut} is a closed rotationally symmetric self-shrinker, and topologically it is a torus.  From Theorem \ref{Thm_Moduli of transl end } and Lemma \ref{Lem_Moduli_one-sided transl end}, there exists $\Sigma_{e}\subset \RR^4$ that is a rotationally symmetric simple translating end over $S_A\times \RR_+$, and all its one-sided deformations $T[\varpi(a)[\Sigma_e]]$ are also rotationally symmetric. In the following, we fix such an end.
		
		\begin{Thm}\label{Thm_Angenent torus fattens}
			There exists $a\in\RR$ such that $T[\varpi(a)[\Sigma_{e}]]$ fattens. In other words, case (ii) in Theorem \ref{Thm_Pf main thm_Transl region} happens when $S = S_A$.
		\end{Thm}
		
		To prove the theorem, we need the following lemma concerning certain rescaled limits of one-sided deformations of translators.   Let $S\subset \RR^n$ be an arbitrary closed embedded smooth self-shrinker.  $\RR^n\setminus S = E_+ \sqcup E_-$, where $E_+$ is unbounded.  For every constant $0<\kappa\leq \kappa_S\ll1$, the argument in \cite{CCMS20_GenericMCF} showed the existence of a unique smooth ancient RMCF $\{S^\kappa_\pm(\tau)\}_{\tau\leq 0}$ asymptotic to $S$ near $-\infty$ such that $S^\kappa_\pm (\tau) \subset E_\pm$, $\forall \tau \leq 0$, and the Hausdorff distance $dist_{\RR^n}(S, S^\kappa_\pm (0)) = \kappa$.

		\begin{Lem} \label{Lem_Fattening_Rescaled lim of one-sided deform}
			Let $\ES[u_0]\subset \RR^{n+1}$ be a pl-simple translating end over $S\times \RR_+$, $\kappa\in (0, \kappa_S]$.  For $a\in \RR$, $u_a\in \varpi(a)[u_0]$ be the one-sided deformations introduced in Section \ref{Subsec_Moduli_RR action}. Suppose that $T[\ES[u_a]]$ is nonfattening for $|a|\gg 1$ (and hence is an $\I$-minimizing translator by Theorem \ref{Thm_Pf main thm_Transl region}).  Then there exist functions $a_\pm(z)$ defined on $z\in \RR_+$ such that when $z\to +\infty$,
			\begin{enumerate}[(i)]
				\item $\pm a_\pm(z)\to +\infty$;
				\item The ancient RMCF $\{\cR_{\tau-\log z}(T[\ES[u_{a_\pm(z)}]])\}_{\tau\leq 0}$ $C^\infty_{loc}$-converges to the one-sided flow $\{S^\kappa_\pm(\tau)\times \RR\}_{\tau\leq 0}$ above with Hausdorff distance $\kappa$ to $S$ at $\tau=0$, where recall \[
				\cR_\tau(T):= e^{\tau/2}\cdot (T - e^{-\tau}\partial_z).  \]
			\end{enumerate}
		\end{Lem}
		\begin{proof}
			WLOG, we work in $E_+$.  Define \[
			a_+(z):= \sup\left\{a\in \RR: T[\ES[u_a]]\cap \RR^n\times \{z\}\subset \sqrt{z}\cdot \BB_\kappa(S) \times \{z\}\right\}.  \]
			Note that by (\ref{Pf main thm_Equ_R^n+1 decomp into T_a}), for every $z\gg 1$, $a_+(z)<+\infty$, and $T_z^+:= T[\ES[u_{a_+(z)}]]$ satisfies
			\begin{align}
				\{x\in \RR^n: (\sqrt{z}x, z)\in T_z^+\} \subset \Clos(\BB_\kappa(S)), &\ & \{x\in \RR^n: (\sqrt{z}x, z)\in T_z^+\} \cap \partial\BB_\kappa(S) \neq \emptyset. \label{Fattening_z slice of T_z^+ intersects bdy of S kappa neighb}
			\end{align}
			Also by the uniform estimate in Lemma \ref{Lem_Moduli_one-sided transl end} and Theorem \ref{Thm_Existence of Transl w prescib end}, we know that $a_+(z)\to +\infty$ as $z\to \infty$. \\
			\textbf{Claim.} There exists a function $\tilde{\rho}(s)\to +\infty$ as $s\to +\infty$ such that for every $z\in [s, \tilde{\rho}(s)\cdot s]$, \[
			\{x\in \RR^n: (\sqrt{z}x, z)\in T_s^+\} \subset E_+ .   \]
			
			Note that once the Claim is proved, we immediately see that in $\RR^n\times \{0\}$ slice, the RMCF $\{\cR_{\tau-\log z}(T_z^+)\}_{\tau\leq 0}$ sits inside $E_+$ in a larger and larger time interval in $(-\infty, 0]$, and hence by Lemma \ref{Lem_Pre_blow down transl split} converges to a splitting one-sided RMCF of $S\times\RR$ and has Hausdorff distance $\kappa$ to $S$ at $\tau=0$ slice by (\ref{Fattening_z slice of T_z^+ intersects bdy of S kappa neighb}), which finishes the proof of (ii) of the Lemma.
			\begin{proof}[Proof of the Claim.]
				We fix $\gamma:= \mu^-/2$. By Lemma \ref{Lem_Moduli_Sharp asymp rate for pl-simple end}, there exists $\bar{z}(S, u_0)\gg 1$ such that \[
				\sup_{R\geq \bar{z}}R^{-\gamma}\cdot\|u_0\|_{C^{2, 1/2}_\star, S, R} <1 .  \]
				For every $s\geq s_0(\bar{z}, S)\gg \bar{z}$, Lemma \ref{Lem_Moduli_one-sided transl end} applies by taking $\delta(s):= (\log s)^{-1}$, $z_0(s):= s$ and $a(s):= \delta(s)s^{-\mu_1}$, which produces $v^{(s)}:= u_{a(s)}$ solving $\scT(v^{(s)})=0$ on $S\times \RR_{\geq s}$, with the following asymptotic estimate, for every $R\geq s$,
				\begin{align}
					\|v^{(s)} - u_0 - \delta(s)\cdot \left(\frac{z}{s}\right)^{\mu_1}\psi_1\|_{C^2_\star, S, R} \leq 2\cdot\delta(s)^{3/2}\left(\frac{R}{s}\right)^{\gamma/2 + \mu_1}.  
					\label{eq:v(z) compare with u_0}
				\end{align}
				In particular, we have for every $R\geq s$,
				\begin{align}
					\|v^{(s)}\|_{C^2_\star, S, R} \leq \|u_0\|_{C^2_\star, S, R} + C(S)(\log s)^{-1}\cdot \left(\frac{R}{s}\right)^{\mu_1} \leq \|u_0\|_{C^2_\star, S, R} + C(S)(\log R)^{-1}. \label{Fattening_v^(s) uniformly (in s) tend to 0}
				\end{align}
				We assert that the RHS above is independent of $s$ and tends to $0$ as $R\to+\infty$;
				And since $|u_0(\cdot, z)|\leq z^\gamma$, there exists some $\rho(s)$ tending to $+\infty$ as $s\to +\infty$ such that for every $s<z<\rho(s)s$,
				\begin{align}
					v^{(s)}(\cdot, z) \geq -\|u_0(\cdot, z)\|_{C^0, S} + C'(S)(\log s)^{-1}\cdot \left(\frac{z}{s}\right)^{\mu_1} \geq C''(S)(\log s)^{-1}\cdot \left(\frac{z}{s}\right)^{\mu_1}. \label{Fattening_v^(s) uniformly positive}    
				\end{align}
				By (\ref{Fattening_v^(s) uniformly (in s) tend to 0}) and Theorem \ref{Thm_Existence of Transl w prescib end}, when $s\gg 1$, there exists constant $\Lambda(S, u_0)\gg 1$ such that $T[\ES[v^{(s)}]]\cap \RR^n\times [\Lambda s, +\infty) = \ES[u^{(s)}]\cap \RR^n\times [\Lambda s, +\infty)$ and that
				\begin{align}
					|(u^{(s)}-v^{(s)})(\cdot, z)| \leq C(S, u_0) e^{-(z-\Lambda s)/2}. \label{Fattening_complete transl u^(s) exp close to end v^(s)}    
				\end{align}
				Combine this with (\ref{Fattening_v^(s) uniformly positive}) and take $s\gg 1$, \[
				u^{(s)} >0\ \ \ \ \text{ on }S\times [\Lambda s, \rho(s)s].   \]
				In other words, for every $z\in [\Lambda s, \rho(s)s]$, we have, \[
				\{x\in \RR^n: (\sqrt{z}x, z)\in T[\ES[v^{(s)}]]\} \subset E_+.  \]
				
				On the other hand, note that $T[\ES[v^{(s)}]] \in \varpi(a(s))[u_0]$, and by (\ref{Fattening_v^(s) uniformly (in s) tend to 0}) and (\ref{Fattening_complete transl u^(s) exp close to end v^(s)}), \[
				|u^{(s)}(\cdot, z)| \leq C(S, u_0)(\log z)^{-1}\ \ \ \ \text{ on }S\times \RR_{\geq \Lambda s},   \]
				which implies that for $s\gg 1$, $T[\ES[v^{(s)}]]\cap \RR^n\times \{\Lambda s\} \subset \sqrt{\Lambda s}\cdot \BB_{\kappa/2}(S) \times \{\Lambda s\}$, and hence by its definition, $a_+(\Lambda s)>a(s)$ and $T^+_{\Lambda s}$ lies outside $T[\ES[v^{(s)}]]$. In particular for every $z\in [\Lambda s, \rho(s)s]$, \[
				\{x\in \RR^n: (\sqrt{z}x, z)\in T^+_{\Lambda s}\} \subset E_+.  \]
				This finishes the proof of the claim.
			\end{proof}
			As explained after the statement of the claim, this completes the proof of the Lemma.
		\end{proof}
		
		\begin{Lem} \label{Lem_Fattening_Ellip Reg of mean convex domain}
			Let $n\geq 2$, $O\subset \RR^n$ be a bounded smooth strictly mean convex domain, in other words, the mean curvature vector $\vec{H}_{\partial O}$ points inward.  Then there exists a constant $\rho_O\gg 1$ such that for every $\rho\geq \rho_O$, if $T\subset \RR^{n+1}$ is a compact $\I$-minimizing hypersurface with boundary $\partial T \subset \rho\cdot O\times \{0\}$, then $T\subset \rho\cdot \Clos(O)\times \RR_{\leq 0}$.
		\end{Lem}
		\begin{proof}
			Suppose $\Lambda\gg 1$ such that $O\subset \BB^n_\Lambda$.  For every $\rho>1$, let $T_\rho\subset \RR^{n+1}$ be an $\I$-minimizing hypersurface with boundary $\partial T_\rho = \rho\cdot\partial O\times \{0\}$. Since the translator-mean-curvature vector of every hyperplane $\RR^n\times \{z\}$ points downward and the translator-mean-curvature vector of every paraboloid $\{(x, |x|^2-M): x\in \RR^n\}$ points inward, by strong maximum principle \cite{SolomonWhite89_Maxim}, we know that $T\setminus \partial T, T_\rho\setminus \partial T_\rho \subset \BB_{\rho\Lambda}^n\times [-\rho^2\Lambda^2, 0)$. Let $U_\rho\subset \RR^{n+1}$ be the open region enclosed by $T_\rho \cup (\rho\cdot O\times \{0\})$.  Then by comparing $\I$-functional, we have $T\subset \Clos(U_\rho)$.
			
			Since $O$ is strictly mean convex, $\exists\ \eta_0\in (0, 1)$ such that for every $\eta\in (0, \eta_0]$, $\BB_\eta(O)$ is also a smooth strictly mean convex domain.  Since the translator-mean-curvature vector of every $\rho\cdot \partial \BB_\eta(O)\times \RR$ points inward, again by strong maximum principle \cite{SolomonWhite89_Maxim}, we know that if 
			\begin{align}
				U_\rho\subset \rho\cdot\BB_{\eta_0}(O)\times\RR_{\leq 0}, \label{Fattening_I minzr in B_eta(O)}
			\end{align}
			then $T\subset \Clos(U_\rho)\subset\rho\cdot \Clos(O)\times \RR_{\leq 0}$, which finish the proof of the Lemma.
			
			To show that (\ref{Fattening_I minzr in B_eta(O)}) holds for all sufficiently large $\rho$, suppose for contradiction that there exist $\rho_j\nearrow +\infty$ and 
			\begin{align}
				(x_j, z_j)\in (\rho_j^{-1}T_{\rho_j})\setminus (\BB_{\eta_0}(O)\times \RR_{\leq 0}) \subset \left(\BB_\Lambda^n \setminus \BB_{\eta_0}(O)\right) \times [-\rho_j\Lambda^2, 0].  \label{Fattening_Points not in B_eta(O)*R}	
			\end{align}
			By Ilmanen's elliptic regularization \cite{Ilmanen94_EllipReg}, $\{T_j(t):= \rho_j^{-1}T_{\rho_j} + \rho_jt\cdot\partial_z\}_{t\geq 0}$ are mean curvature flows which converges in the Brakke sense in $\RR^n\times \RR_{\leq 0}$ to some Brakke motion $\{\mu_t \times \RR\}_{t\geq 0}$ starts at $O\times \RR$.  In particular, $(x_j, 0)\in \spt(T_j(-z_j/\rho_j))$ will subconverge in $\RR^n\times \RR$ to some point in the support of some $\mu_t\times \RR$, which by (\ref{Fattening_Points not in B_eta(O)*R}) lies in $(\BB_\Lambda^n\setminus \BB_{\eta_0}(O)) \times \{0\}$. This implies for some $t\geq 0$, \[
			\spt(\mu_t)\setminus \BB_{\eta_0}(O)\neq \emptyset.  \]
			On the other hand, since $O$ is mean convex, any Brakke flow starting from $O$ moves inward, which means $\spt(\mu_t)\subset \Clos(O)$ for every $t\geq 0$. This becomes a contradiction.
		\end{proof}
		
		Now we go back to Theorem \ref{Thm_Angenent torus fattens}.  Let us recall some basic facts of the Angenent torus $S_A$. The Angenent torus is rotationally symmetric in $\RR^3$, which is homeomorphic to $\SSp^1\times \SSp^1$. In the following, when we say something is rotationally symmetric in $\RR^3$, we mean it has the same rotational axis as the Angenent torus. From now on, we fix a $\kappa\in (0, \kappa_S)$.  Since the one-sided RMCF $\{S_{\pm}(\tau):= S^\kappa_\pm(\tau)\}$ is unique and asymptotic to $S_A$ as $\tau\to-\infty$, when $\tau$ is sufficiently negative, $S_{\pm}(\tau)$ are rotationally symmetric and isotopic to $S_A$.
		
		Next, we derive further topological properties of $S_\pm(\tau)$. The following two lemmas are essentially proved in \cite[Corollary 8.8]{CCMS20_GenericMCF}.
		
		\begin{Lem}\label{Lm: RMCF from A-torus s<0}
			There exists $\tau_->0$, and for any $\eta>0$ there exists $\epsilon_-(\eta)>0$ with the following significance:
			\begin{enumerate}
				\item for $\tau\in(-\infty,\tau_-)$, $S_-(\tau)$ is a regular rotational invariant torus, and $S_-(\tau_-)$ is a rotational invariant $\SSp^1$ curve;
				\item for $\tau\in[\tau_--\epsilon_-,\tau_-)$, $S_-(\tau)$ is a strictly mean convex surface sitting inside an $\eta$-neighbourhood of the curve $S_-(\tau_-)$;
				\item after $\tau_-$, $S_-(\tau)$ is empty.
			\end{enumerate}
		\end{Lem}
		
		\begin{proof}
			From \cite{CIMW13_EntropyMinmzer, CCMS20_GenericMCF},  $\{S_-(\tau)\}_{\tau \in\RR}$ is self-shrinker-mean-convex and it moves inwards. \cite{CIMW13_EntropyMinmzer} implies that $S_-$ has a finite time singularity. Then we can consider the profile curves $\gamma^-(\tau)$ of the flow $S_-(\tau)$, and $\gamma^-(\tau)$ is a circle in the upper half plane $\RR^2$ which moves inwards. By \cite{CIMW13_EntropyMinmzer, CCMS20_GenericMCF}, the tangent flow of a self-shrinker mean convex RMCF has to be either a cylinder or a sphere, which implies that the profile curve must shrink to a point in finite time $\tau_-$, and this is the only singular time. Because by our scaling, $S_-(\tau)$ does not vanish at $\tau=0$, so $\tau_->0$. This shows items (1) and (3).
			
			Item (2) follows from the fact that the tangent flows of $S_-(\tau)$ at the singular time are cylinders. Because $S_-(\tau)$ is rotationally symmetric, as $\tau\to \tau_-$, the profile curve becomes round, which implies that the profile curve is convex when $\tau$ is sufficiently close to $\tau_-$. This allows us to choose $\epsilon_-$ that satisfies item (2).
		\end{proof}

		\begin{Lem}\label{Lm: RMCF from A-torus s>0}
			There exists $\tau_+\in(0, +\infty)$ with the following significance: for $\tau\in[\tau_+,+\infty)$, $S_+(\tau)$ is a regular star shape sphere enclosing the origin, i.e. $\langle X, \nu_{S_+(\tau)}\rangle >0$.
		\end{Lem}
		Here we say a surface $\Sigma$ enclosing a point if the point is in the bounded connected component of $\RR^3\backslash \Sigma$.
		
		\begin{proof}
			In \cite[Theorem 9.1 (10)]{CCMS20_GenericMCF}, it was proved that the one-sided ancient MCF is regular and star-shaped in the time interval $(-\delta,\delta)$ for some $\delta>0$. Notice that star-shapedness is preserved by rescaling, and for the time $t$ slice of a MCF as $t\nearrow 0$ corresponds to the time $\tau$ slice of the corresponding RMCF as $\tau\nearrow +\infty$, hence rescaling to RMCF gives the desired property. 
		\end{proof}

		The next lemma shows how the topology of $\{S_\pm(\tau)\}_{\tau\in\RR}$ implies the topology of $T[\ES[u_{a_\pm(z)}]]$. From now on, we use $(x,z)=(x_1,x_2,y,z)$ to denote the coordinate of $\RR^3\times\RR=\RR^4$, where $y$ denotes the direction of the rotation axis of the Angenent torus. Because the Angenent torus is a rotationally symmetric torus, there must be a constant $\overline{\kappa}>0$, such that $\{(x_1,x_2,y)|\sqrt{x_1^2+x_2^2}<\overline{\kappa}\}$ is disjoint from the bounded region $E_-$ enclosed by Angenent torus.
		\begin{Lem}\label{Lem_topology change}
			Suppose the assumptions in Lemma \ref{Lem_Fattening_Rescaled lim of one-sided deform}. Then for $z$ sufficiently large,
			\begin{enumerate}
				\item $T[\ES[u_{a_-(z)}]]$ does not intersect the generalized cylinder $$\cC_{\overline{\kappa}}:=\{(x_1,x_2,y,z)|y,z\in\RR,\sqrt{x_1^2+x_2^2}\leq \overline{\kappa}\};$$			
				\item $T[\ES[u_{a_+(z)}]]$ must intersect the subspace $\{(0,0,y,z)|y,z\in\RR\}$.
			\end{enumerate}
		\end{Lem}
		
		\begin{proof}
			Lemma \ref{Lem_Fattening_Rescaled lim of one-sided deform} shows that the RMCF $\{\cR_{\tau-\log z}(T[\ES[u_{a_\pm(z)}]])\}_{\tau\leq 0}$ $C_{loc}^\infty$-converges to the one-sided flow $\{S_{\pm}(\tau)\times\RR\}_{\tau\leq 0}$. Moreover, by the uniqueness of this one-sided RMCF proved in \cite[Theorem 9.2]{CCMS20_GenericMCF}, the RMCF $\{\cR_{\tau-\log z}(T[\ES[u_{a_\pm(z)}]])\}_{\tau\in\RR}$ converges to the one-sided flow $\{S_{\pm}(\tau)\times\RR\}_{\tau\in\RR}$ in the Brakke sense, and by Brakke's regularity theorem, $\cR_{\tau-\log z}(T[\ES[u_{a_\pm(z)}]])$ would $C_{\text{loc}}^\infty$-converge to $S_{\pm}(\tau)\times\RR$ whenever $S_{\pm}(\tau)$ is regular. In the following proof we use $T^\pm_z$ to denote the translators $T[\ES[u_{a_\pm(z)}]]$.
			
			Let us first prove item (1). From Lemma \ref{Lm: RMCF from A-torus s<0}, fix an $\eta\in (0, 1)$ such that $\BB^3_\eta(S_-(\tau_-))$ is strictly mean convex and disjoint from $\{(x_1, x_2, y): x_1^2+x_2^2 \leq \bar{\kappa}^2\}$. Then we can choose $\epsilon_- = \epsilon_-(\eta)$ such that item (2) of Lemma \ref{Lm: RMCF from A-torus s<0} is satisfied. $S_-(\tau)$ is regular for all $\tau< \tau_-$, and $S_-(\tau)\subset E_-$ for all $\tau\in(-\infty,\tau_-]$. By the compactness of RMCF, for any $\delta\in (0, 1)$ and sufficiently large $z$ depending on $\delta$, \[
			\cR_{\tau-\log z}(T^-_z)\cap \RR^3\times \{0\} \subset E_-\times\{0\},  \] 
			for $\tau\in[-\delta^{-1},\tau_- - \delta]$. In particular, it does not intersect with a cylinder $\{(x_1,x_2,y, 0)|x_1^2+x_2^2\leq \overline{\kappa}^2\}$ that is disjoint from $E_-\times\{0\}$.  Moreover, if we rescale the flow back to get the translator $T^-_z$, $T^-_z\cap\RR^3\times [c_1(z), c_2(z)]$ does not intersect with $\cC_{\overline{\kappa}}$ provided $z\gg 1$, where $c_1(z):= ze^{-(\tau_--\delta)}$, $c_2(z):= ze^{1/\delta}$.
			
			It suffices to show that $T^-_z\cap\RR^3\times\RR_{<c_1(z)}$ and $T^-_z\cap\RR^3\times\RR_{>c_2(z)}$ also do not intersect $\cC_{\overline{\kappa}}$. For the first set, notice that by Lemma \ref{Lm: RMCF from A-torus s<0} item (2), when $\delta = \epsilon_-$ and $z\gg 1$, we have \[
			T^-_z\cap\RR^3\times \{c_1(z)\} \subset \sqrt{c_1(z)}\BB^3_\eta(S_-(\tau_-)) \times \{c_1(z)\}.  \]
			Because $T^-_z$ is an $\I$-minimizing translator, Lemma \ref{Lem_Fattening_Ellip Reg of mean convex domain} implies that when $z\gg 1$, \[
			T^-_z\cap \RR^3\times \RR_{<c_1(z)} \subset \sqrt{c_1(z)}\BB^3_\eta(S_-(\tau_-)) \times\RR.  \]
			In particular by the choice of $\eta$, $T^-_z\cap \RR^3\times \RR_{<c_1(z)}$ is disjoint from $\cC_{\overline{\kappa}}$.
			
			For the second set, first let $T_0$ be an $\I$-minimizing translator exponentially asymptotic to $\ES[u_0]$, and let $\Omega_0\subset \RR^{n+1}$ be the inner region bounded by $T_0$.  Then since $a_-(z)\to -\infty$ as $z\nearrow +\infty$, by Theorem \ref{Thm_Pf main thm_Transl region}, we have $T^-_z\subset \Omega_0$ when $z$ is sufficiently large.  Because $T_0 = \partial\Omega_0$ has a simple translating end over $S_A\times \RR_+$, we know that $\Omega_0\cap \RR^3\times \RR_{\geq \tilde{z}}\cap \cC_{\overline{\kappa}} = \emptyset$ for some sufficiently large $\tilde{z}>1$. Therefore, $T^-_z \cap \RR^3\times \RR_{\geq c_2(z)}$ doesn't intersect $\cC_{\overline{\kappa}}$ when $z\gg 1$.

			Next, we prove item (2). Because $\cR_{\tau-\log z}(T_z^+)$ $C_{loc}^\infty$-converges to the time slice of the one-sided flow $S_{+}(\tau)\times\RR$ whenever $S_{+}(\tau)\times\RR$ is smooth, $\cR_{\tau_+-\log z}(T_z^+) \cap\RR^3\times \{0\}$ smoothly converges to $S_+(\tau_+)\times\{0\}$ as $z\to +\infty$. By Lemma \ref{Lm: RMCF from A-torus s>0}, this implies that when $z$ is sufficiently large, $\cR_{\tau_+-\log z}(T_z^+)\cap\{(x,z'):z'=0\}$ is a smooth star-shaped surface enclosing the origin in $\{(x,z'):z'=0\}$, in particular it does intersect with the line $\{(0,0,y,0)|y\in\RR\}$. This implies that a slice of $T_z^+$ intersects with the subspace $\{(0,0,y,z)|y,z\in\RR\}$.
		\end{proof}
		
		Now we are ready to prove Theorem \ref{Thm_Angenent torus fattens}.
		
		\begin{proof}[Proof of Theorem \ref{Thm_Angenent torus fattens}]
			We prove this by contradiction. Let $\ES[u_0]$ be a rotationally symmetric simple translating end. Suppose no fattening happens, then for any $a\in\RR$, $T[\varpi(a)[u_0]]$ is the support of a complete embedded $\I$-minimizing translator, and we denote it by $T_a$. Then by the compactness of minimizing surface in any compact region, as well as the convergence of the ends, as $a_i\to a$, $T_{a_i}\to T_a$ smoothly. 
			
			Because by Theorems \ref{Thm_Moduli of transl end }, \ref{Thm_Pf main thm_Transl region} and Lemma \ref{Lem_Moduli_one-sided transl end}, all $T_a$ are rotationally symmetric, we can consider the rotation profile surface $\Theta_a$ of $T_a$, where $\Theta_a$ is a surface in the upper half space $\{(w,y,z)|w\geq 0\}\subset\RR^3$. The item (1) of Lemma \ref{Lem_topology change} shows that when $A>0$ is sufficiently large, $\Theta_A$ is disjoint from  $\{(w,y,z)|w\leq \kappa\}$. Therefore, any rotationally symmetric closed curve generated by a given point on $\Theta_A$ is non-contractible in $T_A$, since it's the generator of the fundamental group $\pi_1(\RR^4\setminus \cC_{\overline{\kappa}}) = \ZZ$.
			
			On the other hand, when $A>0$ is sufficiently large, the item (2) of Lemma \ref{Lem_topology change} shows that $T_{-A}$ intersects with $\{(w,y,z)|w=0\}$. In particular, for any rotationally symmetric closed curve $\sigma_p$ generated by a given point $p\in \Theta_{-A}$, the path on $\Theta_{-A}$ connecting $p$ to a point on $\{(w,y,z)|w=0\}$ generates a disc on $T_{-A}$ bounded by $\sigma_p$, which means $\sigma_p$ is contractible in $T_{-A}$. 
			
			From the item (2) of Theorem \ref{Thm_Moduli of transl end } and Theorem \ref{Thm_Existence of Transl w prescib end}, for each $a\in[-A,A]$, there exists $z_a>0$ and $\delta_a>0$ such that for $a'\in(a-\delta,a+\delta)\cap[-A,A]$, $T_{a'}\cap \RR^3\times\RR_{>z_a}$ is a graph over $T_a\cap\RR^3\times\RR_{>z_a}$, and they are all diffeomorphic to $S\times[0,+\infty)$, and $T_{a}\cap \RR^3\times\{z_a\}$ varies smoothly as $a$ changes. Then by the compactness of $[-A,A]$, there exists a uniform $\Bar{z}>0$, such that for every $a\in[-A,A]$, $T_a\cap\RR^3\times\RR_{\geq\bar{z}}$ is diffeomorphic to $S\times[0,+\infty)$, and when $a$ varies, $T_a\cap\RR^3\times\{\bar{z}\}$ varies smoothly as hypersurfaces in $\RR^3\times\{\bar{z}\}$. Therefore, $\{T_{a}\cap \RR^3\times\RR_{\leq \bar{z}}\}_{a\in [-A, A]}$ is an isotopy in $\RR^3\times\RR_{\leq \bar{z}}$ relative to the boundary in $\RR^3\times\{ \bar{z}\}$, and $\{\partial[T_{a}\cap \RR^3\times\RR_{\leq \bar{z}}]\}_{a\in [-A, A]}$ gives an isotopy of smooth tori in $\RR^3\times\{\bar{z}\}$. It is natural to identify a rotationally symmetric closed curve on the boundary via this boundary isotopy. However, such a closed curve is contractible in $T_{-A}\cap \RR^3\times\RR_{\leq \bar{z}}$ , but non-contractible in $T_{A}\cap \RR^3\times\RR_{\leq \bar{z}}$, which yields a contradiction.
		\end{proof}
		
		\begin{Rem}
			The main reason for the fattening in our construction of translators is the topology of the shrinkers. In contrast, it is not completely clear what the main reason for the fattening is in the flow setting, especially in higher dimensions. It is known that if certain conical singularities show up, then the flow fattens, see \cite{ChodoshDanielsHolgateSchulze24_MCFConical}, but it is not quite clear what the topological types of such conical singularities are. On the other hand, Hershkovits-White \cite{HershkovitsWhite20_nonfattening} proved that if all the blow-ups around a singularity are mean convex, then the flow does not fatten.
		\end{Rem}
		
		We conclude this section by proposing the following conjecture. 
		\begin{Con}
			If $T[\Sigma_0]$ fattens, then either there exists an unstable complete translator that is in the interior of $T[\Sigma_0]$, or $T[\Sigma_0]$ admits a foliation by complete $\I$-stable translators.
		\end{Con}
		
		If $T[\Sigma_0]$ fattens, then the two connected components of $\partial[T[\Sigma_0]]$ are $\I$-minimizing complete translators. In the minimal surface theory, once we have two area-minimizing hypersurfaces, it is natural to use the min-max method to produce the third one, and the third one is usually unstable. Such an idea was adapted by Bernstein-Wang \cite{BernsteinWang2020_Minmax_SelfExpander} to construct unstable self-expanders between two strictly stable self-expanders. Self-expanders form a class of mean curvature flow solitons, and we conjecture that a similar result holds for translators.
		
		\bigskip

		\appendix
		\section{Geometry of Auxiliary Ends} \label{Sec_Append_Aux end}
		Let $S\subset (\RR^n, g)$ be a closed self-shrinker, $\{x^1, \dots, x^{n-1}, y\}$ be local Fermi coordinates at $p\in S$ (hence $S = \{y = 0\}$), and let $g_{ij}$ be the induced Riemannian metrics on $S$, $[g^{ij}] := [g_{ij}]^{-1}$ be the inverse matrix, $\nu:= \partial_y$ be the unit normal field,  $A_{ij} = A_S(\partial_i, \partial_j)$ be the second fundamental form of $S\subset \RR^n$.
		
		Let $u\in C^2(S\times \RR_+)$ be such that for every $z>0$, \[
		\|u(\cdot, z)\|_{C^1, S} \leq \delta_S,    \]
		where $\delta_S\in (0, 1)$ be a small geometric constant.  Consider the parametrization \[
		\Phi_u: S\times \RR_+ \to \RR^{n+1},\ \ \ (x, z)\mapsto (\sqrt{z}(x+u(x,z)\nu_x), z).  \]
		We denote $\ES_u:= \Phi_u(S\times \RR_+)$ to be the image. Note that $\ES = \ES_0$.
		Let $\bar{g}^u:= \Phi_u^* g_{\Euc}$. Then under coordinates $\{x^1,, \dots, x^{n-1}, z\}$, 
		\begin{align}
			\begin{split}
				\bar{g}^u_{ij} & = z(g_{ij} -2uA_{ij} + u^2 A_{ik}A^k_j + u_iu_j) =: zg^u_{ij}, \\
				\bar{g}_{iz} & = \frac{1}{2}(X_i - uA_i^k X_k + u_i(X^\perp + u + 2zu_z)) =: \frac{X_i - \beta^u_i}{2}, \\
				\bar{g}_{zz} & =  1+ \frac{1}{4z}|x + (u+2zu_z)\nu|^2. 
			\end{split}  \label{Append_Transl end_g_ij}    
		\end{align}
		where $X^\perp := \langle x, \nu \rangle$, $X_i:= \langle x, \partial_i \rangle$; Also let $[g_u^{ij}]:= [g^u_{ij}]^{-1}$.  The volume density is given by 
		\begin{align}
			\bar{G}^u := \det[\bar{g}^u] = z^{n-1}\det[g^u]\cdot \left( 1+ \frac{1}{4z}\left(|x + (u+2zu_z)\nu|^2 - g_u^{ij}(X_i-\beta_i^u)(X_j - \beta_j^u) \right) \right). \label{Append_Transl end_G}
		\end{align}
		Define functions $F\in C^\infty(S_x\times \RR_p\times \RR^{n-1}_\xi)$, $E\in C^\infty(S_x\times \RR_z \times \RR_p \times \RR^{n-1}_\xi\times \RR_\eta)$ by
		\begin{align}
			\begin{split}
				F(x, p, \xi) & := \sqrt{\det[g_{ij}^{p, \xi}]\det[g_{ij}]^{-1}}; \\
				E(x, z, p, \xi, \eta) & := \left( 1+ \frac{1}{4z}\left(|x + (p+2\eta)\nu|^2 - g_{p, \xi}^{ij}(X_i-\beta_i^{p, \xi, \eta})(X_j - \beta_j^{p, \xi, \eta}) \right) \right)^{1/2}; \\
				\tilde{F}(x, z, p, \xi, \eta) & := F(x, p, \xi)\cdot E(x, z, p, \xi, \eta).
			\end{split} \label{Append_Transl end_F, E, tilde(F)}
		\end{align}
		where
		\begin{align*}
			g_{ij}^{p, \xi} & := g_{ij} - 2pA_{ij}+ p^2A_{ik}A^k_j + \xi_i\xi_j;& \     & [g^{ij}_{p, \xi}] := [g_{ij}^{p,\xi}]^{-1}; \\
			\beta_i^{p, \xi, \eta} & := pA_i^k X_k - \xi_i(X^\perp + p + 2\eta).
		\end{align*}
		Note that under this, \[
		\bar{G}^u = z^{n-1} \tilde{F}(x, z, u, \nabla u, zu_z)^2\cdot \det[g_{ij}].   \]
		Here $\nabla$ is derivative in $x$.  Moreover, for every $k\geq 0$,  $\nabla^k_x F$ and $\nabla_x^k E$  are power series in $p, \xi, 1/z, \eta$, whose convergence radius depends only on $k, S$, provided $S$ is compact.
		
		The Ilmanen's area functional for $\ES_u$ is,
		\begin{align*}
			\cI[u] & := \I[\ES_u] = \int_{S\times \RR_+} e^z \sqrt{\det[\bar{G}^u]} \\
			& = \int_{S\times \RR_+} e^z z^{(n-1)/2}\tilde{F}(x, z, u, \nabla u, zu_z)\ d\vol_{S\times \RR_+}.   
		\end{align*}
		Hence we define $\scT$ to be the E-L operator of $\cI$, i.e. given by 
		\begin{align}
			\begin{split}
				\scT (u)  = &\ -\text{div}_S \left(\tilde{F}_\xi(x, z, u, \nabla u, zu_z)\right) - z\partial_z \left(\tilde{F}_\eta(x, z, u, \nabla u, zu_z)\right)\\
				&\ - (z+\frac{n+1}{2})\tilde{F}_\eta(x, z, u, \nabla u, zu_z) + \tilde{F}_p(x, z, u, \nabla u, zu_z).
			\end{split}  \label{Append_EL oper of transl}
		\end{align}
		\begin{Lem} \label{Append_scT(0)}
			Let $S$ be a closed self-shrinker. Then for every $j, k\geq 0$, \[
			z^{k+1}|\nabla^j \partial_z^k \scT(0)| \leq C(S, j, k).  \]
		\end{Lem}
		\begin{proof}
			Notice that $H_S = -X^\perp/2$ and,
			\begin{align*}
				\tilde{F}(x, z, 0, \orig, 0) & = \sqrt{1+\frac{|X^\perp|^2}{4z}} =: E_0, &\  \tilde{F}_p(x, z, 0, \orig, 0) & = -H_S E_0 + \frac{X^\perp}{4zE_0},\\
				\tilde{F}_\xi(x, z, 0, \orig, 0) & = \frac{-1}{4zE_0} g^{ij}X_i X^\perp, &\  \tilde{F}_\eta(x, z, 0, \orig, 0) & = \frac{X^\perp}{2z E_0}.   
			\end{align*}
			Therefore, there exists a function $h(x, \hat{z})$ that is $C^\infty$ in $x\in S$ and real-analytic in $\hat{z}$ near $0$ such that $\scT(0)(x, z) = z^{-1}h(x, z^{-1})$, which directly implies the lemma.
		\end{proof}
		We also need the following error estimates for $\scT$. Recall that we are working with the following norms (see (\ref{Moduli_C^k,alpha_star norms}) and (\ref{Moduli_C^k,alpha_sharp norms})), for $\Omega\subset S\times \RR_{>1}$ and $R>1$,
		\begin{align*}
			[f]_{\alpha; \Omega}^\star & := \sup\left\{\frac{|f(x, z)- f(x', z')|}{|x- x'|^\alpha + R^{-\alpha/2}|z-z'|^\alpha}: (x, z), (x', z')\in S\times [R, 2R]\right\}, \\
			\|f\|_{C^\alpha_\star; S, R} & = \|f\|_{C^\alpha_\star; S\times[R, 2R]} := \sup_{S\times [R, 2R]} |f| + [f]_{\alpha; S\times [R, 2R]}^\star, \\
			\|u\|_{C^{2,\alpha}_\star, S, R} & := \Big(\|u\|_{C^\alpha_\star; S, R} + \|\nabla u\|_{C^\alpha_\star; S, R} + R\|\partial_z u\|_{C^\alpha_\star; S, R} \\
			&\ \ \ + \|\nabla^2 u\|_{C^\alpha_\star; S, R} + \sqrt{R}\|\partial_z \nabla u\|_{C^\alpha_\star; S, R} + R\|\partial_z^2 u\|_{C^\alpha_\star; S, R} \Big), \\
			\|u\|_{C^d_\sharp, S, R} & := \sum_{0\leq k+l\leq d}\sup_{S\times [R, R+1]} R^{-l/2}|\partial_z^k \nabla^l_S u|, \\
			\|u\|_{C^{d, \alpha}_\sharp, S, R} & := \|u\|_{C^d_\sharp, S, R}\ + \\ &\    \sup_{\substack{(x, z)\neq (x', z')\in S\times [R, R+1]\\ |x-x'|\leq R^{-1/2}}} \sum_{0\leq k+l \leq d} \frac{R^{-l/2}|\partial_z^k\nabla^l_S u(x, z) - \partial_z^k\nabla^l_S u(x', z')|}{R^{\alpha/2}|x-x'|^\alpha + |z-z'|^\alpha}.	
		\end{align*}

		\begin{Lem} \label{Append_Error est II of scT}
			Let $S$ be a closed self-shrinker.  Then there exists $\delta_S\in (0, 1)$ and $C_S>0$ such that if $u_\pm\in C^2(S\times \RR_{>1})$ satisfies \[
			\sup_{S\times \RR_{>1}} (|u_\pm| + |\nabla u_\pm| + |z\partial_zu_\pm| + |z\partial_z^2 u_\pm| + |\partial_z\nabla u_\pm| + |\nabla^2 u_\pm|) \leq \delta_S.    \]
			Then $v:= u_+ - u_-$ satisfies
			\begin{align*}
				\scT (u_+) - \scT (u_-) = &\ - \Big(z\partial^2_z v + z\partial_z v + \underbrace{(\Delta_S - \frac{x}{2}\cdot \nabla_S + |A_S|^2 + \frac{1}{2})v}_{L_S v} \Big)  \\
				&\ + \bar{\cE_1}\cdot z\partial_z^2v + (\bar{\cE_2} + x^S)\cdot \partial_z\nabla v + \bar{\cE_3}\cdot \nabla^2v + \bar{\cE_4}\cdot z\partial_z v + \bar{\cE_5}\cdot \nabla v + \bar{\cE_6}v.
			\end{align*}
			While for $1\leq l\leq 6$ we have pointwise estimates, \[
			|\bar{\cE_l}|\leq C_S(z^{-1}+ |u_\pm|+ |\nabla u_\pm|+ |z\partial_z u_\pm| + |z\partial_z^2 u_\pm| + |\partial_z\nabla u_\pm| + |\nabla^2 u_\pm|),   \]
			and H\"older estimates, 
			
			\begin{align*}
				\|\bar{\cE_l}\|_{C^\alpha_\star, S, R} & \leq  C(S, \alpha) \left(R^{-1} + \|u_+\|_{C^{2,\alpha}_\star, S, R} + \|u_-\|_{C^{2,\alpha}_\star, S, R} \right), \\
				\|\bar{\cE_l}\|_{C^\alpha_\sharp, S, R} & \leq C(S,\alpha)\left(R^{-1} + R\|u_+\|_{C^{2,\alpha}_\sharp, S, R} + R\|u_-\|_{C^{2,\alpha}_\sharp, S, R} \right).
			\end{align*}
			In particular, if we write $\scT(u)=: \scT(0) -z(\partial_z^2 + \partial_z + z^{-1}L_S)u + z\scR(u)$, then the error term $\scR$ satisfies
			\begin{align*}
				\|\scR(u_+)-\scR(u_-)\|_{C^\alpha_\star, S, R} & \leq C(S, \alpha)R^{-1}\left(R^{-1} + \|u_\pm\|_{C^{2,\alpha}_\star, S, R}\right)\|u_+-u_-\|_{C^{2,\alpha}_\star, S, R} ; \\
				\|\scR(u_+)-\scR(u_-)\|_{C^\alpha_\sharp, S, R} & \leq C(S, \alpha)\left(R^{-1} + \|u_\pm\|_{C^{2,\alpha}_\star, S, R}\right)\|u_+-u_-\|_{C^{2,\alpha}_\sharp, S, R}.	
			\end{align*}	 
		\end{Lem}
		\begin{proof}
			By analyticity of $\tilde{F}$ with respect to $1/z, p, \xi, \eta$, it suffices to compute the corresponding derivatives of $\tilde{F}$ at $(x, \infty, 0, \orig, 0)$. 
			\begin{align*}
				\tilde{F}_\xi(x, \infty, 0, \orig, 0) & = 0; & (z\tilde{F}_\eta)(x, \infty, 0, \orig, 0) &= \frac{X^\perp}{2}; & \tilde{F}_p(x, \infty, 0, \orig, 0) & = - H_S; \\
				\tilde{F}_{\xi\xi}(x, \infty, 0, \orig, 0) & = g^{ij}; & (z\tilde{F}_{\xi\eta})(x, \infty, 0, \orig, 0) & = -\frac{g^{ij}X_i}{2}; & \tilde{F}_{\xi p}(x, \infty, 0, \orig, 0) & = 0; \\
				(z\tilde{F}_{\eta\eta})(x, \infty, 0, \orig, 0) & = 1; & (z\tilde{F}_{\eta p})(x, \infty, 0, \orig, 0) & = H_S^2 + \frac{1}{2}; & \tilde{F}_{p p}(x, \infty, 0, \orig, 0) & = H_S^2 - |A_S|^2.
			\end{align*}
		\end{proof}

		We may also need the following improved error estimate for the divergence form the equation.
		\begin{Lem} \label{Append_Error est III of scT}
			Let $S$ be a closed self-shrinker.  Then there exists $\delta_S\in (0, 1)$ and $C_S>0$ such that if $u^\pm\in C^2(S\times \RR_{>1})$ satisfies \[
			\sup_{S\times \RR_{>z_0}} (|u^\pm| + |\nabla_S u^\pm| + |zu^\pm_z| + |\nabla^2_S u^\pm| + |\sqrt{z}\partial_z \nabla_S u^\pm| + |z\partial_z^2 u^\pm|) \leq \delta < \delta_S.    \]
			for some $z_0> 1$.  Then let $v := u^+ - u^-$, we have
			\begin{align*}
				\scT (u^+) - \scT (u^-) & = - \partial_z\left( (1+b)z\partial_z v - \frac{X}{2}\cdot \nabla_S v\right) - (1+b)z\partial_z v - (\Delta_S - \frac{X}{2}\cdot \nabla_S + |A_S|^2 + \frac{1}{2})v \\
				&\ - \text{div}_S E_1 - \partial_z E_2 - E_3,
			\end{align*}
			where we have pointwise estimates
			\begin{align*}
				|E_j(v)| & \leq C_S\cdot \big((z^{-1}+\delta)\cdot(|\nabla v| +|v|)+|\partial_z v| \big), \\
				|b| + |\nabla_S b| + |\sqrt{z}\partial_z b| & \leq C_S\cdot(z^{-1}+\delta).
			\end{align*}
		\end{Lem}
		\begin{proof}
			By (\ref{Append_EL oper of transl}), it suffices to estimate \[
			\tilde{F}_\sigma(x, z, u^+, \nabla u^+, zu_z^+) - \tilde{F}_\sigma(x, z, u^-, \nabla u^-, zu_z^-).   \]
			For a variable $\sigma\in \{\xi, \eta, p\}$. By the proof of Lemma \ref{Append_Error est II of scT}, 
			\begin{align*}
				&\ \tilde{F}_\xi(x, z, u^+, \nabla u^+, zu_z^+) - \tilde{F}_\xi(x, z, u^-, \nabla u^-, zu_z^-) \\
				= &\ (g^{ij} + \bar{o})\cdot (\partial_i u^+ - \partial_i u^-) \partial_j + \bar{o}\cdot(u^+ - u^-) + O(z^{-1})\cdot (zu_z^+ - zu_z^{-1}) \\
				=: &\ \nabla_S v + E_1(v),
			\end{align*}
			where $\bar{o}$ are error terms in Taylor expansion, with estimate \[
			\bar{o} = O(z^{-1}+ |u^\pm| + |\nabla_S u^\pm| + |zu^\pm_z| + |\nabla^2_S u^\pm| + |\sqrt{z}\partial_z \nabla_S u^\pm| + |z\partial_z^2 u^\pm| ) \leq C_S(z^{-1}+\delta).  \]
			Hence, it gives the desired estimate on $E_1$. Similarly, 
			\begin{align*}
				&\ \tilde{F}_p(x, z, u^+, \nabla u^+, zu_z^+) - \tilde{F}_p(x, z, u^-, \nabla u^-, zu_z^-) \\
				= &\ \bar{o}\cdot (\nabla_S u^+ - \nabla_S u^-) + (H_S^2 - |A_S|^2 + \bar{o})\cdot(u^+ - u^-) + O(z^{-1})\cdot (zu_z^+ - zu_z^-) \\
				=: &\ (H_S^2 - |A_S|^2)v + E_{3,1}(v),
			\end{align*}
			where $E_{3,1}$ satisfies the desired estimate. On the other hand,
			\begin{align*}
				&\ z\left(\tilde{F}_\eta(x, z, u^+, \nabla u^+, zu_z^+) - \tilde{F}_\eta(x, z, u^-, \nabla u^-, zu_z^-) \right) \\
				= &\ (-\frac{X}{2} + \bar{o})\cdot (\nabla_S u^+ - \nabla_S u^-) + (H_S^2 +\frac{1}{2} + \bar{o})\cdot(u^+ - u^-) \\
				+ &\ \left(\int_0^1 z\tilde{F}_{\eta\eta}(x, z, u^s, \nabla u^s, z\partial_z u^s)\ ds\right)\cdot (zu_z^+ - zu_z^-),
			\end{align*}
			where $u^s:= su^+ + (1-s)u^-$. Since by its expression (\ref{Append_Transl end_F, E, tilde(F)}), we have 
			\begin{align*}
				\tilde{F}_{\eta\eta}(x, z, p, \xi, \eta)  = \frac{F}{2E}\cdot\left((E^2)_{\eta\eta} - 2E_\eta^2 \right) = \frac{F}{zE}\cdot (1-g_{p,\xi}^{ij} \xi_i\xi_j) + O(z^{-2}) .
			\end{align*}
			Therefore, we have
			\begin{align*}
				&\ z\left(\tilde{F}_\eta(x, z, u^+, \nabla u^+, zu_z^+) - \tilde{F}_\eta(x, z, u^-, \nabla u^-, zu_z^-) \right) \\
				= &\ -\frac{X}{2}\cdot \nabla_S v + (H_S^2 + \frac{1}{2})v + \underbrace{\left(\int_0^1 \frac{F}{E}\cdot(1-g_{p,\xi}^{ij} \xi_i\xi_j)|_{(x, z, u^s, \nabla u^s, z\partial_z u^s)}\ ds\right)}_{=:\ 1+b}\cdot zv_z + E_2(v),  
			\end{align*}
			where $b$ and $E_2$ satisfies the desired estimate. Combining (\ref{Append_EL oper of transl}) proves the Lemma.
		\end{proof}

		\section{Rescaling of Graphs} \label{Sec_Append_Scal Graph}
		Let $S\subset \RR^n$ be a $C^3$ closed hypersurface with normal field $\nu_S$.  The goal for this section is to prove the following:
		\begin{Lem} \label{Lem_App_Scal graph}
			There exists a geometric constant $\vartheta_S\in (0, 1)$ such that for every $u\in C^1(S\times [0, 1])$ with $\|u\|_{C^1, S\times [0, 1]}\leq \vartheta_S$, there exists a unique $\check{u}\in C^1(S\times [0, 1]\times (1-\vartheta_S, 1+\vartheta_S))$ such that for every $z\in [0, 1]$ and $|a-1|<\vartheta$, \[ 
			\graph_S(\check{u}(\cdot, z, a)) = a\cdot \graph_S(u(\cdot, z)).  \]
			Moreover, for every such $a$, $\check{u}$ satisfies the estimate, 
			\begin{align}
				\|\check{u}(\cdot, \cdot, a)\|_{C^1, S\times [0, 1]} & \leq C(S)(\|u\|_{C^1, S\times [0, 1]} + |a-1|) , \label{App_Scal graph_|check(u)|_C^1} \\
				\|\partial_a \check{u}(\cdot, \cdot, a)\|_{C^0, S\times [0, 1]} & \leq C(S) . \label{App_Scal graph_|partial_a check(u)|}
			\end{align}
		\end{Lem}
		\begin{proof}
			Let $\delta=\delta_S\in (0, 1)$ such that $\mbfG: S\times (-\delta_S, \delta_S) \to \RR^n$, $(x, y)\mapsto x+ y \nu_S(x)$ is a $C^2$ diffeomorphism onto its image; $\Pi: \BB_\delta(S)\to S$ be the $C^2$ nearest point projection onto $S$. 
			By implicit function theorem, there exists $0<\vartheta(S)\ll1$ such that 
			\begin{itemize}
				\item for every $|a-1|\leq \vartheta$, $a \cdot\BB_{\vartheta_S}(S)\subset \BB_{\delta/2}(S)$; 
				\item for every $|a-1|<\vartheta$ and every $u\in C^1(S\times [0, 1])$ with $\|u\|_{C^1, S\times [0, 1]}$, the following  
				\begin{align*}
					\mbfP_{a, u} : S\times [0, 1] \to S\times [0, 1], \ \ \ (x, z) \mapsto (\Pi(a\cdot \mbfG(x, z)), z) , 
				\end{align*}
				is a $C^1$ diffeomorphism and is $C^2$ in $a$. 
			\end{itemize} 
			Moreover, by choosing $\vartheta\ll1$,
			\begin{align*}
				\|\mbfP_{a, u} - id_{S\times [0,1]}\|_{C^1, S\times [0, 1]} + \|\mbfP_{a, u}^{-1} - id_{S\times [0,1]}\|_{C^1, S\times [0, 1]} & \leq C(S)|a-1|; \\
				\|\partial_a \mbfP_{a, u}\|_{C^0, S\times [0, 1]} + \|\partial_a \mbfP_{a, u}^{-1}\|_{C^0, S\times [0, 1]} & \leq C(S).
			\end{align*}
			Let $\mbfQ_{a, u}\in C^1(S\times [0, 1], S)$ be the $S$-component of $\mbfP^{-1}_{a, u}$. Under this notation, $\check{u}\in C^1$ is uniquely defined by the explicit expression, \[
			\check{u}(x, z, a):= \rho_S\left(a\cdot \mbfG\left( \mbfQ_{a, u}(x, z), u\circ \mbfP_{a,u}^{-1}(x, z)\right)\right),  \]
			where $\rho_S:= dist_{\RR^n}(S, \cdot)$.  Since every map is $C^1$, this immediately proves (\ref{App_Scal graph_|partial_a check(u)|}). To see (\ref{App_Scal graph_|check(u)|_C^1}), note that,
			\begin{align*}
				|\check{u}(\cdot, \cdot, a)|& \leq |\check{u}(\cdot, \cdot, 1)| + C(S)|a-1| \leq C(S)(\|u\|_{C^0, S\times [0, 1]}+|a-1|) ; \\
				\|D\check{u}(\cdot, \cdot, a)\| & = \left\|D\rho_S\big|_{a\mbfG} \cdot a\left(\partial_x \mbfG \cdot D \mbfQ_{a, u} + \partial_y \mbfG\cdot Du \cdot D \mbfP_{a, u}^{-1}\right)\right\| \\
				& = a\left\|\left( D\rho_S\big|_{a\mbfG} - D\rho_S\big|_\mbfG \right) \cdot \partial_x \mbfG \cdot D \mbfQ_{a, u} + D\rho_S\big|_{a\mbfG}\cdot \partial_y \mbfG\cdot Du\cdot D \mbfP_{a, u}^{-1} \right\| \\
				& \leq C(S)(|a-1| + \|u\|_{C^1, S\times [0, 1]}).
			\end{align*}
		\end{proof}
		
		\bigskip
		
		\noindent\textbf{Data availability statement.} This work has no associated data.
		
		\medskip
		
		\noindent\textbf{Conflicts of interest.} The authors have no conflicts of interest to declare that are relevant to the content of this article.

		\bibliographystyle{alpha}
		\bibliography{GMT}

\begin{thebibliography}{HMW22b}

\bibitem[AAG95]{AAG95_RotationMCF}
Steven Altschuler, Sigurd~B. Angenent, and Yoshikazu Giga.
\newblock Mean curvature flow through singularities for surfaces of rotation.
\newblock {\em J. Geom. Anal.}, 5(3):293--358, 1995.

\bibitem[Ang92]{Angenent92_Doughnut}
Sigurd~B Angenent.
\newblock Shrinking doughnuts.
\newblock In {\em Nonlinear diffusion equations and their equilibrium states,
  3}, pages 21--38. Springer, 1992.

\bibitem[AS67]{AronsonSerrin67_QuasilinearParabEqu}
Don~G Aronson and James Serrin.
\newblock Local behavior of solutions of quasilinear parabolic equations.
\newblock {\em Archive for Rational Mechanics and Analysis}, 25(2):81--122,
  1967.

\bibitem[AV97]{AV97_DegenerateNeckpinches}
S.~B. Angenent and J.~J.~L. Vel\'{a}zquez.
\newblock Degenerate neckpinches in mean curvature flow.
\newblock {\em J. Reine Angew. Math.}, 482:15--66, 1997.

\bibitem[AW94]{AltschulerWu94_Transl}
Steven~J. Altschuler and Lang~F. Wu.
\newblock Translating surfaces of the non-parametric mean curvature flow with
  prescribed contact angle.
\newblock {\em Calc. Var. Partial Differential Equations}, 2(1):101--111, 1994.

\bibitem[BLT20]{BourniLangfordTinaglia20_Transl}
Theodora Bourni, Mat Langford, and Giuseppe Tinaglia.
\newblock On the existence of translating solutions of mean curvature flow in
  slab regions.
\newblock {\em Analysis \& pde}, 13(4):1051--1072, 2020.

\bibitem[BNS25]{BuzanoNguyenSchulz21_ShrinkerGenus}
Reto Buzano, Huy~The Nguyen, and Mario~B. Schulz.
\newblock Noncompact self-shrinkers for mean curvature flow with arbitrary
  genus.
\newblock {\em J. Reine Angew. Math.}, 818:35--52, 2025.

\bibitem[Bra78]{Brakke78}
Kenneth~A. Brakke.
\newblock {\em The Motion of a Surface by Its Mean Curvature. (MN-20)}.
\newblock Princeton University Press, Princeton, 1978.

\bibitem[BS23]{baldauf-Sun2018_sharp}
Julius Baldauf and Ao~Sun.
\newblock Sharp entropy bounds for plane curves and dynamics of the curve
  shortening flow.
\newblock {\em Comm. Anal. Geom.}, 31(3):595--624, 2023.

\bibitem[BW22]{BernsteinWang2020_Minmax_SelfExpander}
Jacob Bernstein and Lu~Wang.
\newblock A mountain-pass theorem for asymptotically conical self-expanders.
\newblock {\em Peking Math. J.}, 5(2):213--278, 2022.

\bibitem[CCMS24]{CCMS20_GenericMCF}
Otis Chodosh, Kyeongsu Choi, Christos Mantoulidis, and Felix Schulze.
\newblock Mean curvature flow with generic initial data.
\newblock {\em Invent. Math.}, 237(1):121--220, 2024.

\bibitem[CDHS24]{ChodoshDanielsHolgateSchulze24_MCFConical}
Otis Chodosh, J.~M. Daniels-Holgate, and Felix Schulze.
\newblock Mean curvature flow from conical singularities.
\newblock {\em Invent. Math.}, 238(3):1041--1066, 2024.

\bibitem[Cha97]{Chan97}
Claire~C. Chan.
\newblock Complete minimal hypersurfaces with prescribed asymptotics at
  infinity.
\newblock {\em J. Reine Angew. Math.}, 483:163--181, 1997.

\bibitem[CHHW22]{ChoiHaslhoferHershkovitsWhite22_AncientMCF}
Kyeongsu Choi, Robert Haslhofer, Or~Hershkovits, and Brian White.
\newblock Ancient asymptotically cylindrical flows and applications.
\newblock {\em Invent. Math.}, 229(1):139--241, 2022.

\bibitem[CIMW13]{CIMW13_EntropyMinmzer}
Tobias~Holck Colding, Tom Ilmanen, William~P. Minicozzi, II, and Brian White.
\newblock The round sphere minimizes entropy among closed self-shrinkers.
\newblock {\em J. Differential Geom.}, 95(1):53--69, 2013.

\bibitem[CM12]{ColdingMinicozzi12_generic}
Tobias~H. Colding and William~P. Minicozzi, II.
\newblock Generic mean curvature flow {I}: generic singularities.
\newblock {\em Ann. of Math. (2)}, 175(2):755--833, 2012.

\bibitem[CM15]{CM15_Lojasiewicz}
Tobias~Holck Colding and William~P. Minicozzi, II.
\newblock Uniqueness of blowups and {\l}ojasiewicz inequalities.
\newblock {\em Ann. of Math. (2)}, 182(1):221--285, 2015.

\bibitem[CSS07]{ClutterbuckSchnurerSchulze07_TranslCatenoid}
Julie Clutterbuck, Oliver~C Schn{\"u}rer, and Felix Schulze.
\newblock Stability of translating solutions to mean curvature flow.
\newblock {\em Calculus of Variations and Partial Differential Equations},
  29(3):281--293, 2007.

\bibitem[DdPN17]{DavilaDelPinoNguyen17_FiniteTopTransl}
Juan D\'avila, Manuel del Pino, and Xuan~Hien Nguyen.
\newblock Finite topology self-translating surfaces for the mean curvature flow
  in {$\Bbb{R}^3$}.
\newblock {\em Adv. Math.}, 320:674--729, 2017.

\bibitem[DLN18]{DruganLeeNguyen18_Survey_SymShrinker}
Gregory Drugan, Hojoo Lee, and Xuan~Hien Nguyen.
\newblock A survey of closed self-shrinkers with symmetry.
\newblock {\em Results in Mathematics}, 73(1):1--32, 2018.

\bibitem[GT01]{GilbargTrudinger01}
David Gilbarg and Neil~S. Trudinger.
\newblock {\em {{E}}lliptic {{P}}artial {{D}}ifferential {{E}}quations of
  {{S}}econd {{O}}rder}.
\newblock Springer-Verlag, New York, 2001.
\newblock reprint of the 1998 edition.

\bibitem[Ham88]{Hamilton88_Cigar}
Richard~S. Hamilton.
\newblock The {R}icci flow on surfaces.
\newblock In {\em Mathematics and general relativity ({S}anta {C}ruz, {CA},
  1986)}, volume~71 of {\em Contemp. Math.}, pages 237--262. Amer. Math. Soc.,
  Providence, RI, 1988.

\bibitem[Ham95]{Hamilton95_HarnackMCF}
Richard~S. Hamilton.
\newblock Harnack estimate for the mean curvature flow.
\newblock {\em J. Differential Geom.}, 41(1):215--226, 1995.

\bibitem[Her20]{Hershkovits20_Translators}
Or~Hershkovits.
\newblock Translators asymptotic to cylinders.
\newblock {\em J. Reine Angew. Math.}, 766:61--71, 2020.

\bibitem[HIMW19]{HIMW19_TranslGraph}
D.~Hoffman, T.~Ilmanen, F.~Mart\'in, and B.~White.
\newblock Graphical translators for mean curvature flow.
\newblock {\em Calc. Var. Partial Differential Equations}, 58(4):Paper No. 117,
  29, 2019.

\bibitem[HMW22a]{HMW19_ScherkTransl}
D.~Hoffman, F.~Mart\'in, and B.~White.
\newblock Scherk-like translators for mean curvature flow.
\newblock {\em J. Differential Geom.}, 122(3):421--465, 2022.

\bibitem[HMW22b]{HMW22_SemiGraphTransl}
David Hoffman, Francisco Mart\'in, and Brian White.
\newblock Nguyen's tridents and the classification of semigraphical translators
  for mean curvature flow.
\newblock {\em J. Reine Angew. Math.}, 786:79--105, 2022.

\bibitem[HS85]{HardtSimon85}
Robert Hardt and Leon Simon.
\newblock Area minimizing hypersurfaces with isolated singularities.
\newblock {\em J. Reine Angew. Math.}, 362:102--129, 1985.

\bibitem[Hui90]{Huisken90}
Gerhard Huisken.
\newblock Asymptotic behavior for singularities of the mean curvature flow.
\newblock {\em J. Differential Geom.}, 31(1):285--299, 1990.

\bibitem[HW20]{HershkovitsWhite20_nonfattening}
Or~Hershkovits and Brian White.
\newblock Nonfattening of mean curvature flow at singularities of mean convex
  type.
\newblock {\em Comm. Pure Appl. Math.}, 73(3):558--580, 2020.

\bibitem[Ilm94]{Ilmanen94_EllipReg}
Tom Ilmanen.
\newblock Elliptic regularization and partial regularity for motion by mean
  curvature.
\newblock {\em Mem. Amer. Math. Soc.}, 108(520):x+90, 1994.

\bibitem[Ilm96]{Ilmanen96}
T.~Ilmanen.
\newblock A strong maximum principle for singular minimal hypersurfaces.
\newblock {\em Calc. Var. Partial Differential Equations}, 4(5):443--467, 1996.

\bibitem[IW25]{IlmanenWhite25_Fattening}
Tom Ilmanen and Brian White.
\newblock Fattening in mean curvature flow.
\newblock {\em Ars Inven. Anal.}, pages Paper No. 4, 32, 2025.

\bibitem[Ket16]{Ketover16_Shrinker}
Daniel Ketover.
\newblock Self-shrinking platonic solids.
\newblock {\em arXiv preprint arXiv:1602.07271}, 2016.

\bibitem[Ket24]{ketover2024self}
Daniel Ketover.
\newblock Self-shrinkers whose asymptotic cones fatten.
\newblock {\em arXiv preprint arXiv:2407.01240}, 2024.

\bibitem[KKMl18]{KapouleasKleeneMoller18_DesingShrinker}
Nikolaos Kapouleas, Stephen~James Kleene, and Niels~Martin M\o~ller.
\newblock Mean curvature self-shrinkers of high genus: non-compact examples.
\newblock {\em J. Reine Angew. Math.}, 739:1--39, 2018.

\bibitem[KM23]{KapouleasMcGrath20_DoublingShrinker}
Nikolaos Kapouleas and Peter McGrath.
\newblock Generalizing the linearized doubling approach, {I}: {G}eneral theory
  and new minimal surfaces and self-shrinkers.
\newblock {\em Camb. J. Math.}, 11(2):299--439, 2023.

\bibitem[Kne80]{Knerr80_SpacialParabSchauder}
Barry~F Knerr.
\newblock Parabolic interior schauder estimates by the maximum principle.
\newblock {\em Archive for Rational Mechanics and Analysis}, 75(1):51--58,
  1980.

\bibitem[LZ24]{lee2024closed}
Tang-Kai Lee and Xinrui Zhao.
\newblock Closed mean curvature flows with asymptotically conical
  singularities.
\newblock {\em arXiv preprint arXiv:2405.15577}, 2024.

\bibitem[M{\o}l11]{Moller11_ClosedShrinker}
Niels~Martin M{\o}ller.
\newblock Closed self-shrinking surfaces in $\mathbb{R}^3$ via the torus.
\newblock {\em arXiv preprint arXiv:1111.7318}, 2011.

\bibitem[Mul56]{Mullins56_MCF}
William~W Mullins.
\newblock Two-dimensional motion of idealized grain boundaries.
\newblock {\em Journal of Applied Physics}, 27(8):900--904, 1956.

\bibitem[Ngu09]{Nguyen09_TranslTrident}
Xuan~Hien Nguyen.
\newblock Translating tridents.
\newblock {\em Comm. Partial Differential Equations}, 34(1-3):257--280, 2009.

\bibitem[Ngu13]{Nguyen13_Transl}
Xuan~Hien Nguyen.
\newblock Complete embedded self-translating surfaces under mean curvature
  flow.
\newblock {\em J. Geom. Anal.}, 23(3):1379--1426, 2013.

\bibitem[Ngu14]{Nguyen14_Shrinker}
Xuan~Hien Nguyen.
\newblock Construction of complete embedded self-similar surfaces under mean
  curvature flow, {P}art {III}.
\newblock {\em Duke Math. J.}, 163(11):2023--2056, 2014.

\bibitem[Ngu15]{Nguyen15_DoublyPeriodTransl}
Xuan~Hien Nguyen.
\newblock Doubly periodic self-translating surfaces for the mean curvature
  flow.
\newblock {\em Geometriae Dedicata}, 174(1):177--185, 2015.

\bibitem[Rie23]{Riedler22_ClosedShrinker}
Oskar Riedler.
\newblock Closed embedded self-shrinkers of mean curvature flow.
\newblock {\em J. Geom. Anal.}, 33(6):Paper No. 172, 27, 2023.

\bibitem[Sim83]{Simon83_GMT}
Leon Simon.
\newblock {\em Lectures on Geometric Measure Theory}, volume~3 of {\em Proc.
  Centre for Mathematical Analysis, Australian National University}.
\newblock Australian National University, Centre for Mathematical Analysis,
  Canberra, 1983.

\bibitem[Sim89]{Simon89}
Leon Simon.
\newblock Entire solutions of the minimal surface equation.
\newblock {\em J. Differential Geom.}, 30(3):643--688, 1989.

\bibitem[Smi15]{Smith15_TranslPrescibGenus}
Graham Smith.
\newblock On complete embedded translating solitons of the mean curvature flow
  that are of finite genus.
\newblock {\em arXiv preprint arXiv:1501.04149}, 2015.

\bibitem[SW89]{SolomonWhite89_Maxim}
Bruce Solomon and Brian White.
\newblock A strong maximum principle for varifolds that are stationary with
  respect to even parametric elliptic functionals.
\newblock {\em Indiana University Mathematics Journal}, 38(3):683--691, 1989.

\bibitem[SWZ24]{SunWangZhou20_MinmaxShrinker}
Ao~Sun, Zhichao Wang, and Xin Zhou.
\newblock Multiplicity one for min-max theory in compact manifolds with
  boundary and its applications.
\newblock {\em Calc. Var. Partial Differential Equations}, 63(3):Paper No. 70,
  52, 2024.

\bibitem[SX21]{SunXue2021_initial_closed}
Ao~Sun and Jinxin Xue.
\newblock Initial perturbation of the mean curvature flow for closed limit
  shrinker.
\newblock {\em arXiv preprint arXiv:2104.03101}, 2021.

\bibitem[SX25]{SunXue2021_initial_conical}
Ao~Sun and Jinxin Xue.
\newblock Generic dynamics of mean curvature flows with asymptotically conical
  singularities.
\newblock {\em Sci. China Math.}, 2025.

\bibitem[Wan11]{WangXJ11_ConvexMCF}
Xu-Jia Wang.
\newblock Convex solutions to the mean curvature flow.
\newblock {\em Ann. of Math. (2)}, 173(3):1185--1239, 2011.

\bibitem[Whi05]{White05_MCFReg}
Brian White.
\newblock A local regularity theorem for mean curvature flow.
\newblock {\em Ann. of Math. (2)}, 161(3):1487--1519, 2005.

\bibitem[Zhu25]{Zhu20_Lojasiewicz}
Jonathan~J. Zhu.
\newblock \l ojasiewicz inequalities, uniqueness and rigidity for cylindrical
  self-shrinkers.
\newblock {\em Camb. J. Math.}, 13(1):173--224, 2025.

\end{thebibliography}
		
	\end{document}